\documentclass{article}

\usepackage{arxiv}

\usepackage[utf8x]{inputenc} % allow utf-8 input
\usepackage[T1]{fontenc}    % use 8-bit T1 fonts
\usepackage{hyperref}       % hyperlinks
\usepackage{url}            % simple URL typesetting
\usepackage{booktabs}       % professional-quality tables
\usepackage{amsfonts}       % blackboard math symbols
\usepackage{nicefrac}       % compact symbols for 1/2, etc.
\usepackage{microtype}      % microtypography
\usepackage{lipsum}
\usepackage{color}
\usepackage[usenames,dvipsnames,svgnames,table]{xcolor}
\usepackage{graphicx}
\usepackage{mwe}
\usepackage{graphbox}
\usepackage{amssymb}
\usepackage{extarrows}
\usepackage{tikz}
\usetikzlibrary{matrix}
\usepackage{amsthm}
\usepackage{multirow}
\usepackage{tikz-cd}
\usepackage{overpic}

\usepackage{array}

\begin{document}
%\maketitle

\begin{center}
\vspace{0.6cm}

{ \bf \LARGE    Gutierrez-Sotomayor Flows on Singular Surfaces
}

\vspace{0.6cm}
\large 
M.A.J. Zigart \footnote{Partially  financed  by the Coordenação de Aperfeiçoamento de Pessoal de Ní­vel Superior - Brasil (CAPES) - Finance Code 001}
\hspace{0.3cm}
K.A. de Rezende\footnote{Partially supported by CNPq under grant 305649/2018-3 and FAPESP under grant 2018/13481-0.}
\hspace{0.3cm}
N.G. Grulha Jr. \footnote{Supported by CNPq under grant 303046/2016-3; and  FAPESP under grant 2017/09620-2.}
\hspace{0.3cm}
D.V.S. Lima%\footnote{}  
\hspace{0.3cm}

\end{center}

\vspace{1cm}

\begin{abstract}

In this work  we address the realizability of a Lyapunov graph labeled with GS singularities, namely regular, cone, Whitney, double crossing and triple crossing singularities,  as continuous flow on a singular closed  $2$-manifold $\mathbf{M}$.  Furthermore, the Euler characteristic is computed with respect to the types of GS singularities of the flow on $\mathbf{M}$.
Locally, a complete classification theorem for  minimal isolating blocks of GS singularities is presented in terms of the branched one manifolds that make up the boundary.

\end{abstract}

\vspace{0.5cm}


%14B05; 14J17; 58K30.

\vspace{0.25cm}

\section*{Introduction}

 The study of singular varieties plays an important role  in the interface of Algebraic Geometry, Singularity Theory and F-Theory. 
 
 An aspect that was greatly explored by MacPherson
 was the 
 existence and uniqueness of Chern classes for singular algebraic varieties \cite{macpherson1974chern}.  Later Brasselet and Schwartz  \cite{brasselet1981classes,brasselet2009vector, schwartz1965classes} applied the 
  study of vector fields on singular varieties to provide a new framework in which  to consider the study of characteristic classes. Around the same period, Gutierrez and Sotomayor used  Singularity Theory  of differentiable maps to investigate $C^{1}$-structurally stable vector fields  on manifolds with cone, cross-cap, double crossing and triple crossing singularities, and proved their genericity in the set of $C^{1}$-vector fields. 
 In \cite{hernan2019}, these vector fields were studied by considering their associated continuous flows, making it possible to use algebraic topological tools, namely, the Conley Index Theory. These flows were named therein as  Gutierrez-Sotomayor flows, for short, GS flows.
 This work opened up many questions  in Topological Dynamics, more specifically, on  the global qualitative study of continuous flows on  singular varieties.  In \cite{LRdR}, spectral sequences were used as a tool to understand how a flow  undergoes homotopical cancellations on a $2$-dimensional singular variety.

In \cite{hernan2019}, the existence of a Lyapunov function $f$ associated to a GS flow was established, showing the existence of isolating blocks for the singularities of a GS flow and allowing one to define a Lyapunov graph $L_{f}$. Furthermore, the Poincaré-Hopf condition was proved to be a necessary  local condition on the flow defined on an isolating block, as well as, a necessary global condition on the flow defined on a closed singular manifold.

The results in \cite{hernan2019} opened up many questions in Topological Dynamics, more specifically: by imposing the Poincaré-Hopf conditions on an abstract Lyapunov graph $L$, are they sufficient to ensure the local realizability of $L$ as a GS flow on an isolating block? Does the local realizabilty of $L$ imply the global realizability of $L$ as a GS flow on a closed singular manifold?

In this work, we advance the local and global study of GS flows by addressing these open questions in \cite{hernan2019}. On what concerns the local realizability of $L$, that is, the realizability of a semi-graph $L_{v}$ consisting of a single vertex $v$ in $L$ and its incident edges as a GS flow on an isolating block, it turns out that the Poincaré-Hopf condition is not sufficient. On the other hand, even if locally the vertices and their incidents edges are realizable as isolating blocks of GS singularities, it may be the case that there is no global realization corresponding to $L$. The reason for this is the fact that locally the realization of $L_{v}$ as a GS flow on an isolating block $N$ is not unique due to the many choices of distinguished branched $1$-manifolds that can make up the boundary of $N$. This multiplicity imposes the main difficulty in the global realization question which relies heavily on an assignment that correctly matches up pairs of boundary components of isolating blocks, allowing the gluing of the blocks according to $L$ and making it possible to obtain a closed singular variety.

 In Section 1,  background material is presented.  In Section 2, the local realizability question is addressed. More specifically, the realizability of  an abstract Lyapunov semi-graph $L_{v}$ consisting of a single vertex and its incident weighted  edges as a GS flow on an isolating block $N$.  Furthermore, a complete characterization of the branched $1$-manifolds that make up the boundary of a minimal isolating block is provided, as well as, the construction of passageways for tubular flows that have the effect of increasing the number of branch points on the boundary of $N$ in order to match the weights on the edges of $L_{v}$.
 In Section 3, several theorems on the global realizability of an abstract Lyapunov graph $L$ labelled with GS singularities as a continuous  flows on a singular manifold $M$ are presented. In Section 4, the Euler characteristic of these manifolds  is computed in terms of the types of  GS singularities of the flow.

\section{Background}

\subsection{Gutierrez-Sotomayor Vector Fields}

In \cite{gutierrez1982stable}, Gutierrez and Sotomayor introduced the  \textit{$2$-manifolds with simple singularities}\footnote{In this work we follow the  terminology given by Gutierrez and Sotomayor by calling these singularities as {\it simple} and remark that there is a different notion of simple singularities in the classical classification theory of singularities of mappings (see \cite{gibson1979singular}).}. They arise when the regularity conditions in the definition of smooth surfaces of $\mathbb{R}^3$, in terms of implicit functions and immersions, are not satisfied but there is sitll the presence of some certain stability  \cite{gibson1979singular}. The  effect  is the increase of the  types of admissible local charts.

\newtheorem{Def}{Definition}

\begin{Def}\label{def:localcharts}
A subset $\mathbf{M}\subset\mathbb{R}^l$ is called a
 \textbf{$2$-manifold with simple singularities}, or a \textbf{GS $2$-manifold},
if for every point $p\in \mathbf{M}$ there is a neighbourhood $V_p$ of $p$ in $\mathbf{M}$ and 
a  $C^{\infty}$-diffeomorphism $\Psi:V_p\rightarrow \mathcal{P}$ such that
$\Psi(p)=0$ and $\mathcal{P}$ is one of the following subsets of $\mathbb{R}^3$:

\begin{enumerate}
		\item[i) ] $\mathcal{R} = \{(x, y, z): z = 0\}$, plane; 
		\item[ii) ] $\mathcal{C} = \{(x, y, z): z^2 - y^2 - x^2 = 0\}$, cone;
		\item[iii) ] $\mathcal{W} = \{(x, y, z): zx^2 - y^2 = 0, z \geq 0\}$, Whitney's umbrella\footnote{The locus  of the subset  $\mathcal{W}$ is called cross-cap. In this paper, we chose to keep the nomenclature used by Gutierrez and Sotomayor in  \cite{gutierrez1982stable}.};
		\item[iv) ] $\mathcal{D} = \{(x, y, z): xy = 0\}$, double crossing;
		\item[v) ] $\mathcal{T} = \{(x, y, z): xyz = 0\}$, triple crossing;
		\end{enumerate}

\end{Def}

The subsets  $\mathbf{M(\mathcal{P})} \subset \mathbf{M}$  of the  points of  $ \mathbf{M}$ which admit  local charts of  type  $\mathcal{P}$, where $\mathcal{P} = \mathcal{R}, \mathcal{C}, \mathcal{W}, \mathcal{D}, \mbox{ or } \mathcal{T}$, provide a decomposition of  $\mathbf{M}$, where the regular part  $\mathbf{M(\mathcal{R})}$ of $\mathbf{M}$  is a $2$-dimensional manifold, $\mathbf{M(\mathcal{D})}$ is a $1$-dimensional manifold, while $\mathbf{M(\mathcal{C})}$, $\mathbf{M(\mathcal{W})}$ and $\mathbf{M(\mathcal{T})}$  are discrete sets. Moreover, the collection  \{$\mathbf{M(\mathcal{P})}$, \ $\mathcal{P}$\}  is a stratification  of  $M$ in the sense of Thom,  \cite{thom1969ensembles}, hence $$\mathbf{M} = \bigcup_{\mathcal{P}}{ \mathbf{M(\mathcal{P})}}, \mbox{ where } \mathcal{P} = \mathcal{R}, \mathcal{C}, \mathcal{W}, \mathcal{D}, \mathcal{T}.$$

\begin{Def}
A vector field $\mathbf{X}$ of class $C^r$ on $\mathbb{R}^{l}$  is  \textbf{tangent to a manifold $\mathbf{M} \subset \mathbb{R}^{l}$ with simple singularities} if it is tangent to the smooth submanifolds $\mathbf{M(\mathcal{P})}$, for all $\mathcal{P} = \mathcal{R}, \mathcal{C}, \mathcal{W}, \mathcal{D}, \mathcal{T}$. The space of such vector fields  is denoted by  $\mathbf{\mathfrak{X}^{r}(\mathbf{M})}$.
\end{Def}

A flow $X_{t}$ associated to a vector field  $\mathbf{X} \in \mathbf{\mathfrak{X}^{r}(\mathbf{M})}$ is called a  Gutierrez-Sotomayor flow (GS flow, for short) on $\mathbf{M}$.

\begin{Def} Denote by $\Sigma^{r}(\mathbf{M})$  the set of all vector fields  $\mathbf{X} \in \mathbf{\mathfrak{X}^{r}(\mathbf{M})}$ satisfying:

\begin{itemize}
    \item $X$ has finitely many hyperbolic fixed points and hyperbolic periodic orbits;
    
    \item the singular limit cycles of  $X$  are simple and $X$ has no saddle connections;
    
    \item the $\alpha$-limit and  $\omega$-limit sets of every trajectory of  $X$ are fixed points, or a periodic orbits or else a singular cycle.
\end{itemize}

\end{Def}

In this paper we consider  GS flows associated to vector fields in  $\Sigma^{r}(\mathbf{M})$ without periodic orbits and singular cycles. The set of such vector fields is denoted by  $\Sigma^{r}_{0}(\mathbf{M})$.

\begin{Def}

Let $\mathbf{M}$ be a GS $2$-manifold and $\mathbf{X} \in \mathbf{\mathfrak{X}^{r}(\mathbf{M})}$ a vector field on $\mathbf{M}$. The \textbf{set of folds on} $\mathbf{M}$, denoted by $\mathcal{F}(\mathbf{M})$, is defined as:

$$\mathcal{F}(\mathbf{M}) = \mathbf{M}(\mathcal{D}) \setminus \mathcal{S}_{\mathcal{D}},$$

where $\mathcal{S}_{\mathcal{D}}$ is the set of double crossing points on $\mathbf{M}$ which are  singularities (stationary points) of $\mathbf{X}$.

\end{Def}

 A GS flow defined on a fold $F$ has the property that, given any point $x\in F$, the limit sets,  $\alpha(x)$ and $\omega(x)$,  are either a Whitney, double crossing or a triple crossing singularity.

\subsection{Conley Index and Isolating Blocks}

 Given a flow $X_t:\mathbf{M}\rightarrow \mathbf{M}$ and a subset  $S \subset \mathbf{M}$, define the \textit{maximal invariant set} of $S$, denoted by $\mbox{Inv}(S)$, as follows:
$$ \text{Inv}(S) = \{  x\in S \mid  X_{t}(x) \in S, \forall t \in \mathbb{R}  \}.$$  
A subset $S\subset \mathbf{M}$ is called an {\it isolated invariant set} with respect to the flow $X_{t}$ if there exists an {\it isolating neighborhood} $N$ for $S$, i.e., a compact set $N$ such that $ S= \text{Inv}(N) \subset \text{int}(N). $
An {\it index pair}  for an isolated invariant set $S$ is a pair of compact sets $L \subset N$ such that $\overline{N\backslash L}$ is an isolating neighborhood of $S$ in $M$,  $L$ is positively invariant relative to $N$ and  $L$ is the exiting set of the flow in $N$.

Conley proved that, given an isolated invariant set $S$, there exists an index pair $(N,L)$ for $S$. Also, if  $(N,L)$ and $(N',L')$ are two index pairs for $S$, then  the pointed spaces $(N/L,[L])$ and $(N'/L',[L'])$ have the same homotopy type. For more details, see \cite{conley1978isolated}.

\begin{Def}
Let $(N, L)$ be an index pair for an isolated invariant set $S$ with respect to a flow $X_{t}: \mathbf{M} \rightarrow \mathbf{M}$.

\begin{itemize}

\item[i) ] The \textbf{homotopy Conley index} of $S$ is the homotopy type of the pointed space $(N/L,[L])$.

\item[ii) ] The \textbf{homology Conley index} $CH_{k}(S)$ of $S$ is defined as the $k$-th reduced homology group of $(N/L,[L])$.

\item[iii) ] The \textbf{numerical Conley index} $h_{k}$ of $S$ is the rank of the homology Conley index $CH_{k}(S)$. 

\end{itemize}

\end{Def}

In \cite{hernan2019}, it was computed the homotopy and homology Conley indices for each type of  singularities of a GS flow associated to a vector field in $\Sigma^{r}_{0}(\mathbf{M})$.

\begin{Def}
An \textbf{isolating block}  is an isolating neighborhood  $N$ such that its entering and exiting sets, given respectively by  $$N^{+} = \{x \in N \ | \ \phi([0,T),x) \not\subseteq N, \ \forall T < 0\}$$ $$N^{-} = \{x \in N \ | \ \phi([0,T),x) \not\subseteq N, \ \forall T > 0\},$$ are both closed, with the additional property that the flow is transversal to the boundary of $N$.
\end{Def}

The existence of GS isolating blocks  follows directly from the existence of Lyapunov functions for GS flows without periodic orbits and singular cycles, proved in  \cite{hernan2019}. Recall that if  $p$ is a singularity of a GS flow and $f$ is a Lyapunov function with  $f(p) = c$,  then given $\epsilon > 0$ such that  their is no critical value of $f$ in $[c - \epsilon, c + \epsilon]$, the connected component $N$  of  $f^{-1}([c - \epsilon, c + \epsilon])$ which contains  $p$  is an isolating block for $p$. In this case, the exiting set is $N^{-} = N \cap f^{-1}(c - \epsilon)$.

The GS isolating blocks are a very useful tool on the construction of examples of $GS$ flows. In fact, given  a list of GS isolating blocks, we can successively  glue the connected components of the entering set of a block with homeomorphic connected components of the exiting set of another block, until a closed singular manifold is obtained.

In what follows, we consider some topological information on the connected components of the boundary $N^{+}$ and $N^{-}$, given by the Lyapunov semi-graphs associated to GS isolating blocks.

%%%%%%%%%%%%

\subsection{Lyapunov semi-graphs for GS singularities}\label{sec:grafos}

Given a flow $X_{t}$ associated to a vector field $\mathbf{X} \in \Sigma^{r}_{0}(\mathbf{M})$, there is an \textbf{associated  Lyapunov function}  $f: \mathbf{M} \rightarrow \mathbb{R}$ such that $f(p) \neq f(q)$ if $p$ and $q$ are different singularities of $X{_t}$, and for each stratum $\mathbf{M(\mathcal{P})}$ of $\mathbf{M}$ it follows that:

\begin{itemize}
    \item $f|_{\mathbf{M(\mathcal{P})}}$ is a smooth function, with $f$ continuous in $\mathbf{M}$,
    
    \item The critical points of  $f|_{\mathbf{M(\mathcal{P})}}$ are non degenerate and coincide with the singularities of $X_{t}$,
    
    \item $\displaystyle\frac{d}{dt}(f|_{\mathbf{M(\mathcal{P})}}(X_{t}x)) < 0$, if $x$ is not a singularity of $X_{t}$.
\end{itemize}

Note that the Lyapunov function does not need to be smooth globally, only continuous. The existence of this function is proven in \cite{hernan2019}.

In this work we make use of Lyapunov graphs and semi-graphs as a combinatorial tool that keeps track of topological and dynamical data of a flow on 
$\mathbf{M}$. This information on a GS flow $X_{t}$ on $\mathbf{M}$ or on an isolating block $N$ is transferred to a Lyapunov graph, respectively, semi-graph associated to $X_{t}$ and $f$ as follows: 

\begin{Def}
Let $f$ be a Lyapunov function for a GS flow $X_{t}$ associated to $\mathbf{X} \in \Sigma^{r}_{0}(\mathbf{M})$.
Two points  $x$ and $y$ on $\mathbf{M}$ are \textbf{equivalent} if and only if they belong to the same connected component of a level set of $f$. In this case, denote $x\sim_{f}y$. The relation $\sim_{f}$ defines an \textbf{equivalence relation} on $\mathbf{M}$. The quotient space 
$\mathbf{M}/\sim_{f}$ is called a \textbf{Lyapunov graph of  $\mathbf{M}$ asssociated to $X_{t}$ and $f$}. For some isolating block $N \subset \mathbf{M}$, the quotient space $N/\sim_{f}$ is called a \textbf{Lyapunov semi-graph of $N$ asssociated to $X_{t}$ and $f$}. The quocient space satisfies:

\begin{itemize}
    \item[i) ] each connected component of a level set  $f^{-1}(c)$ colapses to a point. Thus, as the value $c$ varies, $f^{-1}(c)/\sim_{f}$ describes  a finite set of  \textbf{edges}, each of which is labelled with a  \textbf{weight}, which corresponds to the first Betti number of the respective connected component;
    
    \item[ii) ] if  $c_{0}$ is such that, a connected component of $f^{-1}(c_{0})$ contains a singularity of  $X_{t}$, then this connected component is a \textbf{vertex} $v$ in the quotient space labelled with the numerical Conley index of the singularity \footnote{Equivalently, the numerical index will be, at times, substituted by the nature of the singularity}, $(h_{0}, h_{1}, h_{2})_{\mathcal{P}}$, together with its type $\mathcal{P} = \mathcal{R}, \mathcal{C}, \mathcal{W}, \mathcal{D}$ or $\mathcal{T}$;
    
    \item[iii) ] the number of positively ( resp. negatively)  incident edges on the vertex $v$, i.e., the indegree (resp. outdegree)  of $v$,  is denoted by $e_{v}^{+} (resp. e_{v}^{-})$.
\end{itemize}

\end{Def}

Moreover, a Lyapunov graph  of $\mathbf{M}$ as well as a Lyapunov semi-graph of $N$ associated to  $X_{t}$ and $f$ is finite, directed and with no oriented cycles. 
For more details, see 
 \cite{hernan2019}.

Let $X_{t}$ be a  GS flow associated to a vector field $\mathbf{X} \in \Sigma^{r}_{0}(\mathbf{M})$, such that $p$ is a singularity of $X_{t}$, and $N$ is an isolating block for $p$. Then  $(N, N^{-})$ constitutes an index pair for  $inv(N) = \{p\}$. In this case, the next result which is referred to as the  \textit{ Poincaré-Hopf condition} for GS flows, see \cite{hernan2019},  follows:

\newtheorem{Teo}{Theorem}

\begin{Teo}[Poincaré-Hopf Condition]\label{teoPH}
Let $(N_{1}, N_{0})$  be an index pair for a singularity $p$ on a GS $2$-manifold, $M$. Let $\mathbf{X} \in \Sigma^{r}_{0}(\mathbf{M})$ and $(h_{0}, h_{1}, h_{2})$ be the numerical Conley indices for  $p$. Then:

\begin{equation}\label{PH}
    (h_{2} - h_{1} + h_{0}) - (h_{2} - h_{1} + h_{0})^{\ast} = e^{+} - \mathcal{B}^{+} - e^{-} + \mathcal{B}^{-}
\end{equation}

where $^{\ast}$denotes the indices of the reverse flow, $e^{+}$ (resp., $e^{-}$) is the number of connected components of the entering (resp., exiting) set  of  $N_{1}$ and  $\mathcal{B}^{+} = \displaystyle\sum_{k=1}^{e_{k}^{+}}{b_{k}^{+}}$ (resp., $\mathcal{B}^{-} = \displaystyle\sum_{k=1}^{e_{k}^{-}}{b_{k}^{-}}$), where $b_{k}^{+}$ (resp., $b_{k}^{-}$) is the first Betti number of the $k$-th connected component of the entering (resp., exiting) set  of  $N_{1}$.
\end{Teo}

On the other hand, note that the boundary of $N$ is nonempty, i.e.
 $\partial N = N^{+} \cup N^{-} \neq \emptyset$.  Hence $H_{2}(N) = 0$. Since $N$ is connected,  $\tilde{H_{0}}(N) = 0$. Consequently, from the long exact sequence $$0 \longrightarrow CH_{2}(p) \overset{\partial_{2}}{\longrightarrow} H_{1}(N^{-}) \overset{i_{1}}{\longrightarrow} H_{1}(N) \overset{p_{1}}{\longrightarrow} CH_{1}(p) \overset{\partial_{1}}{\longrightarrow} \tilde{H_{0}}(N^{-}) \longrightarrow 0$$ one has
\begin{equation}\label{sup}
dim(\tilde{H_{0}}(N^{-})) \leq dim(CH_{1}(p)) = h_{1}    
\end{equation}

Analogously, considering the reverse flow, one has:
\begin{equation}\label{sup2}
dim(\tilde{H_{0}}(N^{+})) \leq dim(CH_{1}(p)) = h_{1}^{\ast}    
\end{equation}

From  (\ref{PH}), (\ref{sup}) and (\ref{sup2}), one can describe how the GS isolating blocks look like in terms of their  Lyapunov semi-graphs. More specifically, let $v$ be the vertex  on the Lyapunov semi-graph of an isolating block for a GS singularity $p$,  then the positively incidents edges in $v$ ($e_{v}^{+}$)  and the  negatively incidents edges in $v$ ($e_{v}^{-}$) satisfy:
\begin{equation}\label{st}
    dim(\tilde{H_{0}}(N^{-}) = dim(H_{0}(N^{-})) - 1 = e_{v}^{-} - 1 \leq h_1
\end{equation}
\begin{equation}\label{nd}
    dim(\tilde{H_{0}}(N^{+}) = dim(H_{0}(N^{+})) - 1 = e_{v}^{+} - 1 \leq h_1^{\ast} 
\end{equation}

Therefore, considering all the possibilities for  $e_{v}^{-}$ in (\ref{st}) and for $e_{v}^{+}$ in (\ref{nd}),  together with the  respective  Poincaré-Hopf conditions (\ref{PH}), the necessary conditions on a Lyapunov semi-graph for a GS isolating block were obtained in  \cite{hernan2019} for each type of GS singularity as well as their nature. 

Roughly speaking, we can define the {\bf nature of a singularity} as follows.  If a singularity is an attractor (resp. repeller), we say it has {attracting} (resp. {repelling}) {nature}, for short, nature $\mathbf{a}$ (resp. nature $\mathbf{r}$). In the particular case of a singularity of type $\mathcal{D}$ since it has two attracting (resp. repelling) natures, we say it has nature $\mathbf{a^{2}}$ (resp. $\mathbf{r^{2}}$). Similarly, in the case of a singularity of type $\mathcal{T}$, since it has  three attracting (resp. repelling) natures, we say it has nature $\mathbf{a^{3}}$ (resp. $\mathbf{r^{3}}$).
A singularity of type $\mathcal{R}$ or  $\mathcal{C}$ which is neither of attracting or repelling nature is said to have saddle nature, for short, nature $\mathbf{s}$. A regular saddle on a disk identified along the stable (resp. unstable) manifolds so as to produce a Whitney chart is of nature $\mathbf{s_s}$ (resp. $\mathbf{s_u}$). Two regular saddles on two disjoint disks identified along their stable (resp. unstable) manifold   so as to produce a double crossing chart is of type $\mathbf{ss_s}$ (resp. $\mathbf{ss_u}$). A regular saddle and an attracting (resp. repelling) singularity on two disjoint disks identified along a stable (resp. unstable) direction so as to produce a double crossing chart is of nature $\mathbf{sa}$ (resp. $\mathbf{sr}$).  Similarly, two regular saddles and  an attracting (resp. repelling) singularity on three disjoint disks identified as follows:
the two saddle disks identify along their unstable (resp. stable) manifolds and subsequently  the attracting (resp., repelling) disk  identified along  stable (resp. unstable) directions so as to produce a triple crossing chart is of nature $\mathbf{ssa}$ (resp. $\mathbf{ssr}$).
 For more details  see \cite{hernan2019}.

In Table \ref{tab:RCWDT} these conditions are presented for all GS singularities with natures $a, \ a^2, \ a^3 \ s, \ s_{s}, \ sa, \ ss_{s}$ and $ssa$. By reserving the flow, one obtains the conditions on Lyapunov semi-graphs for GS singularities with natures  $r, \ r^2, \ r^3, \ s, \ s_{u}, \ sr, \ ss_{u}$ and $ssr$, respectively. Also the corresponding Poincaré-Hopf condition is given by exchanging $\mathcal{B}^{+}$ by $\mathcal{B}^{-}$ and vice versa.

\newcolumntype{C}[1]{>{\centering\let\newline\\\arraybackslash\hspace{0pt}}m{#1}}

%%%%%%%%%%%%%%%%%%%%%%%%%%%%%%%%%%%%%%%%%%%%%%%%%%%%%%%%%%%	
%%%%%%%%%%%%%%%%%%%%%%%%%%%%%%%%%%%%%%%%%%%%%%%%%%%%%%%%%%%
%%%%%%%%%%%%%%%%%%%%%%%%%%%%%%%%%%%%%%%%%%%%%%%%%%%%%%%%%%%
%%%%%%%%%%%%%%%%%%%%%%%%%%%%%%%%%%%%%%%%%%%%%%%%%%%%%%%%%%%
%%%%%%%%%%%%%%%%%%%%%%%%%%%%%%%%%%%%%%%%%%%%%%%%%%%%%%%%%%%

 \begin{table}[!htb]
        \centering
        \resizebox{\linewidth}{!}{
        \begin{tabular}{C{2.2cm}cccc|cccc}
            \cellcolor{gray!40} Type & \multicolumn{3}{c}{\cellcolor{gray!40} Regular} & \multicolumn{2}{c}{\cellcolor{gray!40}} & \multicolumn{3}{c}{\cellcolor{gray!40} Cone} \\
            \hline
            \cellcolor{gray!20} Nature & \cellcolor{gray!20} $a$ & \cellcolor{gray!20} $s$ & \cellcolor{gray!20} $s$ & \multicolumn{2}{c}{\cellcolor{gray!20}} & \cellcolor{gray!20} $a$ & \cellcolor{gray!20} $s$ & \cellcolor{gray!20} $s$ \\
            \hline
            & & & & & & & & \\
            Lyapunov semi-graph & \includegraphics[align=c,scale=0.5]{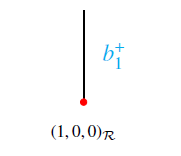} & \includegraphics[align=c,scale=0.5]{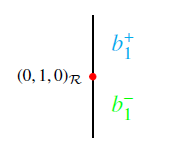} & \includegraphics[align=c,scale=0.5]{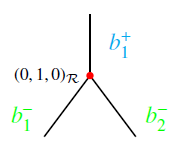} & \hspace{0.5cm} & \hspace{0.5cm} & \includegraphics[align=c,scale=0.5]{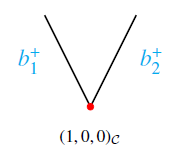} & \includegraphics[align=c,scale=0.5]{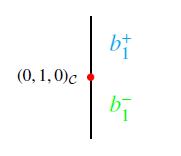} & \includegraphics[align=c,scale=0.5]{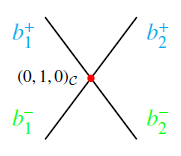} \\
            & & & & & & & & \\
            Poincaré-Hopf condition & $\mathcal{B}^{+} = 1$ & $\mathcal{B}^{+} = \mathcal{B}^{-}$ & $\mathcal{B}^{+} = \mathcal{B}^{-} - 1$ & & & $b_{1}^{+} = b_{2}^{+} = 1$ & $\mathcal{B}^{+} = \mathcal{B}^{-}$ & $\mathcal{B}^{+} = \mathcal{B}^{-}$ \\
            & & & & & & & & \\
            %\hline
%        \end{tabular}}
%        \caption{Lyapunov semi-graphs for a singularity of type regular}
%        \label{tab:R2}
%    \end{table}
%
%
%	
%%%%%%%%%%%%%%%%%%%%%%%%%%%%%%%%%%%%%%%%%%%%%%%%%%%%%%%%%%%%	    
%    
%    
% \begin{table}[!htb]
%        \centering
%        \resizebox{\linewidth}{!}{
%        \begin{tabular}{cccccccc}
            \cellcolor{gray!40} Type & \multicolumn{3}{c}{\cellcolor{gray!40} Whitney} & \multicolumn{2}{c}{\cellcolor{gray!40}} & \multicolumn{3}{c}{\cellcolor{gray!40} Double crossing} \\
            \hline
            \cellcolor{gray!20} Nature & \cellcolor{gray!20} $a$ & \cellcolor{gray!20} $s$ & \cellcolor{gray!20} $s$ & \multicolumn{2}{c}{\cellcolor{gray!20}} & \cellcolor{gray!20} $a$ & \cellcolor{gray!20} $sa$ & \cellcolor{gray!20} $sa$ \\
            \hline
            & & & & & & & & \\
            Lyapunov semi-graph & \includegraphics[align=c,scale=0.5]{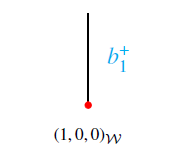} & \includegraphics[align=c,scale=0.5]{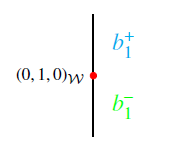} & \includegraphics[align=c,scale=0.5]{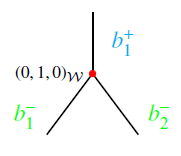} & \hspace{0.5cm} & \hspace{0.5cm} & \includegraphics[align=c,scale=0.5]{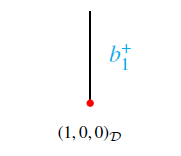} & \includegraphics[align=c,scale=0.5]{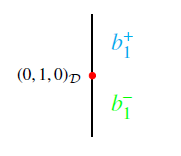} & \includegraphics[align=c,scale=0.5]{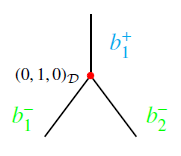} \\
            & & & & & & & & \\
            Poincaré-Hopf condition & $\mathcal{B}^{+} = 2$ & $\mathcal{B}^{+} = \mathcal{B}^{-} + 1$ & $\mathcal{B}^{+} = \mathcal{B}^{-}$ & & & $\mathcal{B}^{+} = 3$ & $\mathcal{B}^{+} = \mathcal{B}^{-} + 2$ & $\mathcal{B}^{+} = \mathcal{B}^{-} + 1$ \\
            & & & & & & & & \\
            %\hline
%        \end{tabular}}
%        \caption{Lyapunov semi-graphs for a singularity of type regular}
%        \label{tab:R2}
%    \end{table}    
%   
%    
%    
%    
%    
%    
%    
%%%%%%%%%%%%%%%%%%%%%%%%%%%%%%%%%%%%%%%%%%%%%%%%%%%%%%%%%%%%    
%
%
%
% \begin{table}[!htb]
%        \centering
%        \resizebox{\linewidth}{!}{
%        \begin{tabular}{cccccccc}
            \cellcolor{gray!40} Type & \multicolumn{8}{c}{\cellcolor{gray!40} Double crossing} \\
            \hline
            \cellcolor{gray!20} Nature & \cellcolor{gray!20} $ss_{s}$ & \cellcolor{gray!20} $ss_{s}$ & \cellcolor{gray!20} $ss_{s}$ & \multicolumn{2}{c}{\cellcolor{gray!20}} & \cellcolor{gray!20} $ss_{s}$ & \cellcolor{gray!20} $ss_{s}$ & \cellcolor{gray!20} $ss_{s}$ \\
            \hline
            \multicolumn{9}{c}{} \\
            Lyapunov semi-graph & \includegraphics[align=c,scale=0.5]{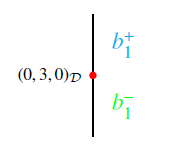} & \includegraphics[align=c,scale=0.5]{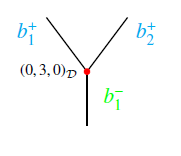} & \includegraphics[align=c,scale=0.5]{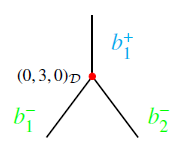} & \multicolumn{2}{c}{} & \includegraphics[align=c,scale=0.5]{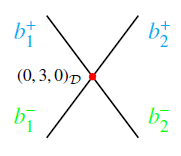} & \includegraphics[align=c,scale=0.5]{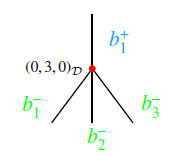} & \includegraphics[align=c,scale=0.5]{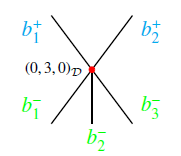} \\
            \multicolumn{9}{c}{} \\
            Poincaré-Hopf condition & $\mathcal{B}^{+} = \mathcal{B}^{-} + 2$ & $\mathcal{B}^{+} - 3 = \mathcal{B}^{-}$ & $\mathcal{B}^{+} = \mathcal{B}^{-} + 1$ & \multicolumn{2}{c}{} & $\mathcal{B}^{+} = \mathcal{B}^{-} + 2$ & $\mathcal{B}^{+} = \mathcal{B}^{-}$ & $\mathcal{B}^{+} = \mathcal{B}^{-} + 1$ \\
            \multicolumn{9}{c}{} \\
            %\hline
%        \end{tabular}}
%        \caption{Lyapunov semi-graphs with the respective Poincaré-Hopf condition}
%        \label{tab:R2}
%    \end{table}
%
%
%	
%%%%%%%%%%%%%%%%%%%%%%%%%%%%%%%%%%%%%%%%%%%%%%%%%%%%%%%%%%%%	  
%
%
%
%
% \begin{table}[!htb]
%        \centering
%        \resizebox{\linewidth}{!}{
%        \begin{tabular}{cccccccc}
           \cellcolor{gray!40} Type & \multicolumn{3}{c}{\cellcolor{gray!40} Double crossing} & \multicolumn{2}{c}{\cellcolor{gray!40}} & \multicolumn{3}{c}{\cellcolor{gray!40} Triple crossing} \\
            \hline
            \cellcolor{gray!20} Nature & \cellcolor{gray!20} $ss_{s}$ & \cellcolor{gray!20} $ss_{s}$ & \cellcolor{gray!20} & \multicolumn{2}{c}{\cellcolor{gray!20}} & \cellcolor{gray!20} $a$ & \cellcolor{gray!20} $ssa$ & \cellcolor{gray!20} $ssa$ \\
            \hline
            & & & & & & & & \\
            Lyapunov semi-graph & \includegraphics[align=c,scale=0.5]{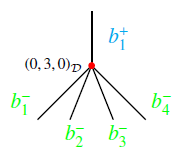} & \includegraphics[align=c,scale=0.5]{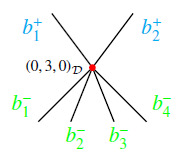} & & \hspace{0.5cm} & \hspace{0.5cm} & \includegraphics[align=c,scale=0.5]{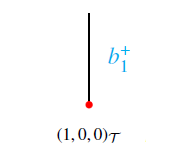} & \includegraphics[align=c,scale=0.5]{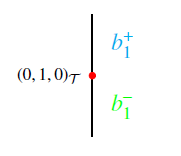} & \includegraphics[align=c,scale=0.5]{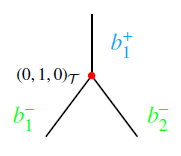} \\
            & & & & & & & & \\
            Poincaré-Hopf condition & $\mathcal{B}^{+} = \mathcal{B}^{-} - 1$ & $\mathcal{B}^{+} = \mathcal{B}^{-}$ & & & & $b_{1}^{+} = 7$ & $\mathcal{B}^{+} = \mathcal{B}^{-} + 2$ & $\mathcal{B}^{+} = \mathcal{B}^{-} + 1$ \\
            & & & & & & & & \\
            \hline
            \multicolumn{9}{c}{} \\
        \end{tabular}}
        \caption{Lyapunov semi-graphs with the respective Poincaré-Hopf condition}
        \label{tab:RCWDT}
    \end{table}

For each semi-graph in Table \ref{tab:RCWDT}, considering that $b_{i}^{\pm}$ is greater or equal to one, we define the respective \textbf{Lyapunov semi-graph with minimal weights} to be the one with smallest positive integers satisfying the Poincaré-Hopf condition.

In the next section we will consider the realizability of an abstract Lypaunov semi-graph satisfying the Poincaré-Hopf condition.

\section{Construction of GS Isolating Blocks}\label{sec:construcao}

In \cite{hernan2019}, the construction of GS isolating blocks was presented in a similar fashion to the construction of compact manifolds via the classical Handle Theory. Roughly speaking, a GS handle of type $\mathcal{R}, \mathcal{C}, \mathcal{W}, \mathcal{D}$ or $\mathcal{T}$ is a subspace of $\mathbb{R}^{3}$ homeomorphic to a chart presented in Definition \ref{def:localcharts} and with a  well-defined local GS flow. Given a GS handle $\mathcal{H}^{\mathcal{P}}_{\eta}$ containing a singularity $p$ of type $\mathcal{P}$ and nature $\eta$, we consider the gluing of the attaching region of the handle $\mathcal{H}^{\mathcal{P}}_{\eta}$ on a \textit{distinguished branched $1$-manifold}. See \cite{hernan2019} for more details.

\begin{Def}\label{def:ramificada}
A \textbf{distinguished branched $1$-manifold}  is a topological space containing at most 
$4$ connected components  each of which has a finite number of  \textbf{branched charts}, where each branched chart is an intersection of two transverse arcs. This intersection of two transverse arcs is called a \textbf{branch point}.
\end{Def}

In what follows, a schematic description of the steps in the construction of GS isolating blocks $N$ for a singularity  $p \in \mathcal{H}^{\mathcal{P}}_{\eta}$ is presented.
\vspace{0.2cm}

 \textbf{Steps in the Construction of the Isolating Blocks}
\vspace{-0.2cm}
\begin{enumerate}

    \item \textbf{Identifying the attaching region $A_{k}$}
    
   The attaching region denoted by  $A_{k}$ is given by the unstable part of $\mathcal{H}^{\mathcal{P}}_{\eta}$, which resembles the attaching sphere in the classical Handle Theory;
    
    \item \textbf{Choosing a distinguished branched  $1$-manifold $N^{-}$}
    
  In order to guarantee that the isolating block is connected, the chosen branched  $1$-manifold, denoted by  $N^{-}$, must not contain more connected components than  $A_{k}$;
    
    \item \textbf{Gluing $A_{k}$ in $N^{-} \times [0,1]$}
    
   The gluing of the handle  $\mathcal{H}^{\mathcal{P}}_{\eta}$ on a collar of a  distinguished  branched $1$- manifold,  $N^{-}\times [0,1]$,  given by  any embedding  $f: A_{k} \rightarrow N^{-} \times \{1\}$ with the property that
 it maps at least one connected component of  $A_{k}$ to each connected component of  $N^{-} \times \{1\}$ produces ${N}$;
    
    \item \textbf{Stretching ${N}$}
    
     Stretch $N$ in the direction of the time-reversed flow.

\end{enumerate}

Note that in the above construction,  the distinguished  branched $1$-manifold is denoted by $N^{-}$ precisely because it corresponds to the exiting set  of the isolating block $N$. Also, the reason that the  branched $1$-manifolds have at most $4$ connected  components is because  the maximum number of connected components admitted by the attaching region of a GS handle is $4$, \cite{hernan2019}.

\subsection{Minimal GS Isolating Blocks}

Let $N$ be a GS isolating block for  $p \in \mathcal{H}^{\mathcal{P}}_{\eta}$, determined  by an embedding $f: A_{k} \rightarrow N^{-} \times \{1\}$, where $A_{k}$ is the unstable part of  $\partial\mathcal{H}^{\mathcal{P}}_{\eta}$ and  $N^{-} \times \{1\}$  is a  distinguished  branched $1$-manifold. Since  $f$ is an embedding, it follows that the number of branched charts in $N^{-} \times \{1\} = N^{-} \times \{0\} = N^{-}$  is greater or equal to the number of branched charts in  $A_{k}$.

\begin{Def}
Let $p \in \mathcal{H}^{\mathcal{P}}_{\eta}$ be a  GS singularity  and $N$ an isolating block for $p$.  $N$ is a \textbf{minimal GS isolating block} for $p$ if the exiting set  $N^{-}$ has the same number of branched charts as the unstable part of $\partial\mathcal{H}^{\mathcal{P}}_{\eta}$.
\end{Def}

\begin{Teo}\label{teo:realizationminimalsemi}
A Lyapunov semi-graph $L_{v}$ with a single vertex $v$ labelled with a GS singularity is associated to a GS flow on a minimal GS isolating block if and only if:
\begin{itemize}

\item[a) ] $L_{v}$ satisfies the Poincaré-Hopf condition with minimal weights;

    \item[b) ]  if $v$ is labelled with a singularity of type $\mathcal{D}$ and nature $ss_{s}$ (resp. $ss_{u}$) such that $e_{v}^{+} = 2$ (resp. $e_{v}^{-} = 2$), then $e_{v}^{-} \leq 2$ and $b_{1}^{+} = b_{2}^{+}$ (resp. $e_{v}^{+} \leq 2$ and $b_{1}^{-} = b_{2}^{-}$);
    
    \item[c) ]  if $v$ is labelled with a singularity of type $\mathcal{T}$ and nature $ssa$ (resp. $ssr$) such that $e_{v}^{-} = 2$ (resp. $e_{v}^{+} = 2$), then $b_{1}^{-} = b_{2}^{-}$ (resp. $b_{1}^{+} = b_{2}^{+}$).
\end{itemize}
\end{Teo}

 \begin{proof}
 It follows from Theorem \ref{teo:colecao} and Theorem \ref{teo:combinatorial}.
 \end{proof}

   In other words, the boundaries of a minimal GS isolating block are distinguished  branched $1$-manifolds that  have the minimal number of branch points admitted by the singularity. In the case of attracting or repelling singularities,  the choice of  distinguished  branched $1$-manifold is unique.

\newtheorem{Prop}{Proposition}

\begin{Prop}\label{proposicao}
Let $p \in \mathcal{H}^{\mathcal{P}}_{a}$ (resp. $p \in \mathcal{H}^{\mathcal{P}}_{r}$) be a GS singularity of attracting (resp. repelling) nature of type $\mathcal{R}, \mathcal{C}, \mathcal{W}, \mathcal{D}$ or $\mathcal{T}$. Then, $p$ admits a unique minimal GS isolating block up to homeomorphism.
\end{Prop}

\begin{proof}

In the case of an attracting singularity $p \in \mathcal{H}^{\mathcal{P}}_{a}$ the result is trivial, since $A_{k} = \emptyset$.

Given  $p \in \mathcal{H}^{\mathcal{P}}_{r}$,  an  repelling  singularity,  since  $A_{k} = \partial \mathcal{H}^{\mathcal{P}}_{r}$, then $N^{-} = \partial \mathcal{H}^{\mathcal{P}}_{r}$ is the unique choice for  a distinguished  branched $1$-manifold that permits the gluing to occur via the embedding $f = Id$. Hence, a repelling singularity  admits a unique isolating block  which is homotopic to the local chart itself,  $\mathcal{H}^{\mathcal{P}}_{r}$. 

\begin{table}[!htb]
			\centering
			\resizebox{0.8\linewidth}{!}{
			\begin{tabular}{ccccccccccccccc}
			 \cellcolor{gray!20} & \cellcolor{gray!20} $p \in \mathcal{H}^{\mathcal{R}}_{a}$ & \cellcolor{gray!20} & \cellcolor{gray!20} & \cellcolor{gray!20} $p \in \mathcal{H}^{\mathcal{C}}_{a}$ & \cellcolor{gray!20} & \cellcolor{gray!20} & \cellcolor{gray!20} $p \in \mathcal{H}^{\mathcal{W}}_{a}$ & \cellcolor{gray!20} & \cellcolor{gray!20} & \cellcolor{gray!20} $p \in \mathcal{H}^{\mathcal{D}}_{a}$ & \cellcolor{gray!20} & \cellcolor{gray!20} & \cellcolor{gray!20} $p \in \mathcal{H}^{\mathcal{T}}_{a}$ & \cellcolor{gray!20} \\
			  \hline
			   &  &  &  &  &  &  &  &  &  &  &  &  &  &  \\
			   & \includegraphics[align=c,scale=0.35]{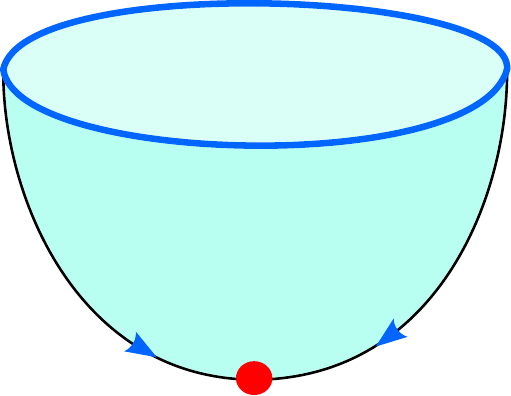} &  &  & \includegraphics[align=c,scale=0.45]{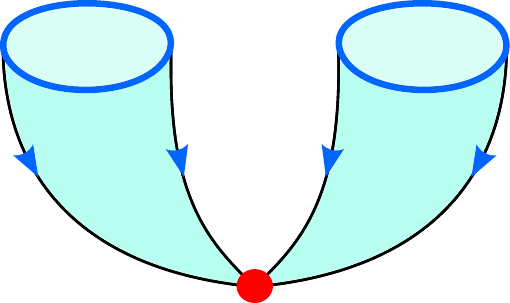} &  &  & \includegraphics[align=c,scale=0.6]{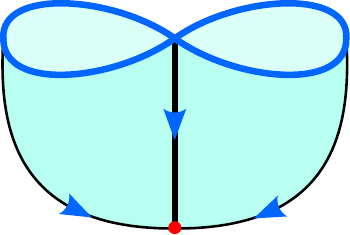} &  &  & \includegraphics[align=c,scale=0.45]{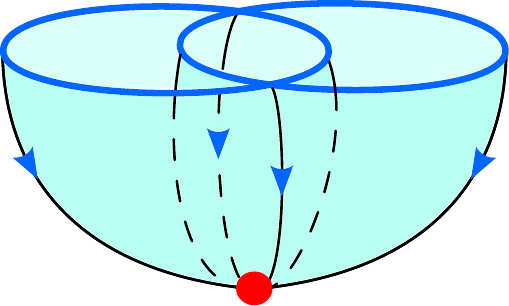} &  &  & \includegraphics[align=c,scale=0.5]{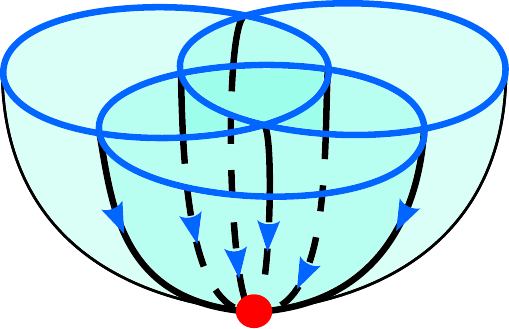} &  \\
			   &  &  &  &  &  &  &  &  &  &  &  &  &  &  \\
			\end{tabular}}
			\caption{GS isolating blocks for attracting singularities}
			\label{tab:blocoA}
		\end{table}

%\newpage

By reversing the flow in Table \ref{tab:blocoA},    the minimal GS isolating blocks for repelling singularities are obtained.

\end{proof}

On the other hand, the construction of minimal GS
isolating blocks  for singularities of different saddle natures, admits different choices of distinguished branched $1$-manifolds and consequently different ways of embedding $A_{k}$ in $N^{-}\times [0,1]$.
In the next result, the possible embeddings are characterized in terms of the  distinguished  branched $1$-manifolds of resulting isolating blocks.

\begin{Teo}\label{teo:colecao}
Let  $\mathcal{H}^{\mathcal{P}}_{\eta}$ be a  GS handle for a singularity $p$
 of type $\mathcal{P}$ and nature $n$ and $N$ a minimal GS isolating block for $p$. Then  all possible distinguished  branched $1$-manifolds that form the connected components of the entering and exiting sets of $N$  are described in
 Table \ref{tab:bordos-colecao}, in terms of the Lyapunov semi-graph of  $N$. This characterization is up to flow reversal.
\end{Teo}

\begin{table}[htb]
    \centering
    \begin{tikzpicture}

    \node {
            \resizebox{\linewidth}{!}{
            \begin{tabular}{ccccc}
                \cellcolor{gray!20} $p \in \mathcal{H}^{\mathcal{R}}_{\eta}$ & \cellcolor{gray!20} & \cellcolor{gray!20} & \cellcolor{gray!20} $p \in \mathcal{H}^{\mathcal{C}}_{\eta}$ & \cellcolor{gray!20} \\
                \hline
			    & & & & \\
			    \begin{tikzcd}
                \begin{overpic}[scale=0.3,align=c]{ra.png}
		        \put(32,88){\includegraphics[scale=0.2,align=c]{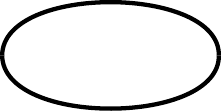}}
	            \end{overpic}
                \end{tikzcd} \begin{tikzcd}
                \begin{overpic}[scale=0.3,align=c]{rs1.png}
		        \put(36,86){\includegraphics[scale=0.2,align=c]{Branched1.pdf}}
		        \put(36,-9){\includegraphics[scale=0.2,align=c]{Branched1.pdf}}
	            \end{overpic}
                \end{tikzcd} \begin{tikzcd}
                \begin{overpic}[scale=0.3,align=c]{rs2.png}
		        \put(34,86){\includegraphics[scale=0.2,align=c]{Branched1.pdf}}
		        \put(66,-9){\includegraphics[scale=0.2,align=c]{Branched1.pdf}}
		        \put(3,-9){\includegraphics[scale=0.2,align=c]{Branched1.pdf}}
	            \end{overpic}
                \end{tikzcd} & & & \begin{tikzcd}
                \begin{overpic}[scale=0.3,align=c]{ca.png}
		        \put(6,88){\includegraphics[scale=0.2,align=c]{Branched1.pdf}}
		        \put(62,88){\includegraphics[scale=0.2,align=c]{Branched1.pdf}}
	            \end{overpic}
                \end{tikzcd} \begin{tikzcd}
                \begin{overpic}[scale=0.3,align=c]{cs1.png}
		        \put(34,86){\includegraphics[scale=0.2,align=c]{Branched1.pdf}}
		        \put(34,-11){\includegraphics[scale=0.2,align=c]{Branched1.pdf}}
	            \end{overpic}
                \end{tikzcd} \begin{tikzcd}
                \begin{overpic}[scale=0.3,align=c]{cs2.png}
		        \put(3,86){\includegraphics[scale=0.2,align=c]{Branched1.pdf}}
		        \put(66,86){\includegraphics[scale=0.2,align=c]{Branched1.pdf}}
		        \put(3,-11){\includegraphics[scale=0.2,align=c]{Branched1.pdf}}
		        \put(66,-11){\includegraphics[scale=0.2,align=c]{Branched1.pdf}}
	            \end{overpic}
                \end{tikzcd} & \\
                & & & & \\
                \cellcolor{gray!20} $p \in \mathcal{H}^{\mathcal{W}}_{\eta}$ & \cellcolor{gray!20} & \cellcolor{gray!20} & \cellcolor{gray!20} $p \in \mathcal{H}^{\mathcal{D}}_{\eta}$ & \cellcolor{gray!20} \\
                \hline
			    & & & & \\
			    \begin{tikzcd}
                \begin{overpic}[scale=0.3,align=c]{wa.png}
		        \put(26,88){\includegraphics[scale=0.2,align=c]{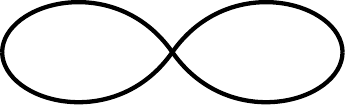}}
	            \end{overpic}
                \end{tikzcd} \begin{tikzcd}
                \begin{overpic}[scale=0.3,align=c]{wss1.png}
		        \put(26,86){\includegraphics[scale=0.2,align=c]{Branched2.pdf}}
		        \put(36,-9){\includegraphics[scale=0.2,align=c]{Branched1.pdf}}
	            \end{overpic}
                \end{tikzcd} \begin{tikzcd}
                \begin{overpic}[scale=0.3,align=c]{wss2.png}
		        \put(26,86){\includegraphics[scale=0.2,align=c]{Branched2.pdf}}
		        \put(66,-9){\includegraphics[scale=0.2,align=c]{Branched1.pdf}}
		        \put(3,-9){\includegraphics[scale=0.2,align=c]{Branched1.pdf}}
	            \end{overpic}
                \end{tikzcd} & & & \begin{tikzcd}
                \begin{overpic}[scale=0.3,align=c]{da.png}
		        \put(25,88){\includegraphics[scale=0.2,align=c]{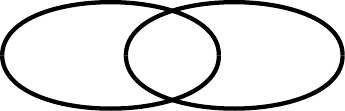}}
	            \end{overpic}
                \end{tikzcd} \begin{tikzcd}
                \begin{overpic}[scale=0.3,align=c]{dsa1.png}
                \put(-15,86){(}
		        \put(-5,86){\includegraphics[scale=0.2,align=c]{Branched3a.pdf}}
		        \put(48,86){,}
		        \put(58,86){\includegraphics[scale=0.2,align=c]{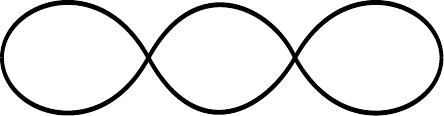}}
		        \put(125,86){)}
		        \put(36,-11){\includegraphics[scale=0.2,align=c]{Branched1.pdf}}
	            \end{overpic}
                \end{tikzcd} \hspace{0.4cm} \begin{tikzcd}
                \begin{overpic}[scale=0.3,align=c]{dsa2.png}
		        \put(27,86){\includegraphics[scale=0.2,align=c]{Branched3a.pdf}}
		        \put(3,-11){\includegraphics[scale=0.2,align=c]{Branched1.pdf}}
		        \put(66,-11){\includegraphics[scale=0.2,align=c]{Branched1.pdf}}
	            \end{overpic}
                \end{tikzcd} & \\
                & & & & \\
                \cellcolor{gray!20} & \cellcolor{gray!20} & \cellcolor{gray!20} $p \in \mathcal{H}^{\mathcal{D}}_{\eta}$ & \cellcolor{gray!20} & \cellcolor{gray!20} \\
                \hline
			    & & & & \\
			    \begin{tikzcd}
                \begin{overpic}[scale=0.3,align=c]{dsss1.png}
                \put(-15,86){(}
		        \put(-5,86){\includegraphics[scale=0.2,align=c]{Branched3a.pdf}}
		        \put(48,86){,}
		        \put(58,86){\includegraphics[scale=0.2,align=c]{Branched3b.pdf}}
		        \put(125,86){)}
		        \put(34,-11){\includegraphics[scale=0.2,align=c]{Branched1.pdf}}
	            \end{overpic}
                \end{tikzcd} \hspace{0.6cm} \begin{tikzcd}
                \begin{overpic}[scale=0.3,align=c]{dsssCG.png}
		        \put(-6,86){\includegraphics[scale=0.2,align=c]{Branched2.pdf}}
		        \put(56,86){\includegraphics[scale=0.2,align=c]{Branched2.pdf}}
		        \put(33,-11){\includegraphics[scale=0.2,align=c]{Branched1.pdf}}
	            \end{overpic}
                \end{tikzcd} \hspace{0.4cm} \begin{tikzcd}
                \begin{overpic}[scale=0.3,align=c]{dsss3.png}
		        \put(-15,86){(}
		        \put(-5,86){\includegraphics[scale=0.2,align=c]{Branched3a.pdf}}
		        \put(48,86){,}
		        \put(58,86){\includegraphics[scale=0.2,align=c]{Branched3b.pdf}}
		        \put(125,86){)}
		        \put(66,-9){\includegraphics[scale=0.2,align=c]{Branched1.pdf}}
		        \put(3,-9){\includegraphics[scale=0.2,align=c]{Branched1.pdf}}
	            \end{overpic}
                \end{tikzcd} & & & \begin{tikzcd}
                \begin{overpic}[scale=0.3,align=c]{dsss4.png}
		        \put(-6,88){\includegraphics[scale=0.2,align=c]{Branched2.pdf}}
		        \put(62,88){\includegraphics[scale=0.2,align=c]{Branched2.pdf}}
		        \put(3,-9){\includegraphics[scale=0.2,align=c]{Branched1.pdf}}
		        \put(62,-9){\includegraphics[scale=0.2,align=c]{Branched1.pdf}}
	            \end{overpic}
                \end{tikzcd} \hspace{0.4cm} \begin{tikzcd}
                \begin{overpic}[scale=0.3,align=c]{dsss5.png}
		        \put(-15,86){(}
		        \put(-5,86){\includegraphics[scale=0.2,align=c]{Branched3a.pdf}}
		        \put(48,86){,}
		        \put(58,86){\includegraphics[scale=0.2,align=c]{Branched3b.pdf}}
		        \put(125,86){)}
		        \put(34,-11){\includegraphics[scale=0.2,align=c]{Branched1.pdf}}
		        \put(0,0){\includegraphics[scale=0.2,align=c]{Branched1.pdf}}
		        \put(68,0){\includegraphics[scale=0.2,align=c]{Branched1.pdf}}
	            \end{overpic}
                \end{tikzcd} \hspace{0.4cm} \begin{tikzcd}
                \begin{overpic}[scale=0.3,align=c]{dsss7.png}
		        \put(27,86){\includegraphics[scale=0.2,align=c]{Branched3a.pdf}}
		        \put(-15,5){\includegraphics[scale=0.2,align=c]{Branched1.pdf}}
		        \put(12,-11){\includegraphics[scale=0.2,align=c]{Branched1.pdf}}
		        \put(58,-11){\includegraphics[scale=0.2,align=c]{Branched1.pdf}}
		        \put(90,5){\includegraphics[scale=0.2,align=c]{Branched1.pdf}}
	            \end{overpic}
                \end{tikzcd} & \\
                & & & & \\
                \cellcolor{gray!20} & \cellcolor{gray!20} & \cellcolor{gray!20} $p \in \mathcal{H}^{\mathcal{T}}_{\eta}$ & \cellcolor{gray!20} & \cellcolor{gray!20} \\
                \hline
			    & & & & \\
			    & & & & \\
			    \begin{tikzcd}
                \begin{overpic}[scale=0.3,align=c]{ta.png}
		        \put(26,95){\includegraphics[scale=0.25,align=c]{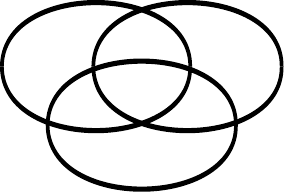}}
	            \end{overpic}
                \end{tikzcd} \hspace{0.5cm} \begin{tikzcd}
                \begin{overpic}[scale=0.3,align=c]{tssa2.png}
		        \put(-32,95){(}
		        \put(-18,95){\includegraphics[scale=0.13,align=c]{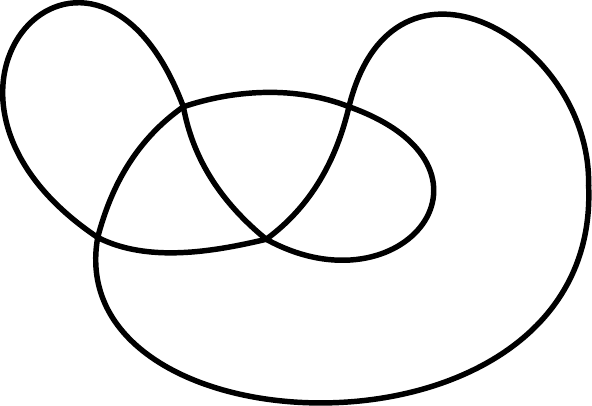}}
		        \put(45,95){,}
		        \put(61,95){\includegraphics[scale=0.2,align=c]{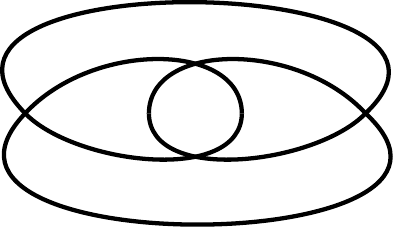}}
		        \put(124,95){)}
		        \put(-4,-11){\includegraphics[scale=0.2,align=c]{Branched2.pdf}}
		        \put(60,-11){\includegraphics[scale=0.2,align=c]{Branched2.pdf}}
	            \end{overpic}
                \end{tikzcd} 
                \hspace{0.5cm}
                \begin{tikzcd}
                \begin{overpic}[scale=0.3,align=c]{tssa1.png}
		        \put(28,100){\includegraphics[scale=0.22,align=c]{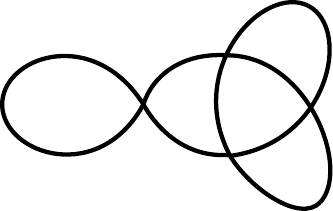}}
		        \put(25,-11){\includegraphics[scale=0.2,align=c]{Branched3b.pdf}}
	            \end{overpic}
                \end{tikzcd} & &  & \hspace{0.5cm} \begin{tikzcd}
                \begin{overpic}[scale=0.3,align=c]{tssa1.png}
		        \put(28,100){\includegraphics[scale=0.2,align=c]{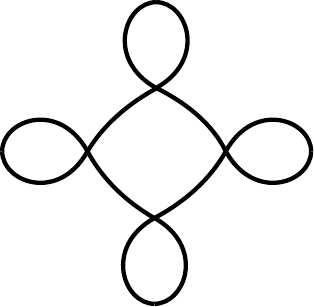}}
		        \put(28,-11){\includegraphics[scale=0.2,align=c]{Branched3a.pdf}}
	            \end{overpic}
                \end{tikzcd} \hspace{2cm} \begin{tikzcd}
                \begin{overpic}[scale=0.3,align=c]{tssa1.png}
		        \put(-63,100){(}
		        \put(-49,100){\includegraphics[scale=0.2,align=c]{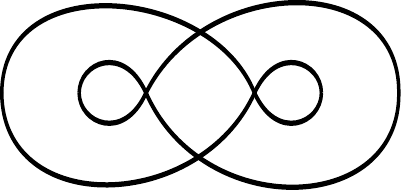}}
		        \put(12,100){,}

		        \put(26,100){\includegraphics[scale=0.13,align=c]{B5T2.pdf}}
		        \put(87,100){,}
		        \put(102,100){\includegraphics[scale=0.2,align=c]{Branched5e.pdf}}
		        \put(163,100){)}
		        \put(-15,-11){(}
		        \put(-5,-11){\includegraphics[scale=0.2,align=c]{Branched3a.pdf}}
		        \put(48,-11){,}
		        \put(58,-11){\includegraphics[scale=0.2,align=c]{Branched3b.pdf}}
		        \put(125,-11){)}
	            \end{overpic}
                \end{tikzcd} \hspace{1cm} & \\
                \multicolumn{5}{c}{} \\
                \hline
                \multicolumn{5}{c}{} \\
                \end{tabular}}
        };
    \end{tikzpicture}
    \caption{Boundary $N^{-}$ and $N^{+}$ of minimal GS isolating blocks  according to the associated  Lyapunov semi-graph.}
    \label{tab:bordos-colecao}
\end{table}

In Table  \ref{tab:bordos-colecao}, the semi-graphs that have edges labelled with more than one distinguished branched $1$-manifold indicate that any combination of choices represents the boundary components of some minimal GS isolating block corresponding to the given semi-graph.

Note that  whenever the singularity $p \in \mathcal{H}^{\mathcal{P}}_{\eta}$ is the $\omega$-limit of all the folds in $\mathcal{H}^{\mathcal{P}}_{\eta}$, the choice  that minimizes the branched charts in $N^{-}$ are circles. In the case $p \in \mathcal{H}^{\mathcal{T}}_{ssa}$, there are six folds in $\mathcal{H}^{\mathcal{T}}_{ssa}$, but only four of them have $p$ as their $\omega$-limit. Hence, the minimal number of branched charts in $N^-$ is two and thus, there are three possible choices for $N^{-}$ as shown in Table  \ref{tab:bordos-colecao}.

In the proof of  Theorem  \ref{teo:colecao}, the construction of a  minimal GS isolating block $N$ is obtained by considering all possibilities   of  distinguished branched $1$-manifolds  that are admissible as exiting sets, as well as, the gluing maps of a GS handle.  
Note that each choice determines a block $N$ with a  distinguished branched $1$-manifold $N^+$ as its entering set.   Theorem  \ref{teo:colecao} provides a complete classification of $(N,N^-,N^+)$.

\begin{proof}

Let  $p \in \mathcal{H}^{\mathcal{P}}_{a}$ (resp., $p \in \mathcal{H}^{\mathcal{P}}_{r}$), the unique minimal isolating block was obtained  in  Proposition \ref{proposicao}.

We now analyze GS singularities   with different saddle natures and of type  $\mathcal{R}, \mathcal{C}, \mathcal{W}, \mathcal{D}$ and $\mathcal{T}$.

\begin{itemize}
    \item[i) ] $p \in    \mathcal{H}^{\mathcal{R}}_{s}$
    
This case  is  well known  in  classical Handle Theory as the \textit{pair of pants}. The possible choices for $N^-$ are one or two circles.  This corresponds to  a semi-graph with $e_v^+ + e_v^- =3$, where $v$ is labelled with $(0,1,0)_{\mathcal{R}}$.  In the case of a semi-graph  with $e_v^+ + e_v^- =2$,  where $v$ is labelled with $(0,1,0)_{\mathcal{R}}$, the corresponding  isolating block $N$ is obtained by gluing a twisted handle on the circle $N^-$. The block $N$   is homotopy equivalent to the nonorientable block formed by a Mobius band minus a disc.

    \item[ii) ] $p \in   \mathcal{H}^{\mathcal{C}}_{s} $
    
     \begin{table}[!htb]
		\centering
		\resizebox{0.65\linewidth}{!}{
		\begin{tabular}{ccccccccc}
			 \cellcolor{gray!20} & \cellcolor{gray!20} Pre quotient & \cellcolor{gray!20} & \cellcolor{gray!20} & \cellcolor{gray!20} Local chart & \cellcolor{gray!20} & \cellcolor{gray!20} & \cellcolor{gray!20} Attaching region & \cellcolor{gray!20} \\
			 %& & & & & & & & & & & \\
			 \hline
			 & & & & & & & & \\
			 & \includegraphics[align=c,scale=0.6]{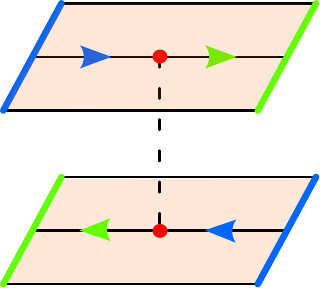} &  &  & \includegraphics[align=c,scale=0.6]{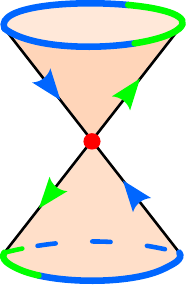} & & & \includegraphics[align=c,scale=0.6]{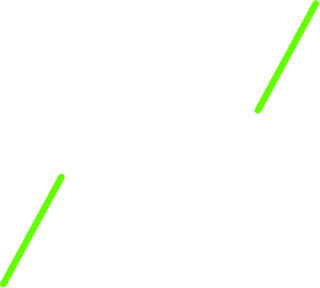}   & \\
			 &  &  &  &  &  & & & \\
			 %\hline
		\end{tabular}}
		%\mbox{ \ \ \ \ í­ndice do Ponto Regular}
		\caption{Representation of a cone handle $\mathcal{H}^{\mathcal{C}}_{s}$}
		\label{tabela:Hcone}
	\end{table}

		A cone handle $\mathcal{H}^{\mathcal{C}}_{s}$ corresponds to the gluing of two discs, each of which has a  tubular flow with a degenerate singularity at its center. Once the two centers are identified, each disc  corresponds to the upper and lower sheets of the cone.    The attaching region of the handle has  two connected components with no branched charts, one on each sheet of the cone, 
		see Table \ref{tabela:Hcone}.   Since there are no branched charts,  $N^-$ is either one or two circles. Hence, the two connected components which make up the attaching region can be glued to one circle, see left side of Table \ref{NEWconeblocks}  or glued to two circles, see right side of Table \ref{NEWconeblocks}.

	\begin{table}[!htb]
			\centering
			\resizebox{\linewidth}{!}{
			\begin{tabular}{ccccc|ccccc}
                \cellcolor{gray!20} Minimal isolating block & \cellcolor{gray!20} & \cellcolor{gray!20} & \cellcolor{gray!20} Lyapunov graph & \cellcolor{gray!20} & \cellcolor{gray!20} & \cellcolor{gray!20} Minimal isolating block & \cellcolor{gray!20} & \cellcolor{gray!20} & \cellcolor{gray!20} Lyapunov graph \\
                \hline
			    & & & & & & & & & \\
                \begin{overpic}[align=c,scale=0.28]{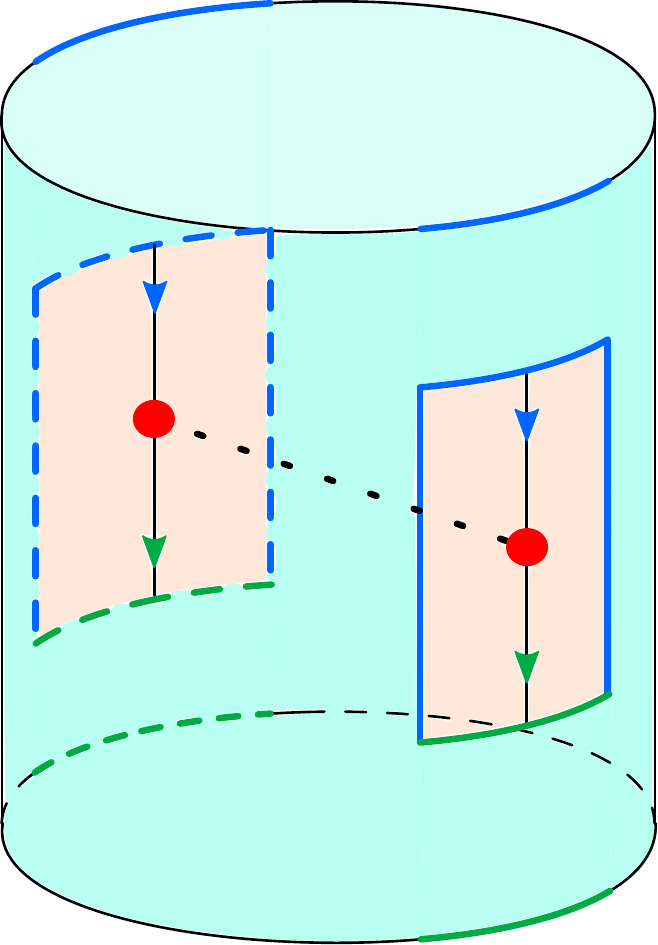}
                \put(78,45){$\rightarrow$}
                \end{overpic} \hspace{0.7cm} \begin{overpic}[align=c,scale=0.28]{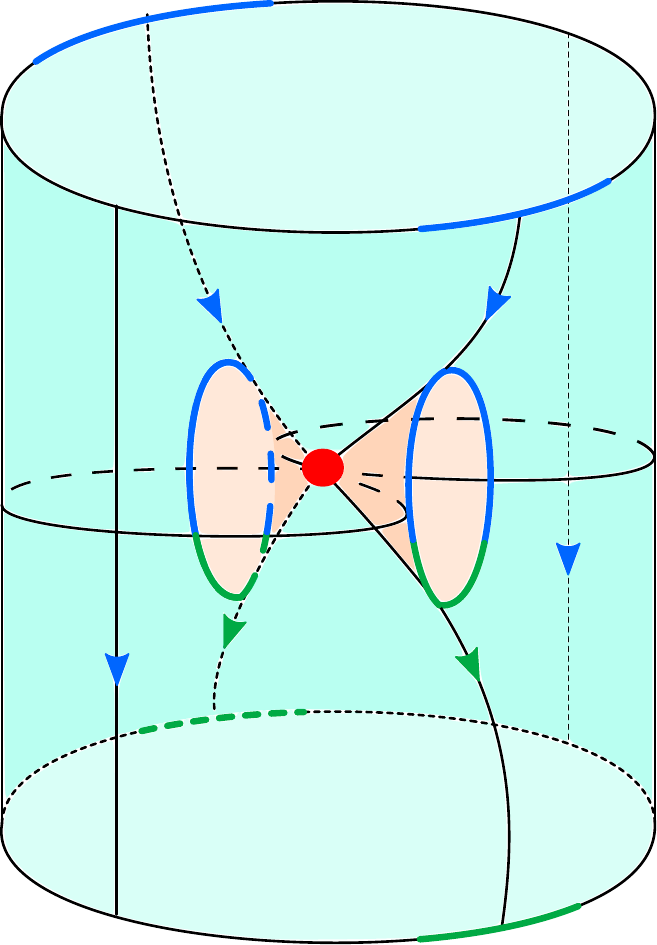}\end{overpic} & & & \begin{overpic}[align=c,scale=0.52]{cs1.png}
                \put(35,88){\includegraphics[scale=0.3,align=c]{Branched1.pdf}}
                \put(35,-8){\includegraphics[scale=0.3,align=c]{Branched1.pdf}}
                \end{overpic} & & & \begin{overpic}[align=c,scale=0.28]{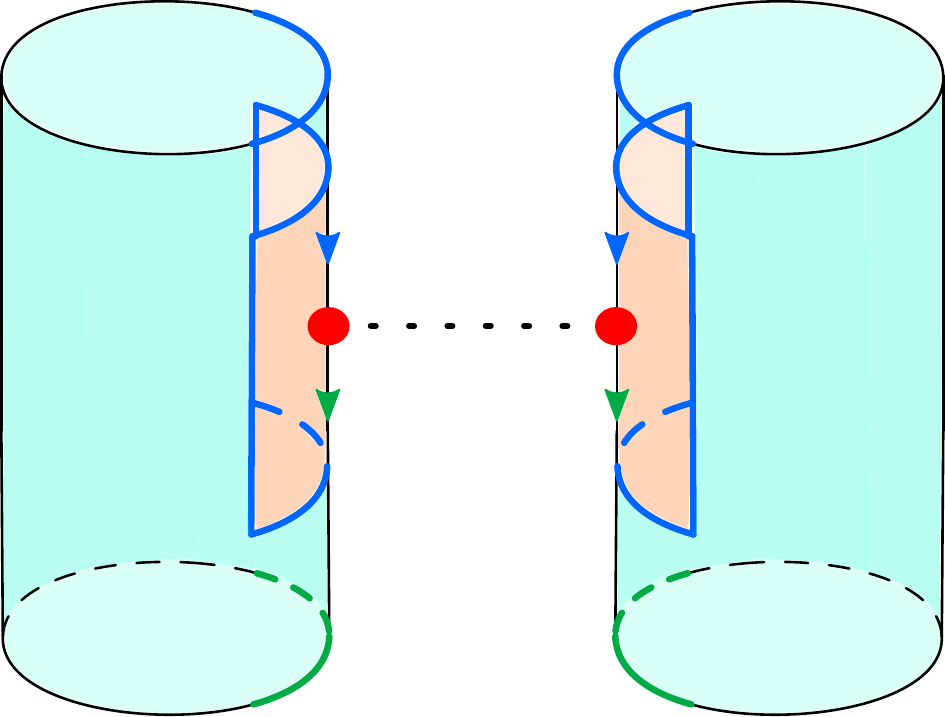}
                \put(110,38){$\rightarrow$}
                \end{overpic} \hspace{0.7cm} \begin{overpic}[align=c,scale=0.28]{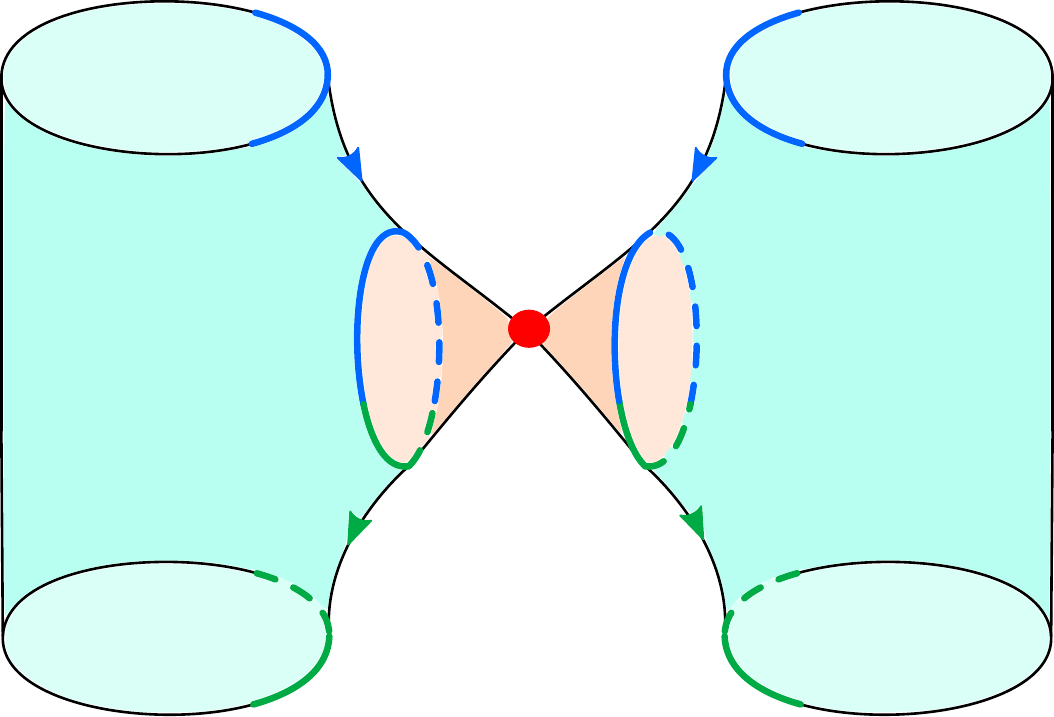}\end{overpic}  & & & \begin{overpic}[align=c,scale=0.52]{cs2.png}
                \put(65,88){\includegraphics[scale=0.3,align=c]{Branched1.pdf}}
                \put(5,88){\includegraphics[scale=0.3,align=c]{Branched1.pdf}}
                \put(5,-8){\includegraphics[scale=0.3,align=c]{Branched1.pdf}}
                \put(65,-8){\includegraphics[scale=0.3,align=c]{Branched1.pdf}}
                \end{overpic} \\
                & & & & & & & & & \\
                & & & & & & & & & \\
            \end{tabular}}
			\caption{Minimal isolating blocks for a singularity  $p \in \mathcal{H}^{\mathcal{C}}_{s}$}
			\label{NEWconeblocks}
		\end{table}

	\item[iii) ] $p \in \mathcal{H}^{\mathcal{W}}_{s_{s}}$
    
    \begin{table}[!htb]
		\centering
		\resizebox{0.65\linewidth}{!}{
		\begin{tabular}{ccccccccc}
			 \cellcolor{gray!20} & \cellcolor{gray!20} Pre quotient & \cellcolor{gray!20} & \cellcolor{gray!20} & \cellcolor{gray!20} Local chart & \cellcolor{gray!20} & \cellcolor{gray!20} & \cellcolor{gray!20} Attaching region & \cellcolor{gray!20} \\
			 %& & & & & & & & & & & \\
			 \hline
			 & & & & & & & & \\
			 & \includegraphics[align=c,scale=0.5]{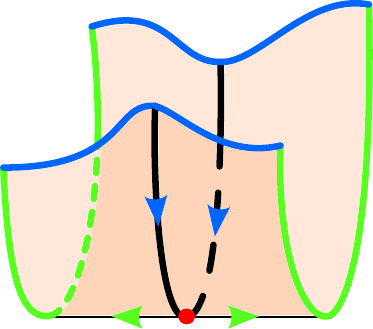} &  &  & \includegraphics[align=c,scale=0.5]{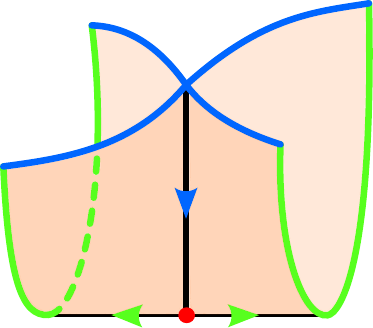} & & & \includegraphics[align=c,scale=0.5]{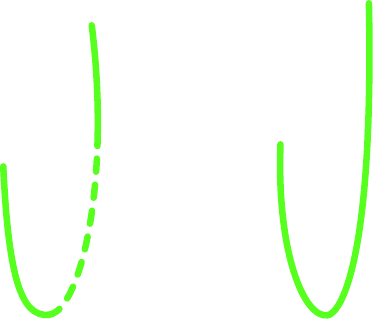} &   \\
			 &  &  &  &  & & & &  \\
		     %\hline
		\end{tabular}}
		%\mbox{ \ \ \ \ í­ndice do Ponto Regular}
		\caption{Representation of a  Whitney handle $\mathcal{H}^{\mathcal{W}}_{s_{s}}$}
		\label{tabela:Hwhitney}
	\end{table}

A Whitney handle $\mathcal{H}^{\mathcal{W}}_{s_{s}}$  corresponds to a regular handle $\mathcal{H}^{\mathcal{R}}_{s}$ followed by  the identification of the two stable orbits, see Table \ref{tabela:Hwhitney}.  Also in this case, the attaching region has no branched point  and hence,   $N^-$ is either one or two circles. 
Thus the two connected components which make up the attaching region can be glued to one circle, see left side of Table \ref{NEWwhitneyblocks}  or glued to two circles, see right side of Table \ref{NEWwhitneyblocks}.

	\begin{table}[htb]
			\centering
			\resizebox{\linewidth}{!}{
			\begin{tabular}{ccccc|ccccc}
                \cellcolor{gray!20} Minimal isolating block & \cellcolor{gray!20} & \cellcolor{gray!20} & \cellcolor{gray!20} Lyapunov graph & \cellcolor{gray!20} & \cellcolor{gray!20} & \cellcolor{gray!20} Minimal isolating block & \cellcolor{gray!20} & \cellcolor{gray!20} & \cellcolor{gray!20} Lyapunov graph \\
                \hline
			    & & & & & & & & & \\
                \begin{overpic}[align=c,scale=0.42]{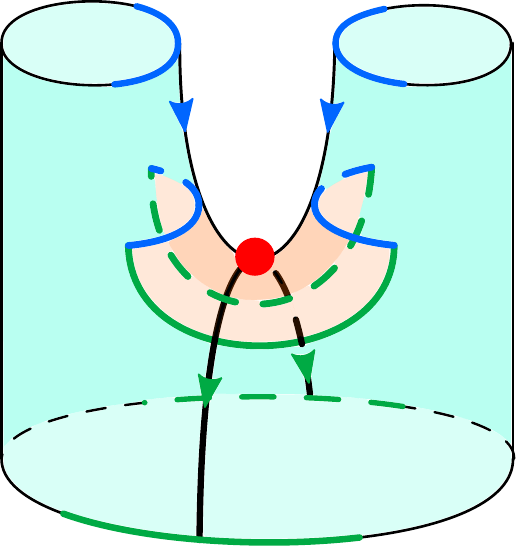}
                \put(105,45){$\rightarrow$}
                \end{overpic} \hspace{0.7cm} \begin{overpic}[align=c,scale=0.42]{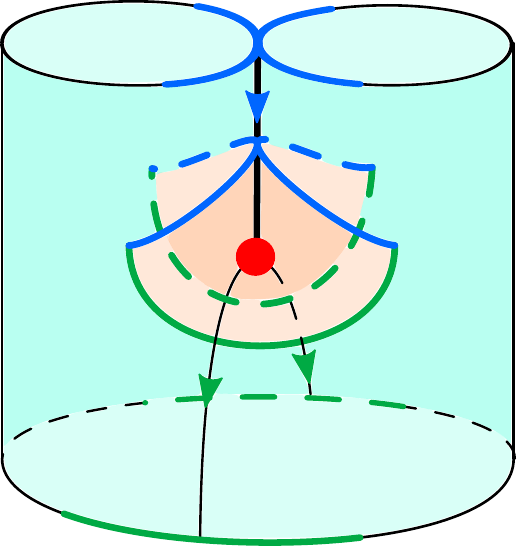}\end{overpic} & & &  \begin{overpic}[align=c,scale=0.5]{wss1.png} \put(30,85){\includegraphics[scale=0.3,align=c]{Branched2.pdf}} \put(38,-5){\includegraphics[scale=0.3,align=c]{Branched1.pdf}}
                \end{overpic} & & & \begin{overpic}[align=c,scale=0.42]{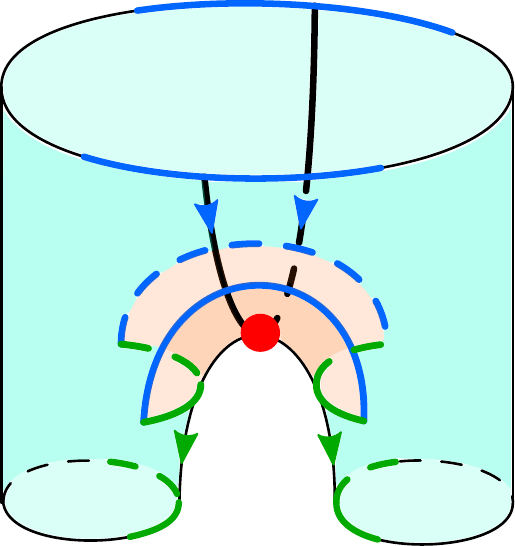}
                \put(105,45){$\rightarrow$}
                \end{overpic} \hspace{0.7cm} \begin{overpic}[align=c,scale=0.42]{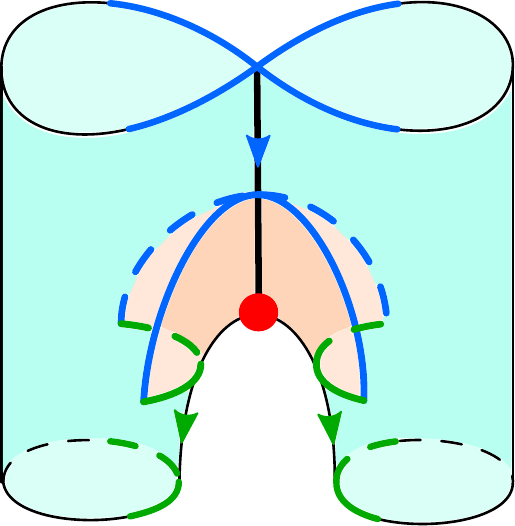}\end{overpic}  & & & \begin{overpic}[align=c,scale=0.52]{wss2.png}
                \put(30,85){\includegraphics[scale=0.3,align=c]{Branched2.pdf}}
                \put(5,-5){\includegraphics[scale=0.3,align=c]{Branched1.pdf}}
                \put(65,-5){\includegraphics[scale=0.3,align=c]{Branched1.pdf}}
                \end{overpic} \\
                
                & & & & & & & & & \\
            \end{tabular}}
			\caption{Minimal isolating blocks for a singularity  $p \in \mathcal{H}^{\mathcal{W}}_{s_{s}}$}
			\label{NEWwhitneyblocks}
		\end{table}

    \item[iv) ] $p \in \mathcal{H}^{\mathcal{D}}_{\eta}$

    \item[a) ]  A double handle $\mathcal{H}^{\mathcal{D}}_{sa}$ corresponds to two regular handles $\mathcal{H}^{\mathcal{R}}_{s}$ and  $\mathcal{H}^{\mathcal{R}}_{a}$ in which two pairs of stable orbits are identified, as shown in Table \ref{tabela:HDuplosa}.
Consequently, by considering the minimal  isolating blocks for the  regular  saddle and the attractor, and then identifying the respective stable pair of orbits, a minimal isolating block for a singularity of type $\mathcal{D}$ and nature $sa$ is produced. See Table 
     \ref{doubleSAblocks} for the case that the isolating blocks are orientable. For the nonorientable case, the schematic construction is shown in Table  \ref{doubleSAblocks2}.

    \begin{table}[!htb]
		\centering
		\resizebox{0.65\linewidth}{!}{
		\begin{tabular}{ccccccccc}
			 \cellcolor{gray!20} & \cellcolor{gray!20} Pre quotient & \cellcolor{gray!20} & \cellcolor{gray!20} & \cellcolor{gray!20} Local chart & \cellcolor{gray!20} & \cellcolor{gray!20} & \cellcolor{gray!20} Attaching region & \cellcolor{gray!20} \\
			 %& & & & & & & & & & & \\
			 \hline
			 & & & & & & & & \\
			 & \includegraphics[align=c,scale=0.7]{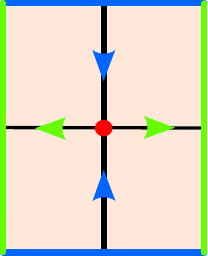} \hspace{0.5cm} \includegraphics[align=c,scale=0.7]{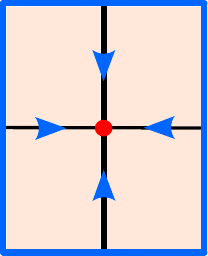} &  &  & \includegraphics[align=c,scale=0.6]{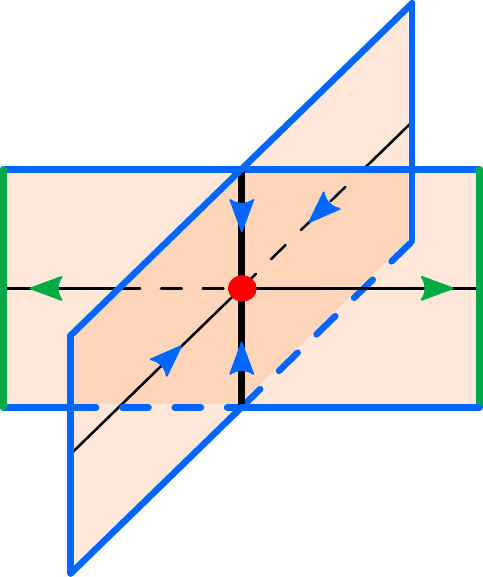} & & & \includegraphics[align=c,scale=0.7]{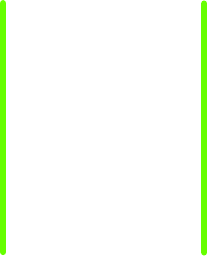} &   \\
			 &  &  &  &  & & & &  \\
		     %\hline
		\end{tabular}}
		%\mbox{ \ \ \ \ í­ndice do Ponto Regular}
		\caption{Representation of a double handle $\mathcal{H}^{\mathcal{D}}_{sa}$}
		\label{tabela:HDuplosa}
	\end{table}

\begin{table}[!htb]
			\centering
			\resizebox{0.95\linewidth}{!}{
			\begin{tabular}{cc|ccc}
                \cellcolor{gray!20} Minimal isolating block & \cellcolor{gray!20} & \cellcolor{gray!20} & \cellcolor{gray!20} Lyapunov graph & \cellcolor{gray!20}  \\
                \hline
			    & & & &  \\
                \begin{overpic}[align=c,scale=0.4]{Paints2.pdf}\end{overpic} \hspace{0.5cm} \begin{overpic}[align=c,scale=0.32]{blocoRA.pdf}
             \put(130,45){$\rightarrow$}
             \end{overpic} \hspace{1.3cm}  \begin{overpic}[align=c,scale=0.35]{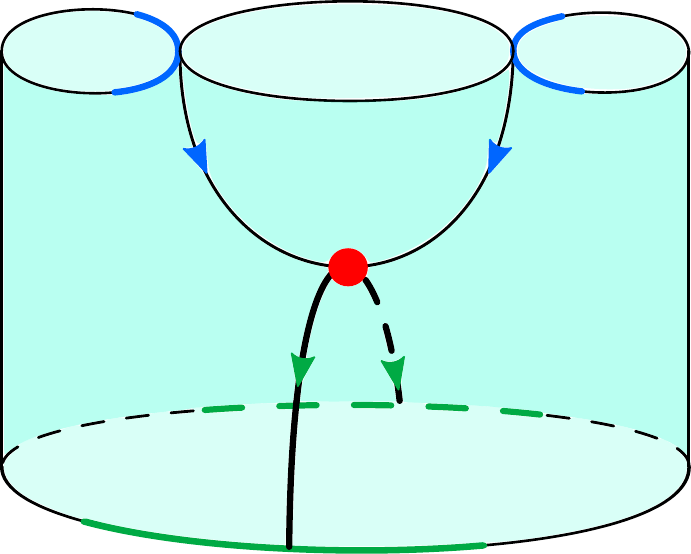}\end{overpic} & & & \begin{overpic}[align=c,scale=0.5]{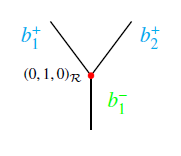}
             \put(5,85){\includegraphics[scale=0.3,align=c]{Branched1.pdf}}
             \put(62,85){\includegraphics[scale=0.3,align=c]{Branched1.pdf}}
             \put(35,-5){\includegraphics[scale=0.3,align=c]{Branched1.pdf}}
             \end{overpic} \begin{overpic}[align=c,scale=0.5]{ra.png}
             \put(33,88){\includegraphics[scale=0.3,align=c]{Branched1.pdf}}
             \put(105,40){$\rightarrow$}
             \end{overpic} \hspace{0.9cm} \begin{overpic}[align=c,scale=0.5]{dsa1.png}
             \put(25,85){\includegraphics[scale=0.3,align=c]{Branched3b.pdf}}
             \put(37,-5){\includegraphics[scale=0.3,align=c]{Branched1.pdf}}
             \end{overpic} & \\
                & & & & \\
                \hline
                & & & & \\
             \begin{overpic}[align=c,scale=0.4]{Paints.pdf}\end{overpic} \hspace{0.5cm} \begin{overpic}[align=c,scale=0.32]{blocoRA.pdf}
             \put(130,45){$\rightarrow$}
             \end{overpic} \hspace{1.3cm}  \begin{overpic}[align=c,scale=0.42]{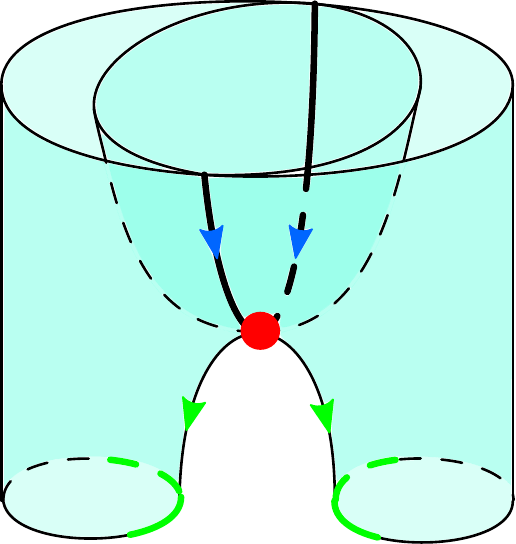}\end{overpic} & &   & \begin{overpic}[align=c,scale=0.5]{rs2.png}
             \put(35,88){\includegraphics[scale=0.3,align=c]{Branched1.pdf}}
             \put(2,-5){\includegraphics[scale=0.3,align=c]{Branched1.pdf}}
             \put(64,-5){\includegraphics[scale=0.3,align=c]{Branched1.pdf}}
             \end{overpic} \begin{overpic}[align=c,scale=0.5]{ra.png}
             \put(33,88){\includegraphics[scale=0.3,align=c]{Branched1.pdf}}
             \put(105,40){$\rightarrow$}
             \end{overpic} \hspace{0.9cm} \begin{overpic}[align=c,scale=0.5]{dsa2.png}
             \put(29,88){\includegraphics[scale=0.3,align=c]{Branched3a.pdf}}
             \put(2,-5){\includegraphics[scale=0.3,align=c]{Branched1.pdf}}
             \put(64,-5){\includegraphics[scale=0.3,align=c]{Branched1.pdf}}
             \end{overpic} & \\
                & & & & \\
                & & & & \\
            \end{tabular}}
			\caption{Minimal isolating blocks for a singularity  $p \in \mathcal{H}^{\mathcal{D}}_{sa}$}
			\label{doubleSAblocks}
	\end{table}

\end{itemize}

\begin{table}[!htb]
			\centering
			\resizebox{0.45\linewidth}{!}{
			\begin{tabular}{c}
			   \cellcolor{gray!20} Lyapunov graph moves \\
                \hline
			    \\
			    \\
			    \begin{overpic}[scale=0.5,align=c]{rs1.png}
		        \put(37,86){\includegraphics[scale=0.3,align=c]{Branched1.pdf}}
		        \put(37,-9){\includegraphics[scale=0.3,align=c]{Branched1.pdf}}
	            \end{overpic} \begin{overpic}[align=c,scale=0.5]{ra.png}
             \put(33,88){\includegraphics[scale=0.3,align=c]{Branched1.pdf}}
             \put(105,40){$\rightarrow$}
             \end{overpic} \hspace{0.9cm} \begin{overpic}[align=c,scale=0.5]{dsa1.png}
             \put(28,85){\includegraphics[scale=0.3,align=c]{Branched3a.pdf}}
             \put(37,-5){\includegraphics[scale=0.3,align=c]{Branched1.pdf}}
             \end{overpic} \\
                 \\
             \end{tabular}}
			\caption{Minimal isolating blocks for a singularity $p \in \mathcal{H}^{\mathcal{D}}_{sa}$}
			\label{doubleSAblocks2}
\end{table}

\begin{itemize}

	\item[b) ] Analogously, a double handle $\mathcal{H}^{\mathcal{D}}_{ss_{s}}$ 
	 corresponds to two regular handles $\mathcal{H}^{\mathcal{R}}_{s}$ in which two pairs of stable orbits are identified. The attaching region has four connected components with no branched charts, hence $N^-$ is either one, two, three or four circles.
	 	In this case, by analyzing the distinct possibilities of gluing the attaching region to the different cases of $N^-$, the two saddles prior to the identification are either  in one minimal isolating block or in two disjoint minimal isolating blocks.
	 	   The isolating block for the  singularity $p$ of nature $ss_s$ is produced by identifying two pairs of stable orbits of the two regular saddles. See Table \ref{doubleSSsblocks} for the case that the isolating blocks are orientable.

\end{itemize}

\begin{table}[!htb]
		\centering
		\resizebox{0.65\linewidth}{!}{
		\begin{tabular}{ccccccccc}
			 \cellcolor{gray!20} & \cellcolor{gray!20} Pre quotient & \cellcolor{gray!20} & \cellcolor{gray!20} & \cellcolor{gray!20} Local chart & \cellcolor{gray!20} & \cellcolor{gray!20} & \cellcolor{gray!20} Attaching region & \cellcolor{gray!20} \\
			 %& & & & & & & & & & & \\
			 \hline
			 & & & & & & & & \\
			 & \includegraphics[align=c,scale=0.7]{Plano1.pdf} \hspace{0.5cm} \includegraphics[align=c,scale=0.7]{Plano1.pdf} &  &  & \includegraphics[align=c,scale=0.6]{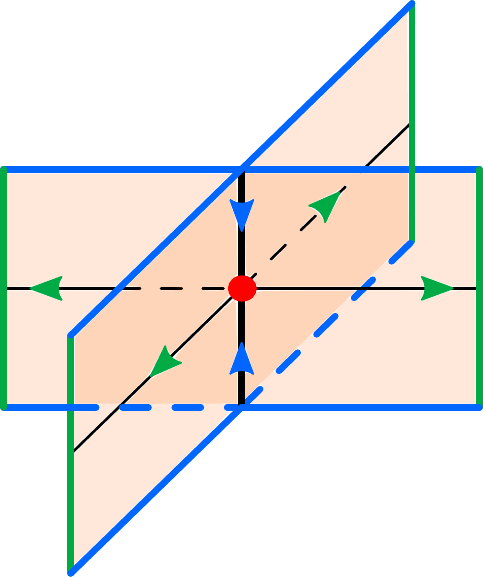} & & & \includegraphics[align=c,scale=0.5]{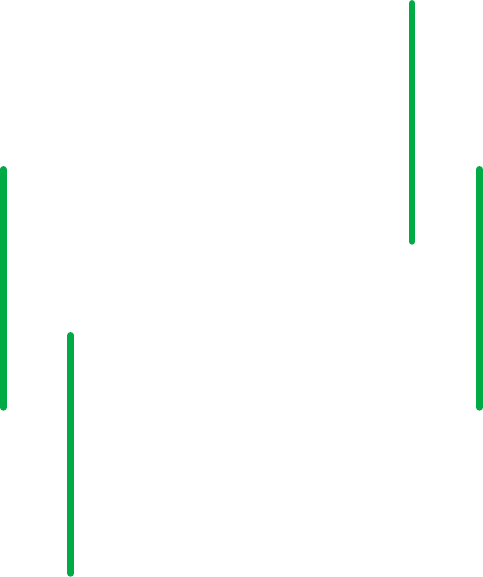} &   \\
			 &  &  &  &  & & & &  \\
		     %\hline
		\end{tabular}}
		%\mbox{ \ \ \ \ í­ndice do Ponto Regular}
		\caption{Representation of a double handle $\mathcal{H}^{\mathcal{D}}_{ss_{s}}$}
		\label{tabela:HDuploss}
	\end{table}

%%%%%%%%%%%%%%%%%%%%%%%%%%%%%%%%%%%%%%%%%%%%%%%%%%%%%%%%%%%%%%%
%%%%%%%%%%%%%%%%%%%%%%%%%%%%%%%%%%%%%%%%%%%%%%%%%%%%%%%%%%%%%%%
%%%%%%%%%%%%%%%%%%%%%%%%%%%%%%%%%%%%%%%%%%%%%%%%%%%%%%%%%%%%%%%

\begin{table}[htb]
			\centering
			\resizebox{0.95\linewidth}{!}{
			\begin{tabular}{ccc|cc}
                \cellcolor{gray!20} Minimal isolating block & \cellcolor{gray!20} & \cellcolor{gray!20} & \cellcolor{gray!20} Lyapunov graph moves & \cellcolor{gray!20}  \\
                \hline
                & & & &  \\
                \begin{overpic}[align=c,scale=0.38]{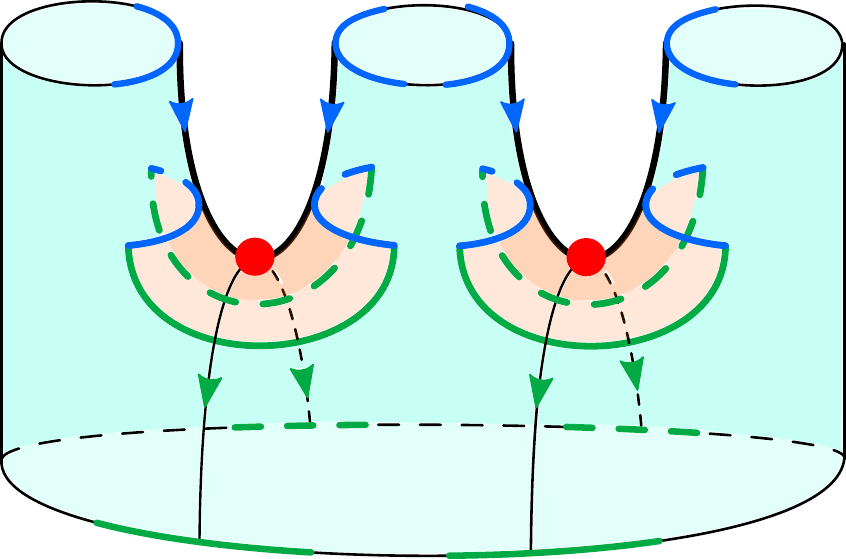}
                \put(116,32){$\rightarrow$}
                \end{overpic} \hspace{1.1cm}  \begin{overpic}[align=c,scale=0.22]{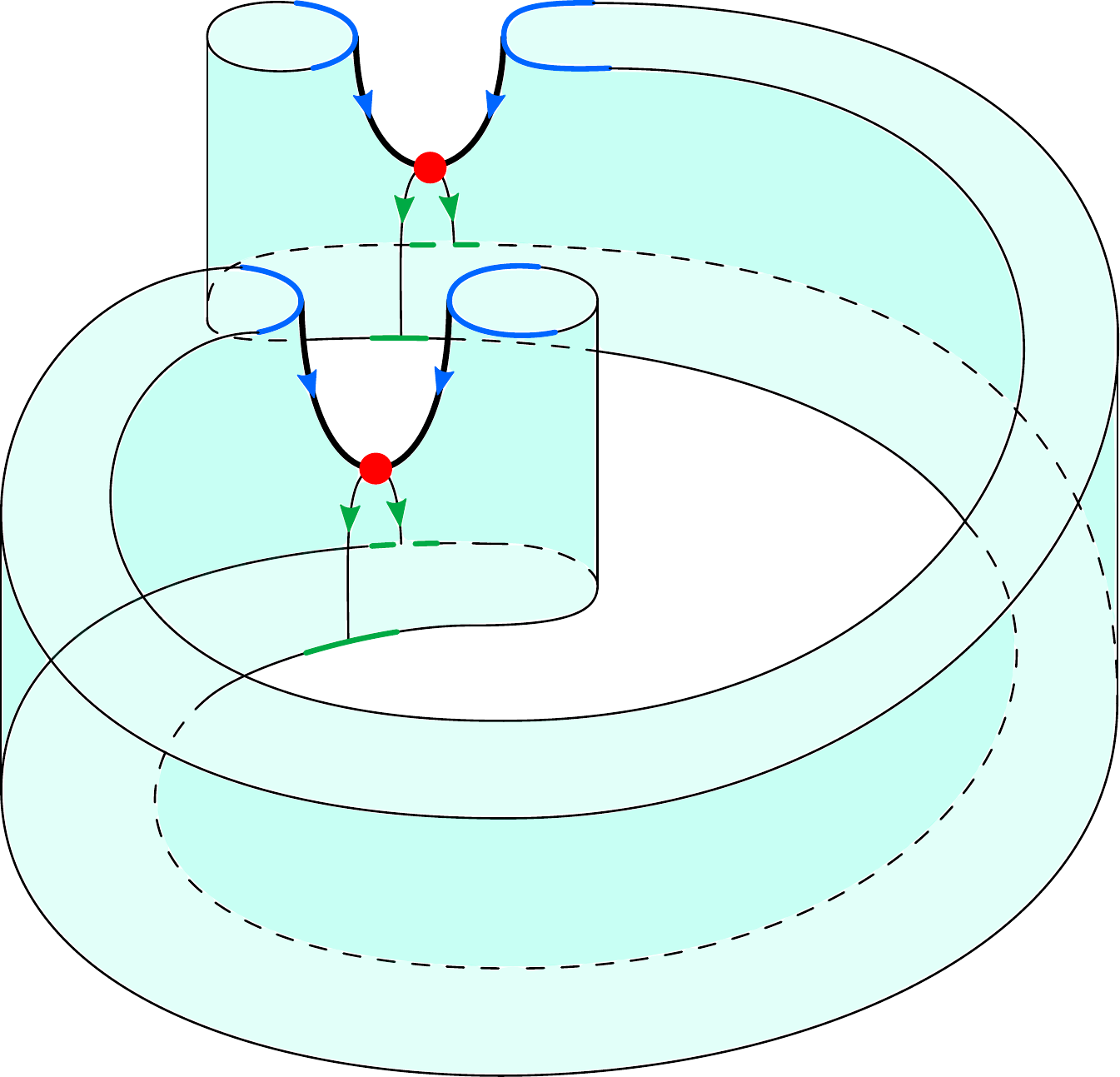}
                \put(117,34){$\rightarrow$}
                \end{overpic} \hspace{1.2cm} \begin{overpic}[align=c,scale=0.27]{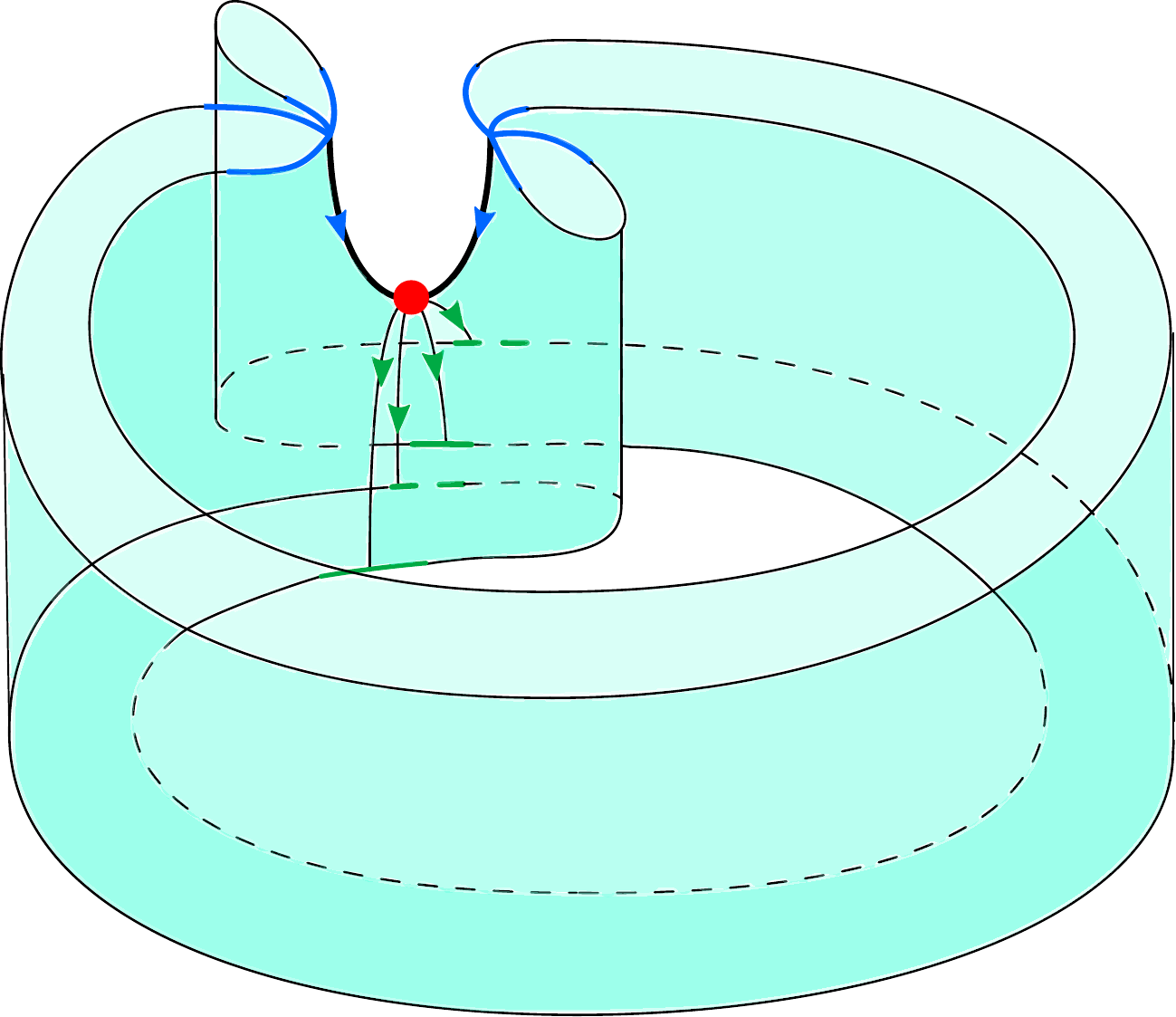}\end{overpic}  & & & \begin{overpic}[align=c,scale=0.5]{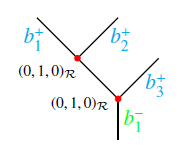}
                \put(2,85){\includegraphics[align=c,scale=0.3]{Branched1.pdf}}
                \put(56,85){\includegraphics[align=c,scale=0.3]{Branched1.pdf}}
                \put(80,62){\includegraphics[align=c,scale=0.3]{Branched1.pdf}}
                \put(50,-8){\includegraphics[align=c,scale=0.3]{Branched1.pdf}}
                \put(115,40){$\rightarrow$}
                \end{overpic} \hspace{0.9cm} \begin{overpic}[align=c,scale=0.5]{dsss1.png}
                \put(36,-8){\includegraphics[align=c,scale=0.3]{Branched1.pdf}}
                \put(25,86){\includegraphics[align=c,scale=0.3]{Branched3b.pdf}}
                \end{overpic} &  \\
                %& & & $\mathcal{B}^{+} = \mathcal{B}^{-} + 2$ &  \\
                & & & &  \\
                \hline
                & & & &  \\
                \begin{overpic}[align=c,scale=0.38]{DSSs1a.pdf}
                \put(116,32){$\rightarrow$}
                \end{overpic} \hspace{1.2cm}  \begin{overpic}[align=c,scale=0.3]{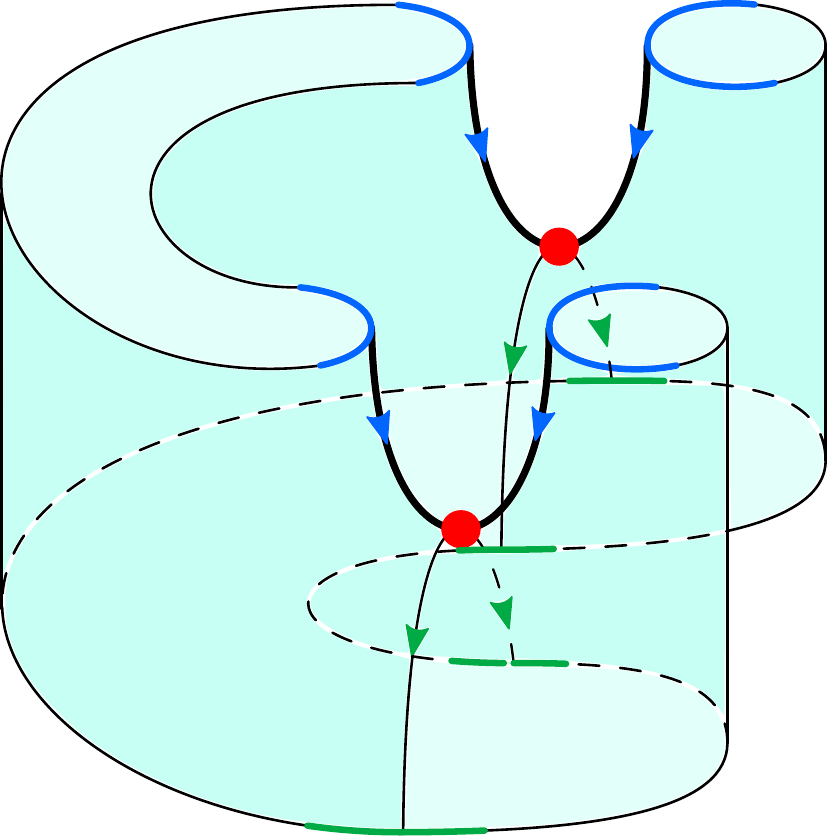}
                \put(120,38){$\rightarrow$}
                \end{overpic} \hspace{1.2cm} \begin{overpic}[align=c,scale=0.3]{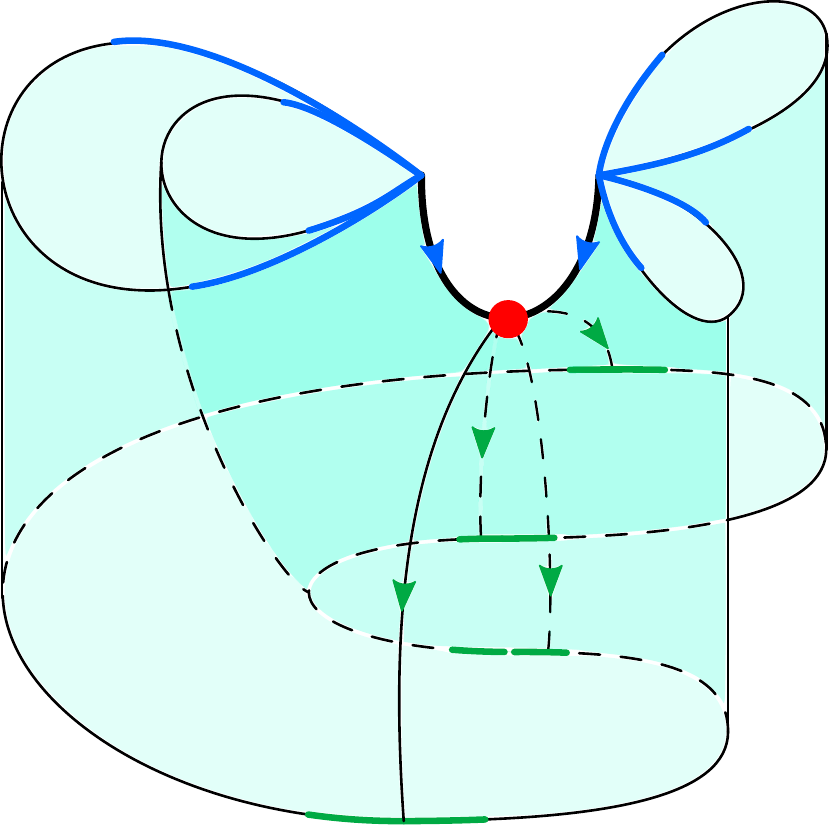}\end{overpic}  & & & \begin{overpic}[align=c,scale=0.5]{rss2.png}
                \put(2,85){\includegraphics[align=c,scale=0.3]{Branched1.pdf}}
                \put(56,85){\includegraphics[align=c,scale=0.3]{Branched1.pdf}}
                \put(80,62){\includegraphics[align=c,scale=0.3]{Branched1.pdf}}
                \put(50,-8){\includegraphics[align=c,scale=0.3]{Branched1.pdf}}
                \put(115,40){$\rightarrow$}
                \end{overpic} \hspace{0.9cm} \begin{overpic}[align=c,scale=0.5]{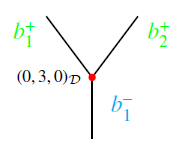}
                \put(2,86){\includegraphics[align=c,scale=0.3]{Branched2.pdf}}
                \put(60,86){\includegraphics[align=c,scale=0.3]{Branched2.pdf}}
                \put(38,-8){\includegraphics[align=c,scale=0.3]{Branched1.pdf}}
                \end{overpic} &  \\
                %& & & $\mathcal{B}^{+} = \mathcal{B}^{-} + 3$  &  \\
                & & & &  \\
                \hline
                & & & &  \\
                \begin{overpic}[align=c,scale=0.38]{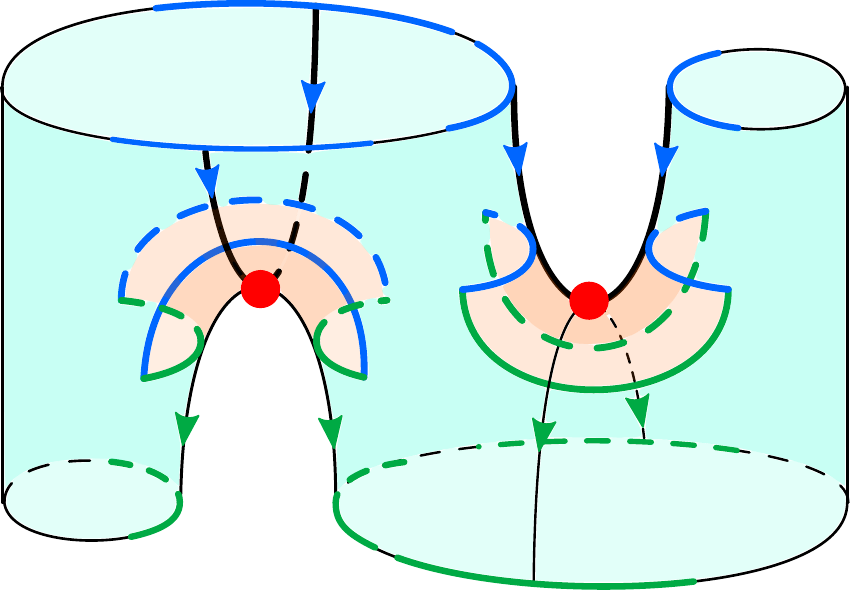}
                \put(115,35){$\rightarrow$}
                \end{overpic} \hspace{1.2cm} \begin{overpic}[align=c,scale=0.22]{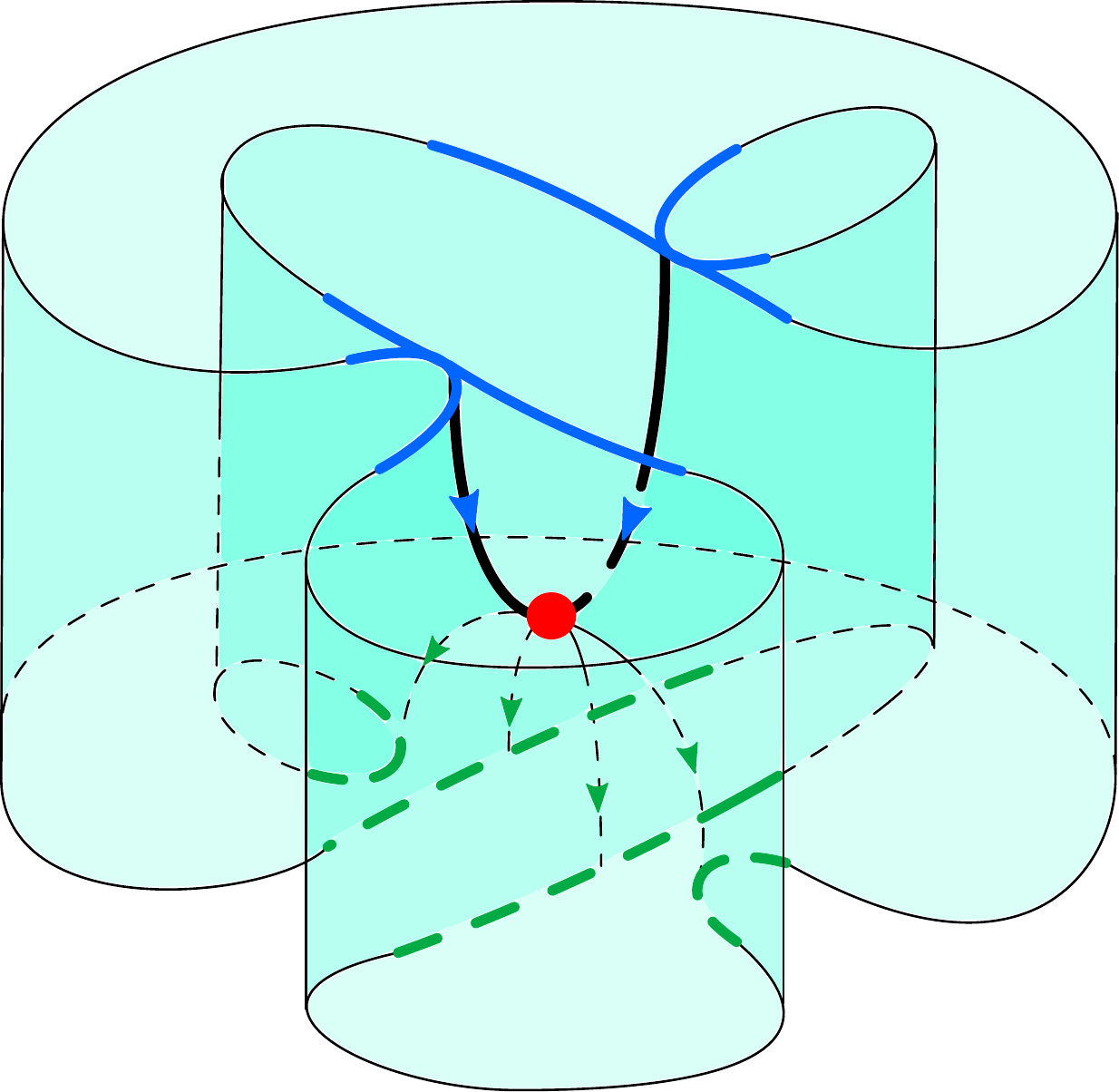}\end{overpic} & & & \begin{overpic}[align=c,scale=0.5]{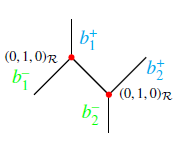}
                \put(25,86){\includegraphics[align=c,scale=0.3]{Branched1.pdf}}
                \put(67,65){\includegraphics[align=c,scale=0.3]{Branched1.pdf}}
                \put(2,17){\includegraphics[align=c,scale=0.3]{Branched1.pdf}}
                \put(45,-6){\includegraphics[align=c,scale=0.3]{Branched1.pdf}}
                \put(115,40){$\rightarrow$}
                \end{overpic} \hspace{0.9cm} \begin{overpic}[align=c,scale=0.5]{dsss3.png}
                \put(24,86){\includegraphics[align=c,scale=0.3]{Branched3b.pdf}}
                \put(5,-6){\includegraphics[align=c,scale=0.3]{Branched1.pdf}}
                \put(65,-6){\includegraphics[align=c,scale=0.3]{Branched1.pdf}}
                \end{overpic} &  \\
                & & & &  \\
                \hline
                & & & &  \\
                \begin{overpic}[align=c,scale=0.38]{Paints2.pdf}\end{overpic} \hspace{0.5cm} \begin{overpic}[align=c,scale=0.38]{Paints2.pdf}
                \put(118,40){$\rightarrow$}
                \end{overpic} \hspace{1.2cm} \begin{overpic}[align=c,scale=0.36]{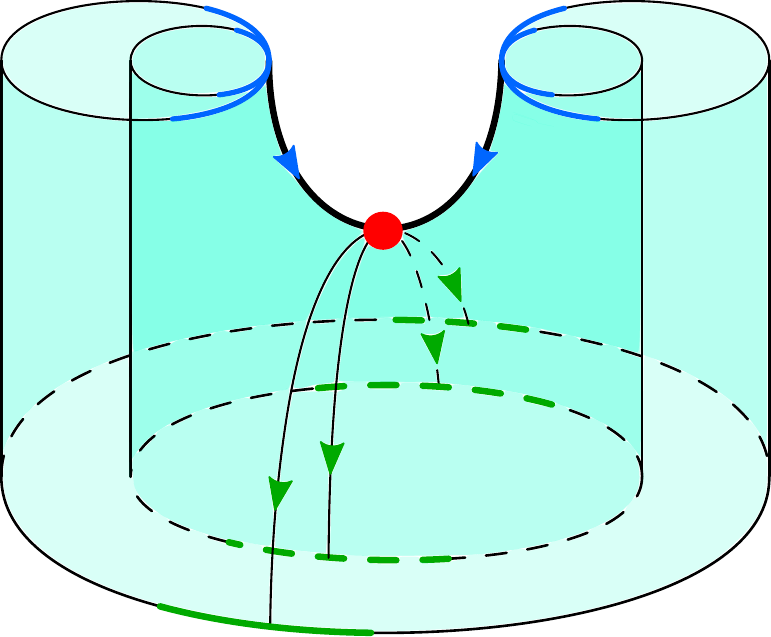}\end{overpic} \hspace{0.4cm}  & & & \begin{overpic}[align=c,scale=0.5]{rs3.png}
                \put(5,85){\includegraphics[scale=0.3,align=c]{Branched1.pdf}}
             \put(62,85){\includegraphics[scale=0.3,align=c]{Branched1.pdf}}
             \put(35,-5){\includegraphics[scale=0.3,align=c]{Branched1.pdf}}
                \end{overpic} \begin{overpic}[align=c,scale=0.5]{rs3.png}
                \put(5,85){\includegraphics[scale=0.3,align=c]{Branched1.pdf}}
             \put(62,85){\includegraphics[scale=0.3,align=c]{Branched1.pdf}}
             \put(35,-5){\includegraphics[scale=0.3,align=c]{Branched1.pdf}}
             \put(100,40){$\rightarrow$}
                \end{overpic} \hspace{0.8cm} \begin{overpic}[align=c,scale=0.5]{dsss4.png}
                \put(5,-8){\includegraphics[scale=0.3,align=c]{Branched1.pdf}}
             \put(65,-8){\includegraphics[scale=0.3,align=c]{Branched1.pdf}}
             \put(0,86){\includegraphics[scale=0.3,align=c]{Branched2.pdf}}
             \put(60,86){\includegraphics[scale=0.3,align=c]{Branched2.pdf}}
                \end{overpic} &  \\
                & & & &  \\
                \hline
                & & & &  \\
                \begin{overpic}[align=c,scale=0.38]{Paints2.pdf}\end{overpic} \hspace{0.5cm} \begin{overpic}[align=c,scale=0.38]{Paints.pdf}
                \put(120,38){$\rightarrow$}\end{overpic} \hspace{1.2cm}  \begin{overpic}[align=c,scale=0.25]{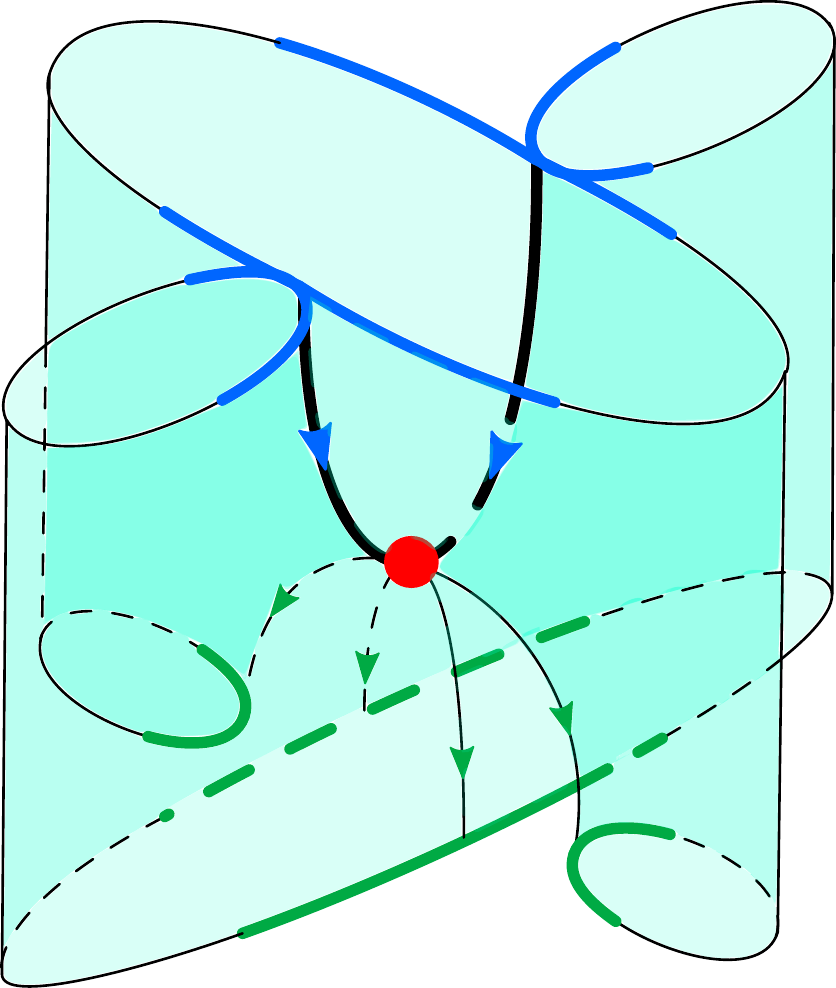}\end{overpic} & & & \begin{overpic}[align=c,scale=0.5]{rs3.png}
                \put(5,85){\includegraphics[scale=0.3,align=c]{Branched1.pdf}}
             \put(62,85){\includegraphics[scale=0.3,align=c]{Branched1.pdf}}
             \put(35,-5){\includegraphics[scale=0.3,align=c]{Branched1.pdf}}
                \end{overpic} \begin{overpic}[align=c,scale=0.5]{rs2.png}
                \put(35,85){\includegraphics[align=c,scale=0.3]{Branched1.pdf}}
                \put(5,-5){\includegraphics[align=c,scale=0.3]{Branched1.pdf}}
                \put(65,-5){\includegraphics[align=c,scale=0.3]{Branched1.pdf}}
                \put(100,38){$\rightarrow$}
                \end{overpic} \hspace{0.8cm} \begin{overpic}[align=c,scale=0.5]{dsss5.png}
		        \put(25,86){\includegraphics[scale=0.3,align=c]{Branched3b.pdf}}
		        \put(38,-11){\includegraphics[scale=0.3,align=c]{Branched1.pdf}}
		        \put(0,0){\includegraphics[scale=0.3,align=c]{Branched1.pdf}}
		        \put(72,0){\includegraphics[scale=0.3,align=c]{Branched1.pdf}}
                \end{overpic} &  \\
                %& & & $\mathcal{B}^{+} = \mathcal{B}^{-}$  &  \\
                & & & &  \\
                \hline
                & & & &  \\
                \begin{overpic}[align=c,scale=0.38]{Paints.pdf}\end{overpic} \hspace{0.5cm} \begin{overpic}[align=c,scale=0.38]{Paints.pdf}
                \put(120,38){$\rightarrow$} \end{overpic} \hspace{1.2cm}  \begin{overpic}[align=c,scale=0.36]{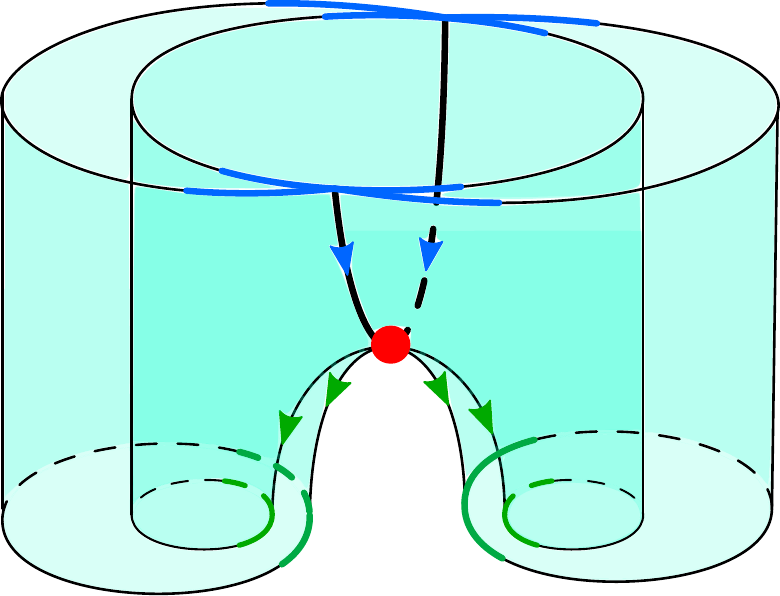}\end{overpic} & & & \begin{overpic}[align=c,scale=0.5]{rs2.png}
                \put(35,85){\includegraphics[align=c,scale=0.3]{Branched1.pdf}}
                \put(5,-5){\includegraphics[align=c,scale=0.3]{Branched1.pdf}}
                \put(65,-5){\includegraphics[align=c,scale=0.3]{Branched1.pdf}}
                \end{overpic} \hspace{0.4cm} \begin{overpic}[align=c,scale=0.5]{rs2.png}
                \put(95,38){$\rightarrow$}
                \put(35,85){\includegraphics[align=c,scale=0.3]{Branched1.pdf}}
                \put(5,-5){\includegraphics[align=c,scale=0.3]{Branched1.pdf}}
                \put(65,-5){\includegraphics[align=c,scale=0.3]{Branched1.pdf}}
                \end{overpic} \hspace{0.8cm} \begin{overpic}[scale=0.5,align=c]{dsss7.png}
		        \put(28,86){\includegraphics[scale=0.3,align=c]{Branched3a.pdf}}
		        \put(-15,5){\includegraphics[scale=0.3,align=c]{Branched1.pdf}}
		        \put(12,-11){\includegraphics[scale=0.3,align=c]{Branched1.pdf}}
		        \put(58,-11){\includegraphics[scale=0.3,align=c]{Branched1.pdf}}
		        \put(90,5){\includegraphics[scale=0.3,align=c]{Branched1.pdf}}
	            \end{overpic} & \\
                %& & & $\mathcal{B}^{+} = \mathcal{B}^{-} - 1$  &  \\
                & & & &  \\
                \end{tabular}}
			\caption{Minimal isolating blocks for a singularity  $p \in \mathcal{H}^{\mathcal{D}}_{ss_{s}}$}
			\label{doubleSSsblocks}
	\end{table}

%\pagebreak 

\begin{itemize}
    \item[v) ] $p \in \mathcal{H}^{\mathcal{T}}_{\eta}$ 
    
    A handle 
 $\mathcal{H}^{\mathcal{T}}_{ssa}$ corresponds to two regular handles  $\mathcal{H}^{\mathcal{R}}_{s}$ plus one regular handle  $\mathcal{H}^{\mathcal{R}}_{a}$, in which six pairs of stable orbits are identified.  Equivalently, one can consider a double handle  $\mathcal{H}^{\mathcal{D}}_{sa}$ plus a regular handle $\mathcal{H}^{\mathcal{R}}_{s}$,  or a double handle $\mathcal{H}^{\mathcal{D}}_{ss_{u}}$ plus a regular handle $\mathcal{H}^{\mathcal{R}}_{a}$, followed by the identification of four pairs of stable orbits. See Tables  \ref{tabela:Htriplo} and \ref{tabela:Htriplo2}.

\begin{table}[htb]
		\centering
		\resizebox{0.85\linewidth}{!}{
		\begin{tabular}{ccccccccc}
			 \cellcolor{gray!20} & \cellcolor{gray!20} Pre quotient  & \cellcolor{gray!20} & \cellcolor{gray!20} & \cellcolor{gray!20} Local chart & \cellcolor{gray!20} & \cellcolor{gray!20} & \cellcolor{gray!20} Attaching region & \cellcolor{gray!20} \\
			 %& & & & & & & & & & & \\
			 \hline
			 & & & & & & & & \\
			 & \includegraphics[align=c,scale=0.55]{AlcaD.pdf} \hspace{0.8cm} \includegraphics[align=c,scale=0.45]{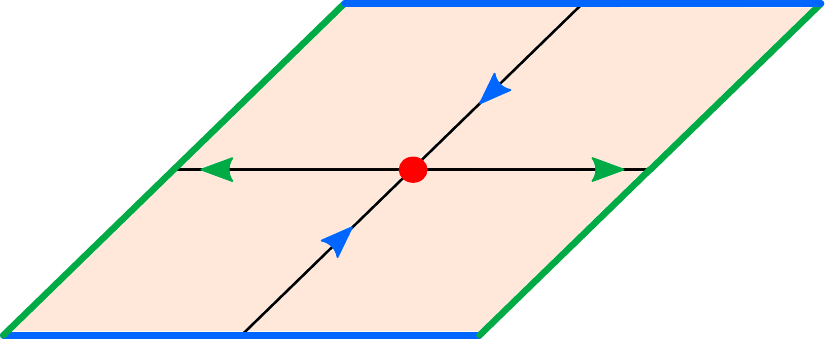} &  &  & \includegraphics[align=c,scale=0.55]{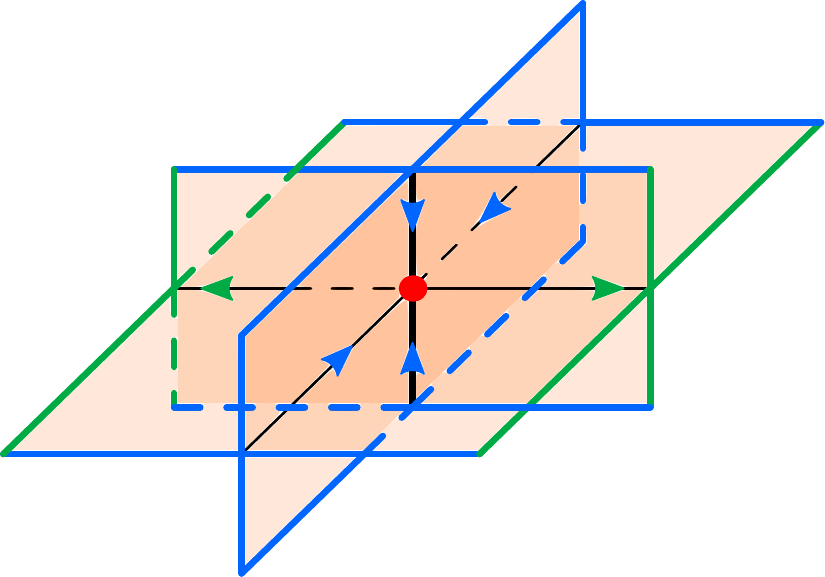} & & & \includegraphics[align=c,scale=0.5]{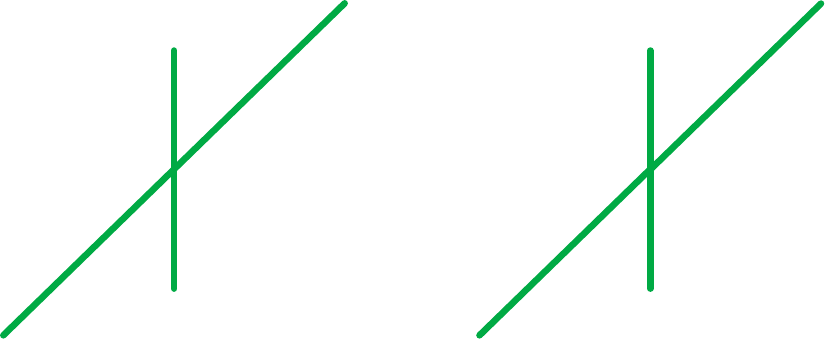} &   \\
			 &  &  &  &  & & & &  \\
		     %\hline
		\end{tabular}}
		%\mbox{ \ \ \ \ í­ndice do Ponto Regular}
		\caption{Representation of a triple handle $\mathcal{H}^{\mathcal{T}}_{ssa}$ by gluing the handles  $\mathcal{H}^{\mathcal{D}}_{sa}$ and $\mathcal{H}^{\mathcal{R}}_{s}$}
		\label{tabela:Htriplo}
	\end{table}

\begin{table}[htb]
		\centering
		\resizebox{0.85\linewidth}{!}{
		\begin{tabular}{ccccccccc}
			 \cellcolor{gray!20} & \cellcolor{gray!20} Pre quotient  & \cellcolor{gray!20} & \cellcolor{gray!20} & \cellcolor{gray!20} Local chart & \cellcolor{gray!20} & \cellcolor{gray!20} & \cellcolor{gray!20} Attaching region & \cellcolor{gray!20} \\
			 %& & & & & & & & & & & \\
			 \hline
			 & & & & & & & & \\
			 & \includegraphics[align=c,scale=0.55]{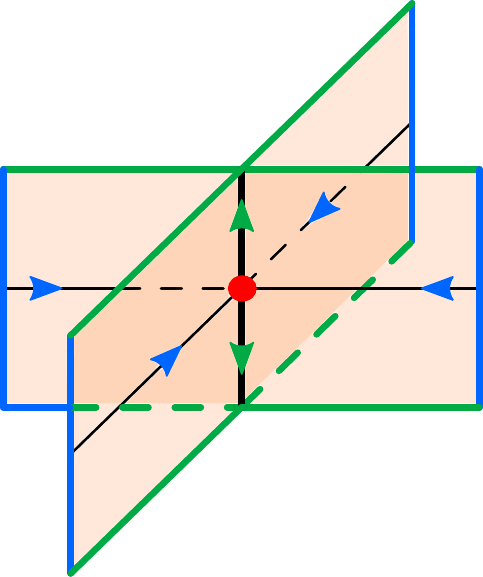} \hspace{0.8cm} \includegraphics[align=c,scale=0.45]{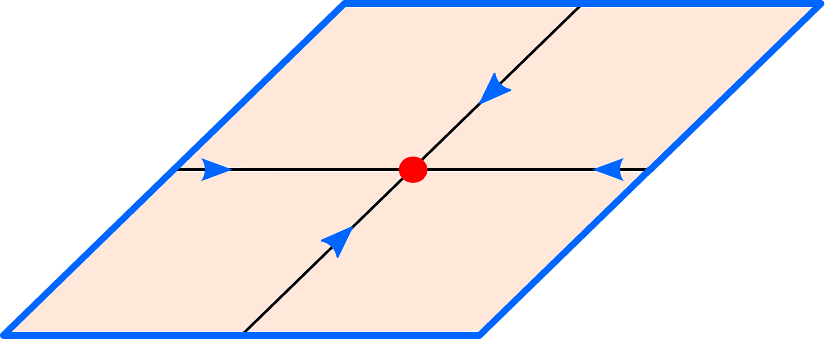} &  &  & \includegraphics[align=c,scale=0.55]{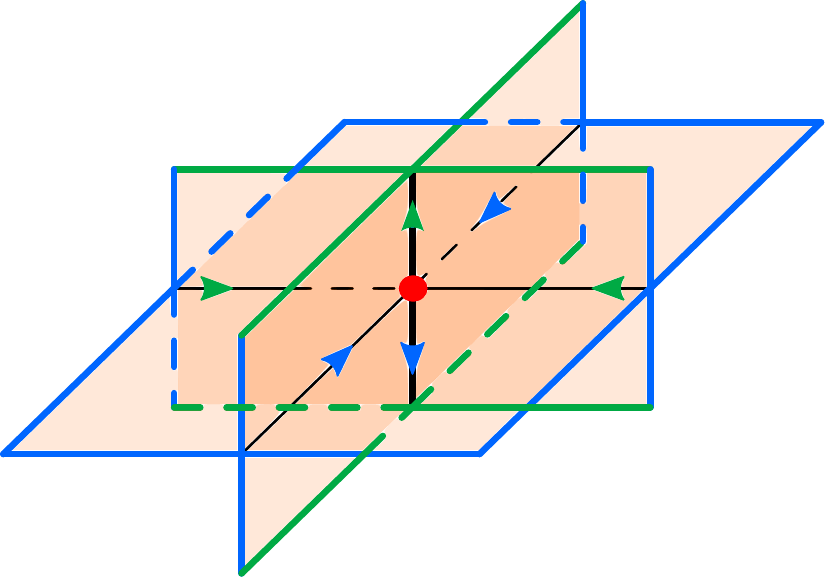} & & & \includegraphics[align=c,scale=0.5]{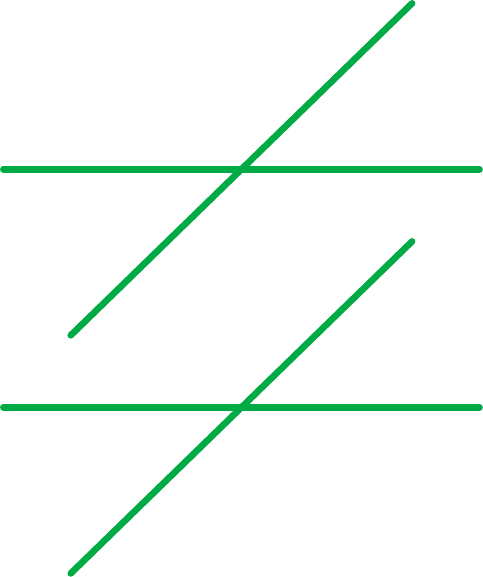} &   \\
			 &  &  &  &  & & & &  \\
		     %\hline
		\end{tabular}}
		%\mbox{ \ \ \ \ í­ndice do Ponto Regular}
		\caption{Representation of a triple handle $\mathcal{H}^{\mathcal{T}}_{ssa}$ by gluing the handles  $\mathcal{H}^{\mathcal{D}}_{ss_{u}}$ and $\mathcal{H}^{\mathcal{R}}_{a}$}
		\label{tabela:Htriplo2}
	\end{table}

 The attaching region has two connected components each of which has one branched chart, hence $N^-$ has either one connected component with two branched charts,  or two connected component with one branched chart each. In the prior case, we have two possibilities: two circles that intersect in two points, or three circles that form a figure eight wedge a circle. In the latter case, $N^-$ is the disjoint union of two figures eight.
 	 	   The isolating block for the singularity $p$ of nature $ssa$ is produced by identifying the corresponding pairs of stable orbits of the saddles $\mathcal{H}^{\mathcal{D}}_{sa}$ and  $\mathcal{H}^{\mathcal{R}}_{s}$ or of the saddle  $\mathcal{H}^{\mathcal{D}}_{ss_{u}}$ and  the handle $\mathcal{H}^{\mathcal{R}}_{a}$. See Tables \ref{tripleblocks1} and \ref{tripleblocks2}  for the schematic representation of the resulting blocks by their corresponding Lyapunov semi-graphs.

\end{itemize}

%\newpage

\begin{table}[!htb]
		\centering
		\resizebox{0.88\linewidth}{!}{
		\begin{tabular}{cc|cc}
			 \multicolumn{4}{c}{\cellcolor{gray!20} \scriptsize{Lyapunov graph moves} ($\mathcal{D}_{sa} + \mathcal{R}_{s} \rightarrow \mathcal{T}_{ssa}$)} \\
		    \hline
		    & & & \\
		    & & & \\
		    & \begin{overpic}[align=c,scale=0.3]{dsss1.png}
                \put(36,-8){\includegraphics[align=c,scale=0.2]{Branched1.pdf}}
                \put(23,86){\includegraphics[align=c,scale=0.2]{Branched3b.pdf}}
                \end{overpic} \begin{overpic}[align=c,scale=0.3]{rs3.png}
                \put(5,85){\includegraphics[scale=0.2,align=c]{Branched1.pdf}}
             \put(62,85){\includegraphics[scale=0.2,align=c]{Branched1.pdf}}
             \put(35,-5){\includegraphics[scale=0.2,align=c]{Branched1.pdf}}
             \put(100,35){$\rightarrow$}\end{overpic} \hspace{0.4cm} \begin{overpic}[scale=0.3,align=c]{tssa1.png}
		        \put(28,100){\includegraphics[scale=0.2,align=c]{Branched5b.pdf}}
		        \put(28,-11){\includegraphics[scale=0.2,align=c]{Branched3a.pdf}}
	            \end{overpic} & \begin{overpic}[align=c,scale=0.3]{dsa1.png}
             \put(25,86){\includegraphics[scale=0.2,align=c]{Branched3a.pdf}}
             \put(36,-8){\includegraphics[scale=0.2,align=c]{Branched1.pdf}}
             \end{overpic} \begin{overpic}[align=c,scale=0.3]{rs3.png}
                \put(5,85){\includegraphics[scale=0.2,align=c]{Branched1.pdf}}
             \put(62,85){\includegraphics[scale=0.2,align=c]{Branched1.pdf}}
             \put(35,-5){\includegraphics[scale=0.2,align=c]{Branched1.pdf}}
             \put(100,35){$\rightarrow$}\end{overpic} \hspace{0.4cm} \begin{overpic}[scale=0.3,align=c]{tssa1.png}
		        \put(26,100){\includegraphics[scale=0.2,align=c]{Branched5d.pdf}}
		        \put(28,-11){\includegraphics[scale=0.2,align=c]{Branched3a.pdf}}
	            \end{overpic} \\
	            & & & \\
	            \hline
	            & & & \\
	            & & & \\
	            & \begin{overpic}[align=c,scale=0.3]{dsss1.png}
                \put(36,-8){\includegraphics[align=c,scale=0.2]{Branched1.pdf}}
                \put(23,86){\includegraphics[align=c,scale=0.2]{Branched3b.pdf}}
                \end{overpic} \begin{overpic}[scale=0.3,align=c]{rs2.png}
		        \put(34,86){\includegraphics[scale=0.2,align=c]{Branched1.pdf}}
		        \put(66,-9){\includegraphics[scale=0.2,align=c]{Branched1.pdf}}
		        \put(3,-9){\includegraphics[scale=0.2,align=c]{Branched1.pdf}} \put(100,35){$\rightarrow$}\end{overpic} \hspace{0.4cm} \begin{overpic}[scale=0.3,align=c]{tssa1.png}
		        \put(26,100){\includegraphics[scale=0.2,align=c]{Branched5d.pdf}}
		        \put(28,-11){\includegraphics[scale=0.2,align=c]{Branched3b.pdf}}
	            \end{overpic} & \begin{overpic}[align=c,scale=0.3]{dsss1.png}
                \put(36,-8){\includegraphics[align=c,scale=0.2]{Branched1.pdf}}
                \put(23,86){\includegraphics[align=c,scale=0.2]{Branched3a.pdf}}
                \end{overpic} \begin{overpic}[scale=0.3,align=c]{rs2.png}
		        \put(34,86){\includegraphics[scale=0.2,align=c]{Branched1.pdf}}
		        \put(66,-9){\includegraphics[scale=0.2,align=c]{Branched1.pdf}}
		        \put(3,-9){\includegraphics[scale=0.2,align=c]{Branched1.pdf}} \put(100,35){$\rightarrow$}\end{overpic} \hspace{0.4cm} \begin{overpic}[scale=0.3,align=c]{tssa1.png}
		        \put(26,100){\includegraphics[scale=0.2,align=c]{Branched5e.pdf}}
		        \put(28,-11){\includegraphics[scale=0.2,align=c]{Branched3b.pdf}}
	            \end{overpic} \\
	            & & & \\
	       \hline  
	       & & & \\
	            & & & \\
	            & \begin{overpic}[align=c,scale=0.3]{dsss1.png}
                \put(36,-8){\includegraphics[align=c,scale=0.2]{Branched1.pdf}}
                \put(23,86){\includegraphics[align=c,scale=0.2]{Branched3b.pdf}}
                \end{overpic} \begin{overpic}[scale=0.3,align=c]{rs1.png}
		        \put(36,86){\includegraphics[scale=0.2,align=c]{Branched1.pdf}}
		        \put(36,-9){\includegraphics[scale=0.2,align=c]{Branched1.pdf}} \put(100,35){$\rightarrow$}\end{overpic} \hspace{0.4cm} \begin{overpic}[scale=0.3,align=c]{tssa1.png}
		        \put(26,100){\includegraphics[scale=0.2,align=c]{Branched5d.pdf}}
		        \put(28,-11){\includegraphics[scale=0.2,align=c]{Branched3a.pdf}}
	            \end{overpic} & \begin{overpic}[align=c,scale=0.3]{dsa1.png}
             \put(25,86){\includegraphics[scale=0.2,align=c]{Branched3a.pdf}}
             \put(36,-8){\includegraphics[scale=0.2,align=c]{Branched1.pdf}}
             \end{overpic} \begin{overpic}[scale=0.3,align=c]{rs1.png}
		        \put(36,86){\includegraphics[scale=0.2,align=c]{Branched1.pdf}}
		        \put(36,-9){\includegraphics[scale=0.2,align=c]{Branched1.pdf}}
             \put(100,35){$\rightarrow$}\end{overpic} \hspace{0.4cm} \begin{overpic}[scale=0.3,align=c]{tssa1.png}
		        \put(26,100){\includegraphics[scale=0.2,align=c]{Branched5e.pdf}}
		        \put(28,-11){\includegraphics[scale=0.2,align=c]{Branched3a.pdf}}
	            \end{overpic} & \\ 
	            & & & \\
	       \hline     
	            & & & \\
	            & & & \\
	            & \begin{overpic}[align=c,scale=0.3]{dsss3.png}
                \put(24,86){\includegraphics[align=c,scale=0.2]{Branched3a.pdf}}
                \put(5,-6){\includegraphics[align=c,scale=0.2]{Branched1.pdf}}
                \put(65,-6){\includegraphics[align=c,scale=0.2]{Branched1.pdf}}
                \end{overpic} \begin{overpic}[align=c,scale=0.3]{rs3.png}
                \put(5,85){\includegraphics[scale=0.2,align=c]{Branched1.pdf}}
             \put(62,85){\includegraphics[scale=0.2,align=c]{Branched1.pdf}}
             \put(35,-5){\includegraphics[scale=0.2,align=c]{Branched1.pdf}}
             \put(100,35){$\rightarrow$}\end{overpic} \hspace{0.4cm} \begin{overpic}[scale=0.3,align=c]{tssa1.png}
		        \put(26,100){\includegraphics[scale=0.2,align=c]{Branched5d.pdf}}
		        \put(28,-11){\includegraphics[scale=0.2,align=c]{Branched3b.pdf}}
	            \end{overpic}  & \begin{overpic}[align=c,scale=0.3]{dsss3.png}
                \put(24,86){\includegraphics[align=c,scale=0.2]{Branched3a.pdf}}
                \put(5,-6){\includegraphics[align=c,scale=0.2]{Branched1.pdf}}
                \put(65,-6){\includegraphics[align=c,scale=0.2]{Branched1.pdf}}
                \end{overpic} \begin{overpic}[scale=0.3,align=c]{rs1.png}
		        \put(36,86){\includegraphics[scale=0.2,align=c]{Branched1.pdf}}
		        \put(36,-9){\includegraphics[scale=0.2,align=c]{Branched1.pdf}}
             \put(100,35){$\rightarrow$}\end{overpic} \hspace{0.4cm} \begin{overpic}[scale=0.3,align=c]{tssa1.png}
		        \put(26,100){\includegraphics[scale=0.2,align=c]{Branched5e.pdf}}
		        \put(28,-11){\includegraphics[scale=0.2,align=c]{Branched3b.pdf}}
	            \end{overpic} \\
	            & & & \\
	            \hline
	            \multicolumn{4}{c}{} \\
	            \multicolumn{4}{c}{} \\
	            \multicolumn{4}{c}{\begin{overpic}[align=c,scale=0.3]{dsss3.png}
                \put(24,86){\includegraphics[align=c,scale=0.2]{Branched3a.pdf}}
                \put(5,-10){\includegraphics[align=c,scale=0.2]{Branched1.pdf}}
                \put(65,-10){\includegraphics[align=c,scale=0.2]{Branched1.pdf}}
                \end{overpic} \hspace{0.2cm} \begin{overpic}[scale=0.3,align=c]{rs2.png}
		        \put(34,86){\includegraphics[scale=0.2,align=c]{Branched1.pdf}}
		        \put(66,-9){\includegraphics[scale=0.2,align=c]{Branched1.pdf}}
		        \put(3,-9){\includegraphics[scale=0.2,align=c]{Branched1.pdf}} \put(100,35){$\rightarrow$}\end{overpic} \hspace{0.4cm} \begin{overpic}[scale=0.3,align=c]{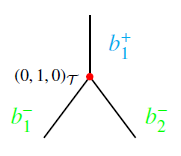}
		        \put(60,-10){\includegraphics[scale=0.2,align=c]{Branched2.pdf}}
		        \put(-3,-10){\includegraphics[scale=0.2,align=c]{Branched2.pdf}}
		         \put(24,95){\includegraphics[align=c,scale=0.2]{Branched5e.pdf}} \end{overpic} \hspace{0.2cm}} \\
		        \multicolumn{4}{c}{} \\
\end{tabular}}
	   \caption{ Distinguished branched $1$-manifold on the boundary of minimal isolating blocks for $p \in \mathcal{H}^{\mathcal{T}}_{ssa}$
	   %$1$-manifolds variedades ramificadas nos blocos isolantes minimais para $p \in \mathcal{H}^{\mathcal{T}}_{ssa}$
	    }		\label{tripleblocks1}
\end{table}

\begin{table}[!htb]
		\centering
		\resizebox{0.88\linewidth}{!}{
		\begin{tabular}{cc|cc}
			 \multicolumn{4}{c}{\cellcolor{gray!20} \scriptsize{Lyapunov graph moves} ($\mathcal{D}_{ss_{u}} + \mathcal{R}_{a} \rightarrow \mathcal{T}_{ssa}$)} \\
		    \hline
		    & & & \\
		    & & & \\
		    & \begin{overpic}[scale=0.3,align=c]{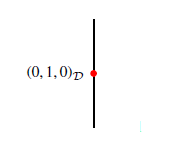}
                \put(-15,-11){(}
		        \put(-5,-11){\includegraphics[scale=0.2,align=c]{Branched3a.pdf}}
		        \put(48,-11){,}
		        \put(58,-11){\includegraphics[scale=0.2,align=c]{Branched3b.pdf}}
		        \put(125,-11){)}
		        \put(34,86){\includegraphics[scale=0.2,align=c]{Branched1.pdf}}
	            \end{overpic} \hspace{0.4cm} \begin{overpic}[scale=0.3,align=c]{ra.png}
		        \put(32,88){\includegraphics[scale=0.2,align=c]{Branched1.pdf}}
             \put(100,35){$\rightarrow$}\end{overpic} \hspace{0.4cm} \begin{overpic}[scale=0.3,align=c]{tssa1.png}
		        \put(25,100){\includegraphics[scale=0.13,align=c]{B5T2.pdf}}
		        \put(-15,-11){(}
		        \put(-5,-11){\includegraphics[scale=0.2,align=c]{Branched3a.pdf}}
		        \put(48,-11){,}
		        \put(58,-11){\includegraphics[scale=0.2,align=c]{Branched3b.pdf}}
		        \put(125,-11){)}
	            \end{overpic} \hspace{0.4cm} & \hspace{0.4cm}  \begin{overpic}[scale=0.3,align=c]{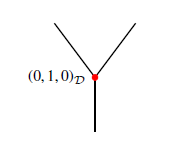}
		        \put(28,-9){\includegraphics[scale=0.2,align=c]{Branched3a.pdf}}
		        \put(66,86){\includegraphics[scale=0.2,align=c]{Branched1.pdf}}
		        \put(3,86){\includegraphics[scale=0.2,align=c]{Branched1.pdf}}
	            \end{overpic} \hspace{0.4cm} \begin{overpic}[scale=0.3,align=c]{ra.png}
		        \put(32,88){\includegraphics[scale=0.2,align=c]{Branched1.pdf}}
             \put(100,35){$\rightarrow$}\end{overpic} \hspace{0.4cm} \begin{overpic}[scale=0.3,align=c]{tssa1.png}
		        \put(21,95){\includegraphics[scale=0.2,align=c]{Branched5e.pdf}}

		        \put(25,-9){\includegraphics[scale=0.2,align=c]{Branched3a.pdf}}
		        
	            \end{overpic} & \\
	            & & & \\
	            \hline
	            & & & \\
	            & & & \\
                & \begin{overpic}[scale=0.3,align=c]{ssu2.png}
		        \put(25,-9){\includegraphics[scale=0.2,align=c]{Branched3b.pdf}}
		        \put(66,86){\includegraphics[scale=0.2,align=c]{Branched1.pdf}}
		        \put(3,86){\includegraphics[scale=0.2,align=c]{Branched1.pdf}}
	            \end{overpic} \hspace{0.4cm} \begin{overpic}[scale=0.3,align=c]{ra.png}
		        \put(32,88){\includegraphics[scale=0.2,align=c]{Branched1.pdf}}
             \put(100,35){$\rightarrow$}\end{overpic} \hspace{0.4cm} \begin{overpic}[scale=0.3,align=c]{tssa1.png}
		        \put(20,95){\includegraphics[scale=0.22,align=c]{Branched5c.pdf}}
             
		        \put(25,-10){\includegraphics[scale=0.2,align=c]{Branched3b.pdf}}
		       
	            \end{overpic} \hspace{0.4cm} & \hspace{0.2cm} \begin{overpic}[scale=0.3,align=c]{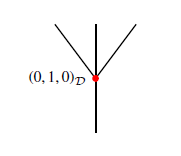}
		        \put(-15,-11){(}
		        \put(-5,-11){\includegraphics[scale=0.2,align=c]{Branched3a.pdf}}
		        \put(48,-11){,}
		        \put(58,-11){\includegraphics[scale=0.2,align=c]{Branched3b.pdf}}
		        \put(125,-11){)}
		        \put(35,86){\includegraphics[scale=0.2,align=c]{Branched1.pdf}}
		        \put(0,75){\includegraphics[scale=0.2,align=c]{Branched1.pdf}}
		        \put(74,75){\includegraphics[scale=0.2,align=c]{Branched1.pdf}}
	            \end{overpic} \hspace{0.4cm} \begin{overpic}[scale=0.3,align=c]{ra.png}
		        \put(32,88){\includegraphics[scale=0.2,align=c]{Branched1.pdf}}
             \put(100,35){$\rightarrow$}\end{overpic} \hspace{0.4cm} \begin{overpic}[scale=0.3,align=c]{tssa1.png}
		        \put(20,100){\includegraphics[scale=0.23,align=c]{Branched5d.pdf}}
		        \put(-15,-11){(}
		        \put(-5,-11){\includegraphics[scale=0.2,align=c]{Branched3a.pdf}}
		        \put(48,-11){,}
		        \put(58,-11){\includegraphics[scale=0.2,align=c]{Branched3b.pdf}}
		        \put(125,-11){)}
	            \end{overpic}  \\
	            & & & \\
	            \hline
	            & & & \\
	            & & & \\
	            & \begin{overpic}[scale=0.3,align=c]{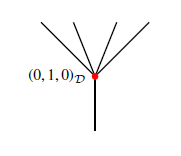}
		        \put(27,-11){\includegraphics[scale=0.2,align=c]{Branched3a.pdf}}
		        \put(-12,75){\includegraphics[scale=0.2,align=c]{Branched1.pdf}}
		        \put(15,90){\includegraphics[scale=0.2,align=c]{Branched1.pdf}}
		        \put(55,90){\includegraphics[scale=0.2,align=c]{Branched1.pdf}}
		        \put(86,75){\includegraphics[scale=0.2,align=c]{Branched1.pdf}}
	            \end{overpic} \hspace{0.4cm} \begin{overpic}[scale=0.3,align=c]{ra.png}
		        \put(32,88){\includegraphics[scale=0.2,align=c]{Branched1.pdf}}
             \put(100,35){$\rightarrow$}\end{overpic} \hspace{0.4cm} \begin{overpic}[scale=0.3,align=c]{tssa1.png}
		        \put(28,100){\includegraphics[scale=0.2,align=c]{Branched5b.pdf}}
		        \put(28,-11){\includegraphics[scale=0.2,align=c]{Branched3a.pdf}}
	            \end{overpic} \hspace{0.2cm} & \hspace{0.2cm} \begin{overpic}[scale=0.3,align=c]{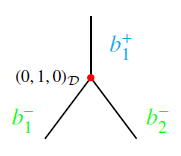}
	            \put(33,88){\includegraphics[scale=0.2,align=c]{Branched1.pdf}}
	            \put(60,-10){\includegraphics[scale=0.2,align=c]{Branched2.pdf}}
		        \put(-3,-10){\includegraphics[scale=0.2,align=c]{Branched2.pdf}} \end{overpic} \hspace{0.4cm} \begin{overpic}[scale=0.3,align=c]{ra.png}
		        \put(32,88){\includegraphics[scale=0.2,align=c]{Branched1.pdf}}
             \put(100,35){$\rightarrow$}\end{overpic} \hspace{0.4cm} \begin{overpic}[scale=0.3,align=c]{tssa12.png}
             \put(60,-10){\includegraphics[scale=0.2,align=c]{Branched2.pdf}}
		        \put(-3,-10){\includegraphics[scale=0.2,align=c]{Branched2.pdf}}
		         \put(25,96){\includegraphics[align=c,scale=0.13]{B5T2.pdf}} \end{overpic}  \\
             & & & \\
             \hline
	            \multicolumn{4}{c}{} \\
	            \multicolumn{4}{c}{} \\
	            \multicolumn{4}{c}{\begin{overpic}[scale=0.3,align=c]{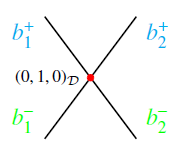}
	            \put(3,88){\includegraphics[scale=0.2,align=c]{Branched1.pdf}}
	            \put(60,88){\includegraphics[scale=0.2,align=c]{Branched1.pdf}}
	            \put(60,-10){\includegraphics[scale=0.2,align=c]{Branched2.pdf}}
		        \put(-3,-10){\includegraphics[scale=0.2,align=c]{Branched2.pdf}} \end{overpic} \hspace{0.4cm} \begin{overpic}[scale=0.3,align=c]{ra.png}
		        \put(32,88){\includegraphics[scale=0.2,align=c]{Branched1.pdf}}
             \put(100,35){$\rightarrow$}\end{overpic} \hspace{0.4cm} \begin{overpic}[scale=0.3,align=c]{tssa12.png}
             \put(60,-10){\includegraphics[scale=0.2,align=c]{Branched2.pdf}}
		        \put(-3,-10){\includegraphics[scale=0.2,align=c]{Branched2.pdf}}
		         \put(22,95){\includegraphics[align=c,scale=0.2]{Branched5e.pdf}} \end{overpic}} \\
		         \multicolumn{4}{c}{} \\
	   \end{tabular}}
	   \caption{
	  Distinguished branched $1$-manifold on the boundary of minimal isolating blocks for $p \in \mathcal{H}^{\mathcal{T}}_{ssa}$
	   %$1$-manifolds variedades ramificadas nos blocos isolantes minimais para $p \in \mathcal{H}^{\mathcal{T}}_{ssa}$
	   }		\label{tripleblocks2}
\end{table}

\end{proof}

The next theorem proves the non-realizability of certain  Lyapunov  semi-graphs as a  minimal GS isolating block. The vertex on these semi-graphs are labelled with the Conley indices of a GS singularity and satisfy the Poincaré-Hopf condition.

\newtheorem{Lem}{Lemma}

\begin{Teo}[Local non-realizability]\label{teo:combinatorial}
Let $X_{t}$  be a GS flow associated to a vector field  $\mathbf{X} \in \Sigma^{r}_{0}(\mathbf{M})$, such that $p$ is a singularity of  $X_{t}$ and  $N$ is a minimal isolating block for $p$.  Then there is no Lyapunov semi-graph associated to $N$ such that:

\begin{itemize}
    \item[a) ] $e_{v}^{+} = 2$, $e_{v}^{-} = 4$, $p \in \mathcal{H}^{\mathcal{D}}_{ss_{s}}$;
    
    \item[b) ] $e_{v}^{+} = 2$, $e_{v}^{-} = 3$, $p \in \mathcal{H}^{\mathcal{D}}_{ss_{s}}$;
    
    \item[c) ] $e_{v}^{+} = 2$, $e_{v}^{-} = 1$, $b_{1}^{-} = 1$, $b_{1}^{+} \neq b_{2}^{+}$, $p \in \mathcal{H}^{\mathcal{D}}_{ss_{s}}$;
    
    \item[d) ] $e_{v}^{+} = 2$, $e_{v}^{-} = 2$, $b_{1}^{-} = b_{2}^{-} = 1$, $b_{1}^{+} \neq b_{2}^{+}$, $p \in \mathcal{H}^{\mathcal{D}}_{ss_{s}}$;
    
    \item[e) ] $e_{v}^{+} = 1$, $e_{v}^{-} = 2$, $b_{1}^{-} \neq b_{2}^{-}$, $p \in \mathcal{H}^{\mathcal{T}}_{ssa}$.
\end{itemize}
\end{Teo}

\begin{proof}

The idea of the proof is to analyze the connected components of $N^{+}$  in view of the steps of the construction of GS isolating blocks presented in the beginning of Section  \ref{sec:construcao}. This analysis is based on  the different possibilities of embeddings of the attaching regions of handles to $N^-$.

In order to prove $(a)$, recall that a double handle  
 $\mathcal{H}^{\mathcal{D}}_{ss_{s}}$ corresponds to  two regular handles $H_{1} = H_{2} = \mathcal{H}^{\mathcal{R}}_{s}$, where two stable orbits of  $H_{1}$ are identified,  biuniquivocally, to two stable orbits of  $H_{2}$ as shown in Table \ref{tabela:HDuploss}.  Furthermore, the attaching region  $A_{k}$ of the handle $\mathcal{H}^{\mathcal{D}}_{ss_{s}}$  is the union of two copies of  $S^{0} \times D^1$, where $D^1$ is the one-dimensional disc, i.e. $D^1\simeq [0,1]$.

    Considering $e_{v}^{-} = 4$, it follows that $N^{-}\times [0,1]$ has $4$ connected components, that is $N^{-}\times [0,1] = \cup_{i=1}^{4}{N_{i}\times [0,1]}$.  It follows from the construction of GS isolating blocks that an embedding $f: A_{k} \rightarrow N^{-}\times \{1\}$  maps  each disc $D^{1}_i$ in $A_{k}$ to $N_{i}\times\{1\}$.
    
The attaching map  of  $H_{1}$ connects it to 
    $N_{1}\times\{1\}$ and  $N_{2}\times \{1\}$. Also the attaching map of $H_2$ connects it to  $N_{3}\times\{1\}$ and $N_{4}\times \{1\}$. Whether these attaching maps are orientation preserving or not has no effect on the gluing of the handles.

  Thus,    $N^{+}$ is the union of two circles given by the stable part of  $N_{1}\times\{1\} \cup H_{1} \cup N_{2}\times \{1\}$ and $N_{3}\times\{1\} \cup H_{2} \cup N_{4}\times \{1\}$, intersecting in two points which correspond to the intersection of  $H_{1}$ and $H_{2}$ in $\mathcal{H}^{\mathcal{D}}_{ss_{s}}$.  Hence, it follows that  $N^{+}$ is connected, i.e., $e_{v}^{+} = 1$. See Figure \ref{fig:bordos-exemplo}.

    \begin{figure}[h]
    \centering
    \begin{tikzpicture}

    \node {
            \resizebox{0.9\linewidth}{!}{
            \begin{tabular}{cccccccccc}
                \cellcolor{gray!20} Gluing onto $N^{-}$ & \cellcolor{gray!20} & \cellcolor{gray!20} & \cellcolor{gray!20} Stable part & \cellcolor{gray!20} & \cellcolor{gray!20} & \cellcolor{gray!20} $N^{+}$ identification & \cellcolor{gray!20} & \cellcolor{gray!20} & \cellcolor{gray!20} Lyapunov graph \\
                \hline
			    & & & & & & & & & \\
			    \begin{tikzcd}
                \begin{overpic}[scale=0.6,align=c]{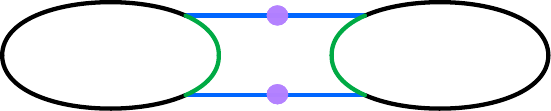}
		        \put(49,26){$\alpha$}
		        \put(49,-10){$\alpha^{'}$}
	            \end{overpic}
                \end{tikzcd} \hspace{0.4cm} \begin{tikzcd}
                \begin{overpic}[scale=0.6,align=c]{Bordo1.pdf}
		        \put(49,26){$\beta$}
		        \put(49,-10){$\beta^{'}$}
	            \end{overpic}
                \end{tikzcd} & & & \begin{tikzcd}
                \begin{overpic}[scale=0.6,align=c]{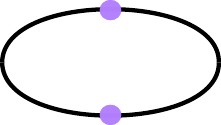}
		        \put(49,68){$\alpha$}
		        \put(49,-22){$\alpha^{'}$}
	            \end{overpic}
                \end{tikzcd} \hspace{0.4cm} \begin{tikzcd}
                \begin{overpic}[scale=0.6,align=c]{Bordo3.pdf}
		        \put(49,68){$\beta$}
		        \put(49,-22){$\beta^{'}$}
	            \end{overpic}
                \end{tikzcd} & & & 
			    \begin{tikzcd}
                \begin{overpic}[scale=0.6,align=c]{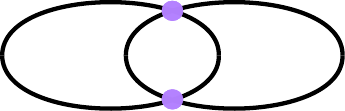}
		        \put(35,42){$\alpha \sim \beta$}
		        \put(35,-18){$\alpha^{'} \sim \beta^{'}$}
	            \end{overpic}
                \end{tikzcd} & & & \includegraphics[scale=0.6,align=c]{dsss7.png}\\
			    & & & & & & & & & \\
            \end{tabular}}
            
        };
    \end{tikzpicture}
    \caption{Boundary  $N^{-}$ and $N^{+}$ of an isolating block for $p \in \mathcal{H}^{\mathcal{D}}_{ss_{s}}$ with $e_{v}^{-} = 4$}
    \label{fig:bordos-exemplo}
\end{figure}

    Therefore, no  minimal GS isolating block for  $p \in \mathcal{H}^{\mathcal{D}}_{ss_{s}}$ admits a Lyapunov semi-graph with $e_{v}^{+} = 2$, $e_{v}^{-} = 4$.

The proof of items  $b)$, $c)$, $d)$ and  $e)$ is similar.
\end{proof}

In Theorem \ref{teoPH}, the set of Lyapunov semi-graphs with a single vertex which satisfies the Poncaré-Hopf condition for Conley indices that correspond to the Conley indices of GS singularities, were presented. Now, by Theorem \ref{teo:colecao} together with Theorem \ref{teo:combinatorial}, we have realized all possible  Lyapunov semi-graphs with minimal weight  for GS singularities as minimal GS isolating blocks.

This classification implies the following corollary.

\newtheorem{Cor}{Corollary}

\begin{Cor}
There are, up to homeomorphism and flow reversal,  $33$ minimal GS isolating blocks.
\end{Cor}

\begin{proof}
Theorem \ref{teo:combinatorial} implies that the collection of  Lyapunov semi-graphs of Theorem \ref{teo:colecao} (Figura \ref{tab:bordos-colecao}) contains all semi-graphs that are realizable as minimal GS isolating blocks. On the other hand, Theorem  \ref{teo:colecao} implies that, up to homeomorphims and flow reversal, the total number of minimal GS isolating blocks for singularities of types  $\mathcal{R}, \mathcal{C}, \mathcal{W}, \mathcal{D}$ and  $\mathcal{T}$ are, respectively, $3, 3, 3, 13$ and $11$, totalizing $33$ blocks.
\end{proof}

In the next section, the collection of   Lyapunov semi-graphs  with a single vertex labelled with the Conley indices of a GS singularity and which  satisfies the Poincaré-Hopf condition for GS flows will be analyzed   by removing the minimality weight condition on the edges. In other words, arbitrary weights on the edges  will be allowed as long as the  Poincaré-Hopf condition  is satisfied.

\subsection{GS Isolating Blocks with  passageway}\label{sec:passageway}

Let $X_{t}$ be a GS flow associated to a vector field  $\mathbf{X} \in \Sigma^{r}_{0}(\mathbf{M})$ such that $p$ is a singularity of  $X_{t}$ and  $N$  is an isolating block for $p$. Note that each branched chart on the entering boundary of $N^+$ (resp., exiting boundary $N^{-}$)  represents a fold within the isolating block.
In terms of the  Lyapunov semi-graph, 
$ b_{i}^{+} - 1$  (resp., $b_{i}^{-} - 1$)  equals the number of folds that enter (resp. exit) through the corresponding connected component of the block.

Note that in the case of  a minimal GS isolating block  $N$, the total number $F$ of folds that enter and exit $N$ is:

$$
F = \left\{ 
\begin{array}{ll}
0 \ , &   \text{if} \  p \in \mathbf{M}(\mathcal{R})\cup \mathbf{M}(\mathcal{C}) \\
 1 \ , &   \text{if} \   p \in \mathbf{M}(\mathcal{W})\\
 2 \ , &   \text{if} \   p \in \mathbf{M}(\mathcal{D}) \\
  6 \ , &  \text{if} \  p \in \mathbf{M}(\mathcal{T})
\end{array}
\right.
$$

Furthermore, the  $\omega$-limit (resp. $\alpha$-limit) of all folds that enter  (resp., exit) through  $N^{+}$ (resp. $N^{-}$) is the singularity  $p \in N$. 

\begin{Def}
Let $N$ be a GS isolating block for a singularity  $p$  of a vector field  $\mathbf{X} \in \Sigma^{r}_{0}(\mathbf{M})$.  $N$ is a GS isolating block with  \textbf{passageways} if  there exists at least one  fold in $N$  for  which  $p$ is neither the  $\alpha$-limit nor  $\omega$-limit.  
\end{Def}

As a direct consequence of the definition,  the next result follows.

\begin{Cor}
Let  $p$ be a singularity of a vector field   $\mathbf{X} \in \Sigma^{r}_{0}(\mathbf{M})$, of attracting  (resp. repelling) nature of type $\mathcal{R}, \mathcal{C}, \mathcal{W}, \mathcal{D}$ or $\mathcal{T}$. Then  $p$  does not admit a GS isolating block with passageways.
\end{Cor}

\begin{proof}
Note that if  $p$ is of attracting  (resp. repelling) nature and  $N$ is an isolating block for  $p$, then  $p$ is  $\omega$-limit (resp. $\alpha$-limit) of all orbits in $N$. 
\end{proof}

%A seguir, mostramos como construir blocos isolantes GS com passageway a partir de blocos isolantes GS minimais.

\begin{Cor}\label{cor:construcao_passageway}
Given $N$  a minimal GS isolating block for $p \in \mathcal{H}^{\mathcal{P}}_{\eta}$, let $\Gamma = \{(\gamma_{1},\gamma_{2}), (\gamma_{3},\gamma_{4}), \ldots, (\gamma_{2k-1},\gamma_{2k})\}$ be a collection of pairs of orbits in the regular part of $N$, with $\gamma_{i}\neq\gamma_{j}$ if $i\neq j$,  and such that  $p$ is not the  $\alpha$-limit nor the  $\omega$-limit of any $\gamma_{i}$ in  $\Gamma$. Then the quotient space  $N/\sim$  obtained by identifying each pair of orbits  $\gamma_{2i-1}\sim\gamma_{2i}$, for $i=1,\ldots,k$ is an isolating block for  $p \in \mathcal{H}^{\mathcal{P}}_{\eta}$ with $k$ passageways.
\end{Cor}

\begin{proof}
Since  $N$ is an isolating block for $p \in \mathcal{H}^{\mathcal{P}}_{\eta}$, it follows that  $N/\sim$ is also an isolating block for $p$, which has a total of $k$ folds given by the identification of  the pairs of orbits  $\gamma_{2i-1}\sim\gamma_{2i}$, for $i=1,\ldots,k$.  Given that  $p$ is not the  $\alpha$-limit nor  the $\omega$-limit of any  $\gamma_{i}$ in $\Gamma$,  it follows that  $p$ is not the $\alpha$-limit nor the $\omega$-limit of any of the $k$ folds in $N/\sim$. 
\end{proof}

\newtheorem{Ex}{Example}

\begin{Ex}
In Table \ref{tabela:Passageway}, there are three example of GS isolating blocks with  passageways, constructed from a minimal GS isolating block for a cone type singularity with saddle nature. 
 The pairs of orbits identified  in order to obtain each one of the three blocks are, respectively, given by  $\Gamma_{1} = \{(\gamma_{5}, \gamma_{6})\}$, $\Gamma_{2} = \{(\gamma_{1},\gamma_{3}), (\gamma_{2},\gamma_{4}), (\gamma_{5},\gamma_{6})\}$, and $\Gamma_{3} = \{(\gamma_{3},\gamma_{6})\}$.

\begin{table}[!htb]
		\centering
		\resizebox{0.9\linewidth}{!}{
		\begin{tabular}{ccc|ccc}
			 \cellcolor{gray!20} & \cellcolor{gray!20} Minimal isolating block & \cellcolor{gray!20} & \cellcolor{gray!20} & \cellcolor{gray!20} Passageways in isolating blocks & \cellcolor{gray!20}  \\
			 %& & & & & & & & & & & \\
			 \hline
			 & & & & &  \\
			 & & & & &  \\
			 & \begin{overpic}[align=c,scale=0.35]{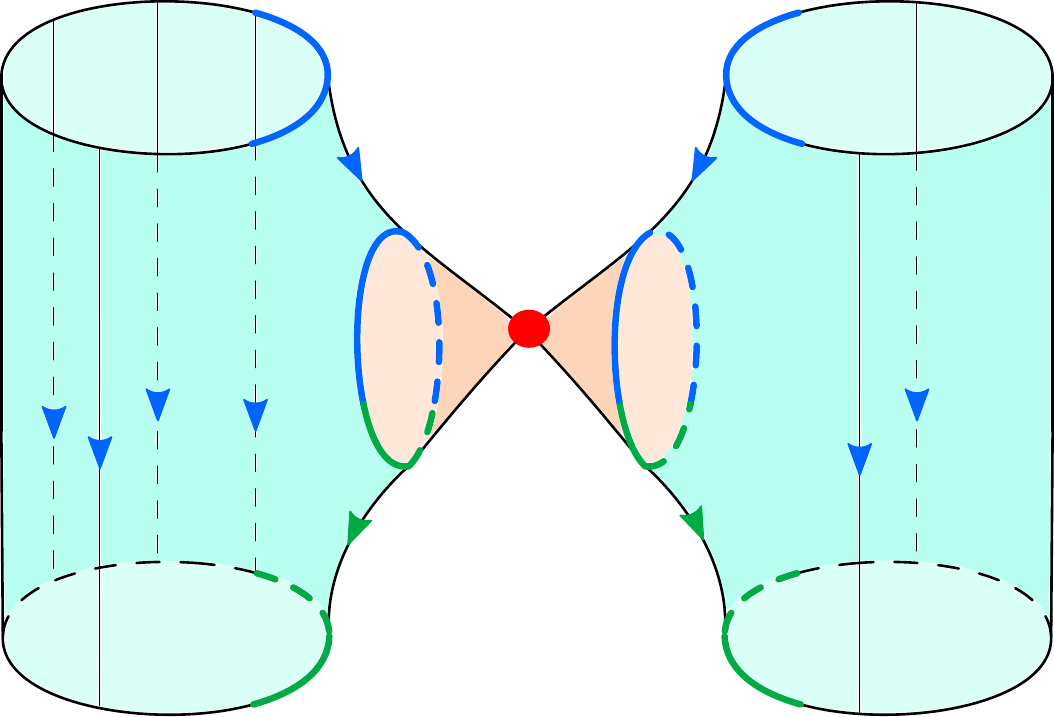}
			 \put(5,-8){$\gamma_{1}$}
			 \put(-5,71){$\gamma_{2}$}
			 \put(9,75){$\gamma_{3}$}
			 \put(25,71){$\gamma_{4}$}
			 \put(76,-8){$\gamma_{5}$}
			 \put(86,73){$\gamma_{6}$} \end{overpic} &  &  & \begin{overpic}[align=c,scale=0.28]{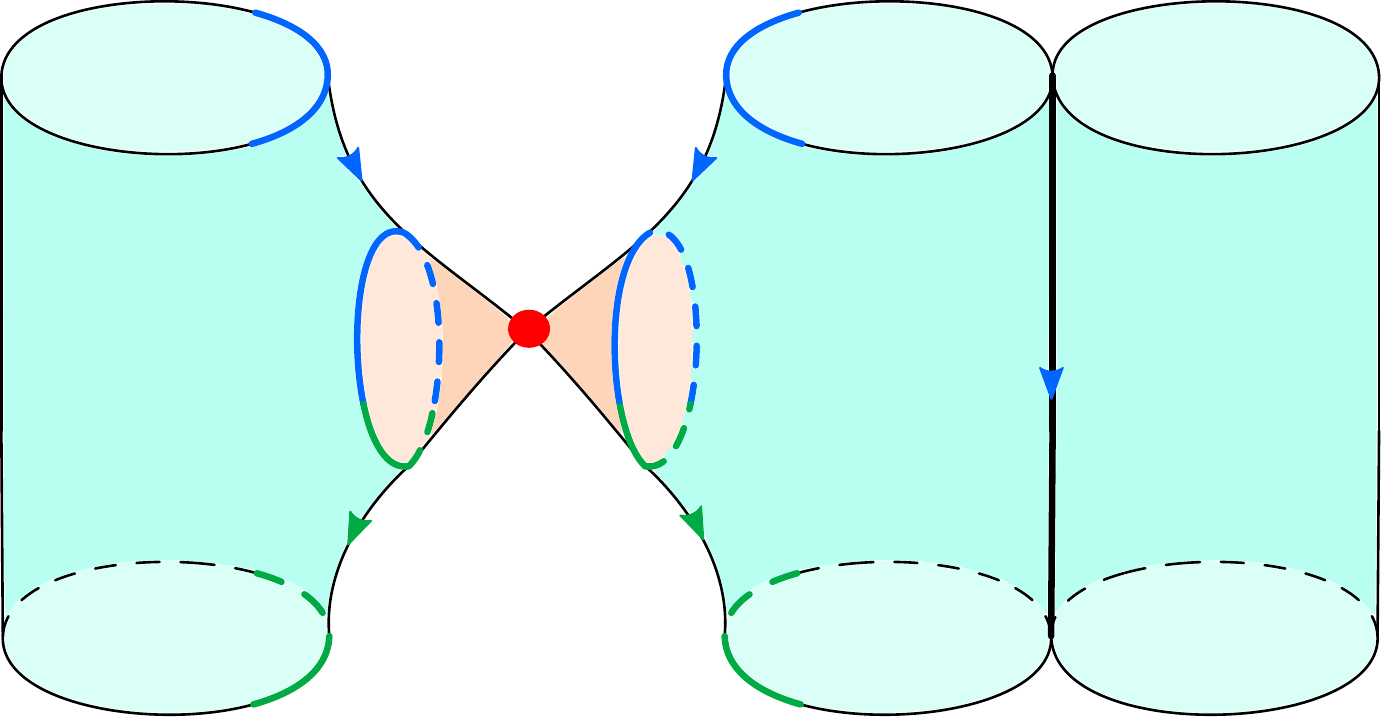}\end{overpic} \hspace{0.8cm} \begin{overpic}[align=c,scale=0.28]{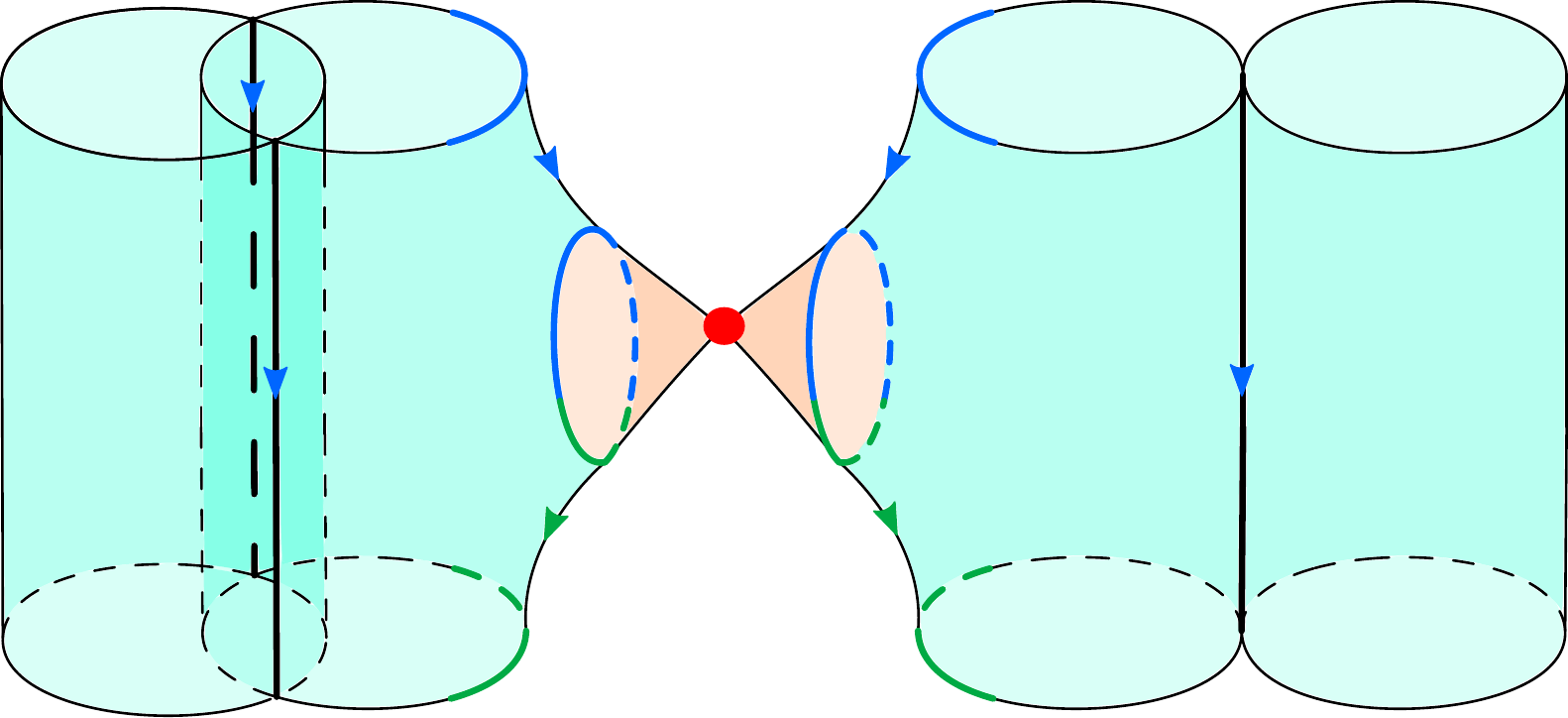}\end{overpic} \hspace{0.8cm} \begin{overpic}[align=c,scale=0.25]{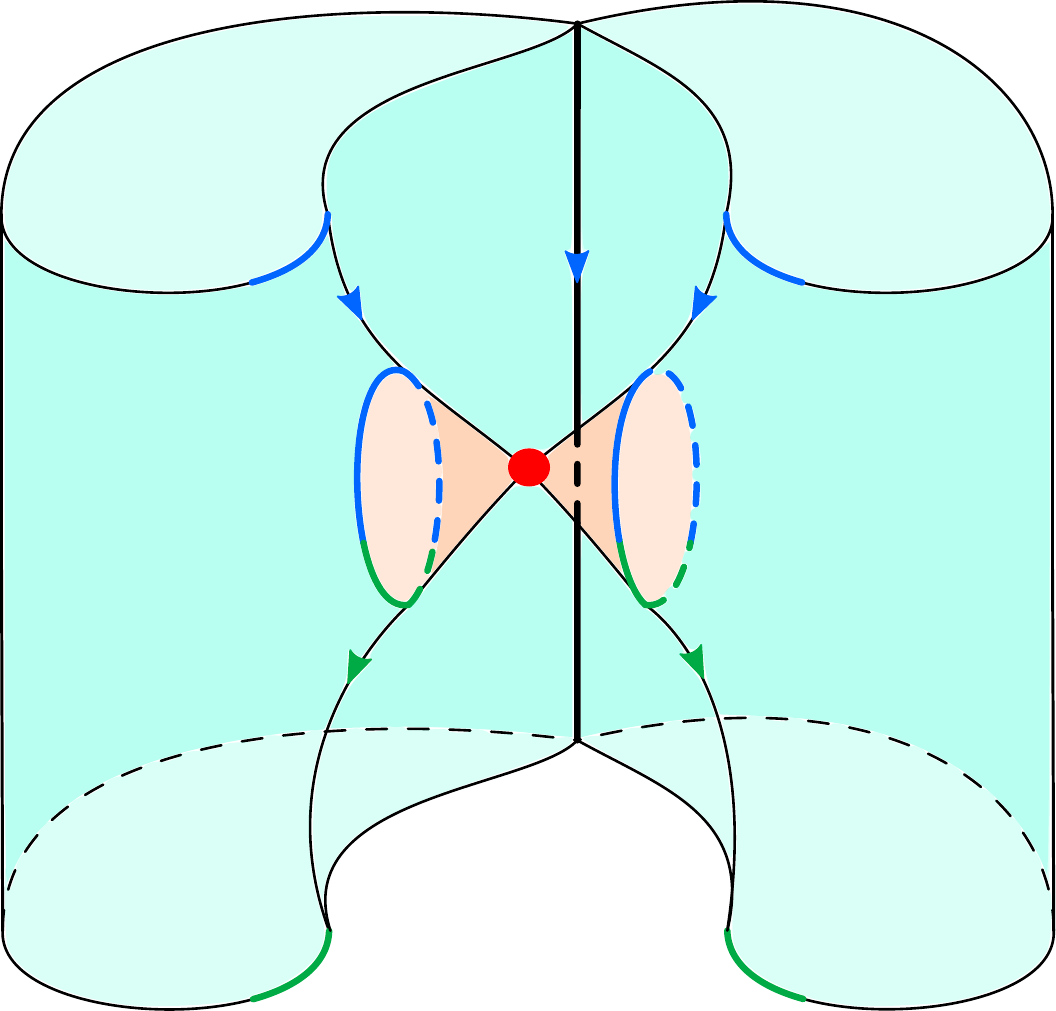}\end{overpic}  &   \\
			 &  &  &  &  &   \\
			 & & & & &  \\
		     %\hline
		\end{tabular}}
		%\mbox{ \ \ \ \ í­ndice do Ponto Regular}
		\caption{Examples of GS isolating blocks with passageways for  $p \in \mathcal{H}^{\mathcal{C}}_{s}$}
		\label{tabela:Passageway}
	\end{table}

\end{Ex}

Corollary \ref{cor:construcao_passageway} provides a procedure to construct GS isolating blocks  with passageways from a minimal GS isolating block. On the other hand, the next lemma shows that any GS isolating block  with passageways can be obtained by this method.

\newtheorem{AF}{Afirmação}

\begin{Lem}\label{lem:passageway_minimal}
Let $N$ be a GS isolating block with $k$ passageways for $p \in \mathcal{H}^{\mathcal{P}}_{\eta}$. Then there exists a minimal GS isolating block  $N_{0}$ for $p \in \mathcal{H}^{\mathcal{P}}_{\eta}$, and a collection  $\Gamma = \{(\gamma_{1},\gamma_{2}), (\gamma_{3},\gamma_{4}), \ldots, (\gamma_{2k-1},\gamma_{2k})\}$  of pairs of regular orbits in  $N_{0}$, with the properties that $\gamma_{i}\neq\gamma_{j}$ if $i\neq j$, and $p$ is neither the $\alpha$-limit nor the $\omega$-limit of any  $\gamma_{i}$ in $\Gamma$, such that the quotient space  $N_{0}/\sim$, obtained by the identification of each pair of orbits $\gamma_{2i-1}\sim\gamma_{2i}$, for $i=1, \ldots, k$, is homeomorphic to $N$.
\end{Lem}

\begin{proof}
Given a GS isolating block $N$ with  $k$ passageways for $p \in \mathcal{H}^{\mathcal{P}}_{\eta}$, denote by $F_{1}, \ldots, F_{k}$ the $k$ folds in $N$  for which  $p$ is neither  the $\alpha$-limit nor the $\omega$-limit. A minimal GS isolating block $N_0$ for  $p \in \mathcal{H}^{\mathcal{P}}_{\eta}$ is  constructed as follows. For  each $i=1, \ldots, k$: 

\begin{itemize}
\item[i) ] Consider ${N - U_{i}}$ where  $U_{i}$ is an  $\epsilon$-neighborhood  of  $F_{i}$ in $N$. Then $\overline{\partial (N - U_{i}) - (\partial N \cap ({N - U_{i}}) )}$  is the union of four disjoint discs $D^1_{i1}$, $D^1_{i2}$, $D^1_{i3}$ and $D^1_{i4}$. 

\item[ii) ] Identify  the pair of discs in ${N - U_{i}}$,  say $\gamma_{2i-1} = D^1_{i1} \sim D^1_{i3}$  and  $\gamma_{2i} = D^1_{i2} \sim D^1_{i4}$, so that the resulting block has the same number of connected components as $N$.

\end{itemize}

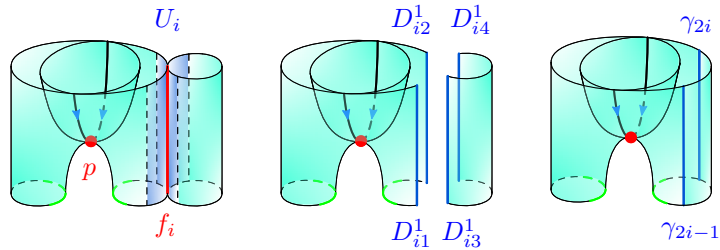
\begin{figure}[htb]
\centering
\begin{overpic}[scale=0.4,align=c]{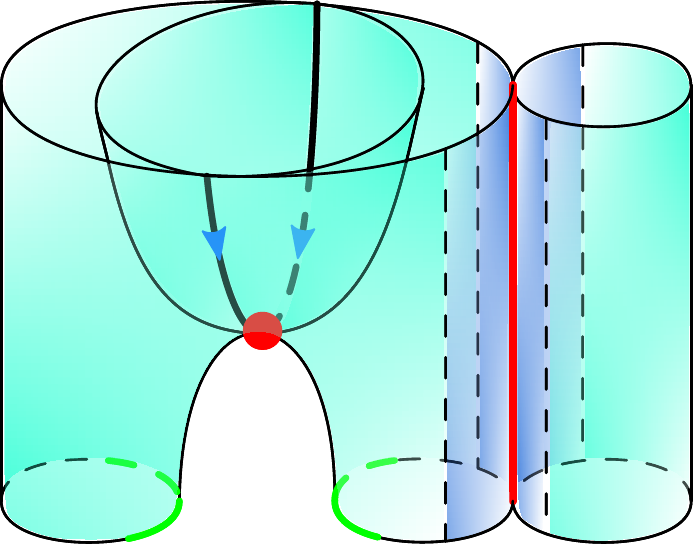}
\put(34,15){\textcolor{red}{$p$}}
\put(68,-12){\textcolor{red}{$f_{i}$}}
\put(68,85){\textcolor{blue}{$U_{i}$}} \end{overpic} \hspace{0.6cm} \begin{overpic}[scale=0.4,align=c]{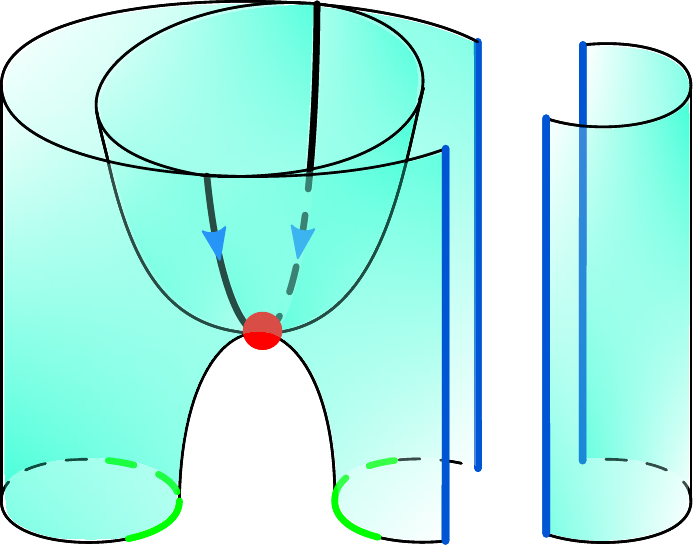}
\put(52,-16){\textcolor{blue}{$D_{i1}^{1}$}}
\put(76,-16){\textcolor{blue}{$D_{i3}^{1}$}}
\put(52,85){\textcolor{blue}{$D_{i2}^{1}$}}
\put(80,85){\textcolor{blue}{$D_{i4}^{1}$}} \end{overpic} \hspace{0.6cm} \begin{overpic}[scale=0.4,align=c]{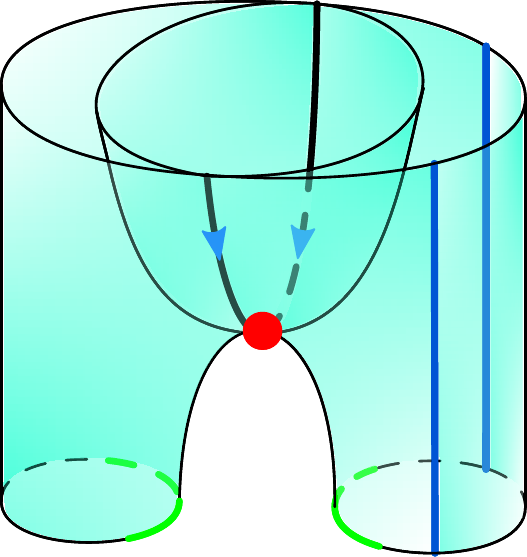} 
\put(65,-16){\textcolor{blue}{$\gamma_{2i-1}$}}
\put(76,105){\textcolor{blue}{$\gamma_{2i}$}} \end{overpic} \vspace{0.5cm}
\caption{Unzipping passageways }\label{fig:patches}
\end{figure}

 Since we have removed all the folds $F_1,\dots,F_k$ in this process, the resulting block $N_{0}$  is a minimal GS isolating block for $p \in \mathcal{H}^{\mathcal{P}}_{\eta}$.% with the same  Lyapunov semi-graph of $N$ up to the weights on the edges.

Finally, note that  $\Gamma = \{(\gamma_{1},\gamma_{2}), (\gamma_{3},\gamma_{4}), \ldots, (\gamma_{2k-1},\gamma_{2k})\}$   is a collection of pairs of regular orbits in $N_{0}$, such that  $p$ is neither the $\alpha$-limit nor the $\omega$-limit of any  $\gamma_{i}$ in $\Gamma$.
 Hence, the identifications $\gamma_{2i-1}\sim\gamma_{2i}$, for $i=1,\ldots,k$, imply that $N_{0}/\sim$ is homeomorphic to  $N$.  
\end{proof}

One can ask whether an abstract Lyapunov semi-graph  with a single vertex $v$ labelled with the numerical Conley indices of a GS singularity and arbitrary weights satisfying the Poincaré-Hopf condition is realizable or not as a GS flow on an isolating block.

The following Theorem presents necessary and sufficient conditions for a positive answer to this question. So the following theorem  generalizes Theorem \ref{teo:realizationminimalsemi}.

\begin{Teo}[Local realization characterization]\label{teo:passageway_cone}
A Lyapunov semi-graph $L_{v}$ with a single vertex $v$ labelled with a GS singularity is associated to a GS flow on a  GS isolating block $N$ if and only if:
\begin{itemize}
    \item[i) ] the Poincaré-Hopf condition is satisfied;

    \item[ii) ]  if $v$ is  labelled with a singularity of type $\mathcal{C}$ and nature $s$, with $e_{v}^{-} = e_{v}^{+} = 2$, then $b_{1}^{+}$ is equal to either $b_{1}^{-}$ or $b_{2}^{-}$;
    
    \item[iii) ]   if $v$ is  labelled with a singularity of type $\mathcal{D}$ and nature $ss_{s}$ (resp. $ss_{u}$), with $e_{v}^{+} = 2 $, (resp., $e_{v}^{-} = 2 $) then $e_{v}^{-} \leq 2$ (resp., $e_{v}^{+} \leq 2$). Moreover, if $L$ is minimal then $b_1^+ = b_2^+$ (resp., $b_1^{-} = b_2^{-}$).
    
\item[iv) ]  If $v$ is labelled with a singularity of type $\mathcal{T}$ with $e^{-}_{v}=2$ (resp., $e^{+}_{v}=2$ ) and $L$ is minimal then  $b_{1}^{-} = b_{2}^{-}$ (resp., $b_{1}^{+} = b_{2}^{+}$).
\end{itemize}   
\end{Teo}

\begin{proof}

$(\Rightarrow)$ First, suppose $L$ is associated to a GS flow on a GS isolating block $N$. Then, we have that:

\begin{itemize}

\item[i) ] $(N,N^{-})$ is an index pair for the singularity $p$ in $N$. And it follows from Theorem \ref{PH} that $L$ satisfies the Poincaré-Hopf condition.    

\item[ii) ]  If  $L$  is minimal, the condition is trivially satisfied, since  $b_{1}^{+} = b_{1}^{-} = b_{2}^{-} = 1$. In this case, $N$ is a minimal GS isolating block given by two cylinders intersecting  in a single point which is the cone singularity, as shown in Table  \ref{NEWconeblocks}.      

If $L$ is not minimal, then $N$ is a GS isolating block with passageways. Thus, Lemma \ref{lem:passageway_minimal} states that $N$ is obtained from the identification of pairs of orbits in a minimal GS isolating block $N_{0}$, such that these orbits do not intersect the singular part of $N_{0}$. So, each pair of orbits must enter and exit $N_{0}$ through the same cylinder, that is, the number of folds that enter and the number of folds that exit each cylinder  of $N_{0}$ must be the same. Hence the number  $b_{1}^{+} - 1$ must be equal to $b_{1}^{-} - 1$ or $b_{2}^{-} - 1$. Therefore,  $b_{1}^{+}$ is equal to  $b_{1}^{-}$ or $b_{2}^{-}$.

\item[iii) ] Suppose $v$ is labelled with a singularity of type $\mathcal{D}$ and nature $ss_{s}$, with $e_{v}^{+} = 2$. In addition, if $N$ is minimal, Theorem \ref{teo:combinatorial} asserts that  the exiting set of $N$ has at most two connected components, that is, $e_{v}^{-} \leq 2$. On the other hand, if $N$ is not minimal, Lemma \ref{lem:passageway_minimal} implies that $N$ can be obtained via the identification of pairs of orbits on a minimal isolating block $N_{0}$. Since these identifications may lower the number of connected components of $N_{0}$, it follows that the exiting set of $N$ also has at most two connected components, i.e, we also have $e_{v}^{-} \leq 2$ in this case. Similarly, we prove the condition for a singularity of type $\mathcal{D}$ and nature $ss_{u}$.

\item[iv) ]  This is a direct consequence of item $e)$ of Theorem \ref{teo:combinatorial}.  

\end{itemize}

$(\Leftarrow)$ Now, suppose $L$ satisfies conditions $i)$ through $iv)$. If $L$ is minimal, then $L$ can be realized with the distinguished branched $1$-manifolds of Table \ref{tab:bordos-colecao}, as shown in Theorem \ref{teo:colecao}. On the other hand, if $L$ is not minimal, consider the minimal semi-graph $\ell$ with the same labellings, indegree and outdegree as $L$. As before, we can realize $\ell$ by a minimal GS isolating block $N_{0}$. Then, we can identify a pair of orbits that enter and exit $N_{0}$ through the same connected components of the boundary of $N_{0}$. This identification increases the weights of the respective boundary components of $N_{0}$ by one, while preserving the number of connected components of $N_{0}^{+}$ and $N_{0}^{-}$. Hence, such identification of orbits of $N_{0}$ results in a GS isolating block $N_{1}$ which realizes a semi-graph with same labellings, indegree and outdegree as $\ell$. Therefore, after performing a finite number $k$ of identifications so that the weights on each boundary component of the resulting block $N_{k}$ are equal to the weights $b_{i}^{\pm}$ of $L$, we have that $N_{k}$ is a realization of $L$, proving that $L$ is associated to a GS flow on a GS isolating block.

\end{proof}

This concludes the characterization of all Lyapunov semi-graphs labelled with numerical Conley indices of GS singularities
and satisfying the Poincaré-Hopf condition. We will denote these semi-graphs as \textbf{GS semi-graphs}.

\begin{Def}\label{def:grafoGS}
%A Lyapunov semi-graph $L$ that satisfies   Theorem \ref{teo:passageway_cone}  for all vertices $v$ of $L$ will be called  \textbf{GS semi-graphs}. 
%Semi-graphs that satisfy Theorem \ref{teo:colecao} (Table \ref{fig:bordos-colecao}) and their orientation reverse counterpart  with any choice of weights and degrees that satisfy Theorem \ref{teo:passageway_cone} will be called  \textbf{GS semi-graphs}. 
In the case that the weights on the edges of a GS semi-graph are the smallest satisfying Poincaré-Hopf it will  be called a 
\textbf{minimal GS semi-graph}. Hence, we say that a 
 Lyapunov graph $L$ is a  \textbf{GS graph}  (resp.,  \textbf{minimal GS graph}) if for all vertices  $v$ of $L$ the semi-graph formed by $v$ and its incident edges is a GS  semi-graph (resp., \textbf{minimal GS semi-graph}). 
\end{Def}

In the next section, we investigate under what conditions a GS graph is realizable as a GS flow on a closed $2$-dimensional singular manifold.

\section{Realization of abstract GS graphs as singular flows}\label{sec:realization}

In this section we consider the realization of GS graphs as singular flows.

\begin{Def} We say  that a GS graph $L$ is \textbf{realizable} if there exists a triple ($\mathbf{M}$, $X_{t}$,$f$), where $\mathbf{M}$ is a closed GS $2$-manifold, $X_{t}$ is a flow associated to a vector field  $\mathbf{X} \in \Sigma^{r}_{0}(\mathbf{M})$ and $f:\mathbf{M} \rightarrow \mathbb{R}$ is a Lyapunov function associated to  $X_{t}$, such that $L$ is the Lyapunov graph of   ($\mathbf{M}$, $X_{t}$,$f$). In this case, we say that  the flow $X_{t}$ defined   on  $\mathbf{M}$  is a  \textbf{realization} of $L$.
\end{Def}

 Let $L$ be a GS graph. We would like to investigate whether $L$ is realizable.So far, we have shown that for each vertex 
 $v$  of $L$, the GS semi-graph  formed by $v$  and its incident edges admits at least one realization for the singularity associated to $v$ as a minimal isolating block or as an isolating block with passageways.   
 The question now is to determine if these isolating blocks can be glued to each other  as indicated by 
 $L$.

In order to glue the exiting boundary of an isolating block onto the entering boundary of another isolating block, the distinguished branched $1$-manifolds which make up these two boundary components must be homeomorphic.

Usually, in a realization, one isolating block $B$ needs to be glued to at least two other blocks, $A$ and $C$, so that the exiting boundary of $A$ must be homeomorphic to the entering boundary of $B$, and at the same time, the exiting boundary of $B$ must be homeomorphic to the entering boundary of $C$.

Notice that the exiting boundary of $A$ and the entering boundary of $C$ are independent from one another. However, the distinguished branched $1$-manifold on the entering boundary of $B$ restricts the possible distinguished branched $1$-manifolds on the exiting boundary of $B$ and vice versa. So it might happen that we are able to glue $A$ and $B$, or $B$ and $C$, but not all three of them simultaneously.

In other words, a realization of a GS graph $L$ is reduced to the problem of assigning distinguished branched $1$-manifolds  to the edges of $L$ in such a way that the semi-graph of each vertex $v$ of $L$ is realizable  as a GS isolating block with precisely this choice of distinguished branched $1$-manifolds as  boundary components.
 
Most of the results in this section,  with the exception of Theorem  \ref{teo:minimal},  will consider GS flows without triple crossing singularities. The reason being the complexity of the distinguished branched $1$-manifolds that make up the boundary components of the corresponding isolating blocks and the fact that this choice has a trickle down effect on  boundary components of isolating blocks of other singularities.

\subsection{Realization of GS  graphs with weight or degree restrictions}

The following theorem will impose the minimality restriction on the weights of a GS graph.  This makes it possible to have a better control on the choices of the distinguished branched $1$-manifolds that make up the boundary of the isolating blocks.

\begin{Teo}[Minimal case]\label{teo:minimal}
Let $L$ be a minimal GS graph containing  singularities of type  $ \mathcal{R}, \mathcal{C}$, $\mathcal{W}$, $\mathcal{D}$ and $\mathcal{T}$.  $L$ admits  a realization.
\end{Teo}

\begin{proof}

Since $L$ is minimal, the possible weights on the edges of $L$ are $1$
 $2$, $3$, $5$ or $7$. For each edge, consider the assignment  of distinguished branched $1$-manifolds given in Table \ref{tab:branchedminimal}.

\begin{table}[!h]
			\centering
			\resizebox{0.65\linewidth}{!}{
			\begin{tabular}{c|ccccccccccccccc}
			 \cellcolor{gray!20} w = weight & \cellcolor{gray!20} & \cellcolor{gray!20} $w = 1$ & \cellcolor{gray!20} & \cellcolor{gray!20} & \cellcolor{gray!20} $w = 2$ & \cellcolor{gray!20} & \cellcolor{gray!20} & \cellcolor{gray!20} $w = 3$ & \cellcolor{gray!20} & \cellcolor{gray!20} & \cellcolor{gray!20} $w = 5$ & \cellcolor{gray!20} & \cellcolor{gray!20} & \cellcolor{gray!20} $w = 7$ & \cellcolor{gray!20} \\
			  \hline
			   \cellcolor{gray!20} & &  &  &  &  &  &  &  &  &  &  &  &  &  &  \\
			   \cellcolor{gray!20} $N^{-}$ & & \includegraphics[align=c,scale=0.3]{Branched1.pdf} &  &  & \includegraphics[align=c,scale=0.3]{Branched2.pdf} &  &  & \includegraphics[align=c,scale=0.3]{Branched3a.pdf} &  &  & \includegraphics[align=c,scale=0.25]{Branched5e.pdf} &  &  & \includegraphics[align=c,scale=0.3]{Branched7.pdf} &  \\
			   \multicolumn{16}{c}{}  \\
			\end{tabular}}
			\caption{Distinguished branched $1$-manifolds with weight $w$}
			\label{tab:branchedminimal}
		\end{table}

This choice relies on the collection of minimal GS isolating blocks in Theorem \ref{teo:colecao} (see Table \ref{tab:bordos-colecao}). Note that there is a unique choice of distinguished branched $1$-manifolds for weights equal to  $1$, $2$ and $7$. In the case of weight $3$ there are two possibilities. 
However, the chosen one is common to all GS semi-graphs with  weight $3$ and hence, it is always realizable. In the case of weight $5$, there are two possible choices of distinguished branched $1$-manifolds common to all GS semi-graphs, anyone of which can be chosen. 
\end{proof}

\begin{Ex}\label{ex:minRCWD}

In Figure \ref{fig:realization}, a minimal GS graph is realized as a singular flow on a  GS $2$-manifold, as asserted in Theorem \ref{teo:minimal}.

\begin{figure}[!ht] 
\begin{equation*}
  \begin{tikzcd}
   \includegraphics[align=c,scale=0.5]{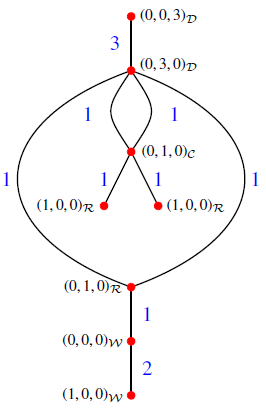} \hspace{0.5cm} \xrightarrow{Realization} \hspace{1cm} \includegraphics[align=c,scale=0.3]{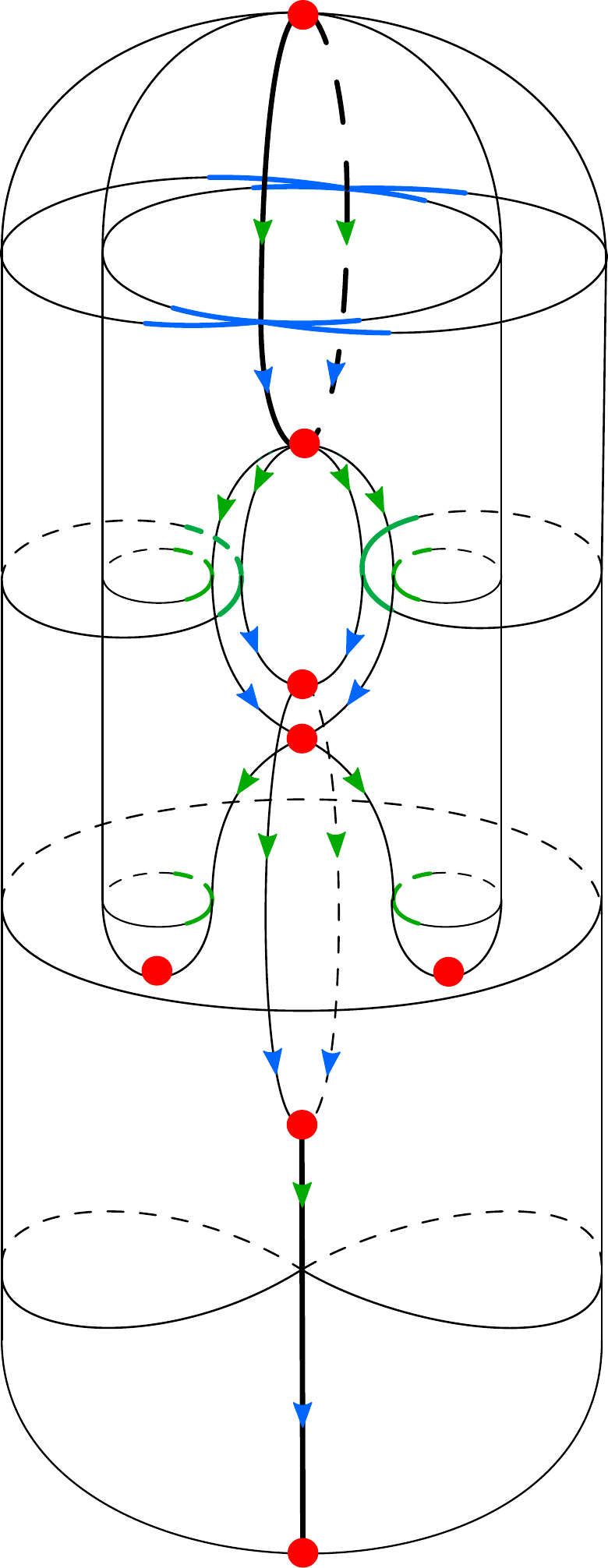}
  \end{tikzcd}
\end{equation*}
\caption{Realization of a minimal GS graph as a GS flow}
\label{fig:realization}
\end{figure}

\end{Ex}

{\bf Remark:} %\label{ex:minTar}
   The only example of a  minimal GS graph with a vertex labelled with a  singularity  of type $\mathcal{T}$ and of attracting or repelling nature is the repeller-attractor pair. In other words, a flow made up of two singularities of type $\mathcal{T}$, one  attractor and one repeller.

The next theorem will impose a degree restriction on a GS graph. It turns out that GS graphs with no bifurcation vertices are always realizable.

\begin{Teo}[Linear Graph]\label{teo:linear}
Let  $L$  be a GS graph  containing singularities of type $ \mathcal{R}, \mathcal{C}, \mathcal{W}$ and $\mathcal{D}$ such that every vertex $v$ of $L$  has degree less than or equal to $2$. Then  $L$ is realizable.
\end{Teo}

\begin{proof}
Given an edge of $L$ with weight $w$, the realization of $L$ can be achieved by choosing distinguished branched $1$-manifolds according to Table \ref{tab:bordosAAA}.

\begin{table}[!htb]
			\centering
			\resizebox{0.95\linewidth}{!}{
			\begin{tabular}{c|cccccccccccccccccc}
			 \cellcolor{gray!20} w = weight & \cellcolor{gray!20} & \cellcolor{gray!20} $w = 1$ & \cellcolor{gray!20} & \cellcolor{gray!20} & \cellcolor{gray!20} $w = 2$ & \cellcolor{gray!20} & \cellcolor{gray!20} & \cellcolor{gray!20} $w = 3$ & \cellcolor{gray!20} & \cellcolor{gray!20} & \cellcolor{gray!20} $w = 4$ & \cellcolor{gray!20} & 
			 \cellcolor{gray!20} $\ldots$ & \cellcolor{gray!20}  &
			 \cellcolor{gray!20} $w = 2k - 1$ & \cellcolor{gray!20}& \cellcolor{gray!20} & \cellcolor{gray!20} $w = 2k$ \\
			  \hline
			   \cellcolor{gray!20} & &  &  &  &  &  &  &  &  &   & & & &  &  & & & \\
			   \cellcolor{gray!20} $N^{-}$ & & \includegraphics[align=c,scale=0.3]{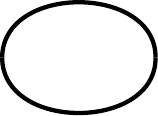} &  &  & \includegraphics[align=c,scale=0.3]{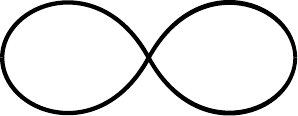} &  &  & \includegraphics[align=c,scale=0.3]{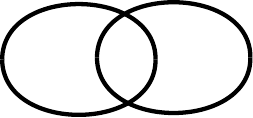} &  &  & \includegraphics[align=c,scale=0.3]{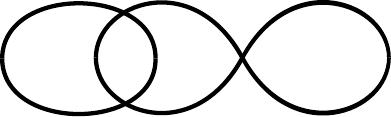} &  &         & & \includegraphics[align=c,scale=0.3]{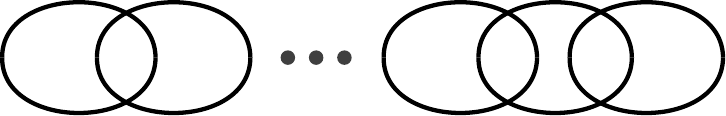} & & & \includegraphics[align=c,scale=0.3]{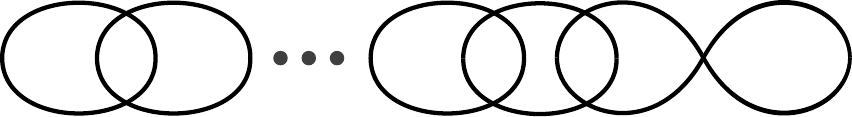} \\
			   \cellcolor{gray!20} & &  &  &  &  &  &  &  &  &  &   & & &  &  & & & \\
			   
			\end{tabular}}
			\caption{Distinguished branched $1$-manifolds with weight $w$}
			\label{tab:bordosAAA}
		\end{table}

Let $N^{+}$ be a distinguished branched $1$-manifold   which is the entering boundary of a minimal GS isolating block  $N$. Consider the  sequence of identifications of a pair of points in Figure \ref{fig:sequence1}, where these pair of points in  $N^{+}$ are on orbits in  $N$ whose $\omega$-limit does not belong  to $N$.

\begin{figure}[h] 
\begin{equation*}
  \begin{tikzcd}
   \includegraphics[align=c,scale=0.4]{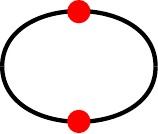} \hspace{0.4cm} \rightarrow \hspace{0.4cm} \includegraphics[align=c,scale=0.4]{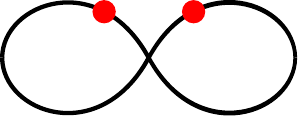} \hspace{0.4cm} \rightarrow \hspace{0.4cm} \includegraphics[align=c,scale=0.4]{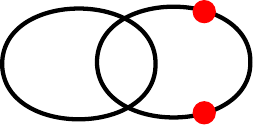} \hspace{0.4cm} \rightarrow \hspace{0.4cm} \includegraphics[align=c,scale=0.4]{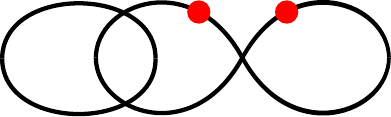} \hspace{0.4cm} \rightarrow  \hspace{0.4cm} \includegraphics[align=c,scale=0.4]{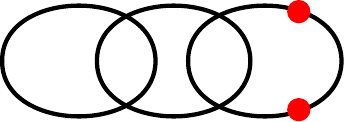} \hspace{0.4cm} \rightarrow \hspace{0.4cm} \ldots
  \end{tikzcd}
\end{equation*}
\caption{Identification sequence }
\label{fig:sequence1}
\end{figure}

Each identification produces a branch point.  After performing $k$ identifications, the GS isolating block will have $k$ passageways. It is easy to verify that the exiting boundaries of any isolating block arising in this way  will have as boundary components the distinguished branched $1$-manifolds in Table \ref{tab:bordosAAA}.

Each GS semi-graph $L_{v}$.  with one vertex  $v \in L$ can be realized  by the above procedure starting from the corresponding minimal GS semi-graph. This corresponding  minimal  GS semi-graph has the same indegree, outdegree  and labelling  as the semi-graph $L_{v}$.

Since any given boundary with weight $n$ has a unique distinguished branched $1$-manifold that is realized on the boundary of isolating blocks, it is trivial to glue one block to the other. Hence, $L$ is realizable.

\end{proof}

\begin{Ex}\label{ex:RCWDlinear}
In Figure \ref{fig:realizationRCWDlinear}, a GS graph (on the left) satisfying Theorem  \ref{teo:linear} is realized as a GS flow (on the right).

\begin{figure}[h] 
\begin{equation*}
  \begin{tikzcd}
   \includegraphics[align=c,scale=0.6]{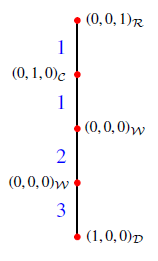} \hspace{0.5cm} \xrightarrow{Realization} \hspace{1cm} \includegraphics[align=c,scale=0.26]{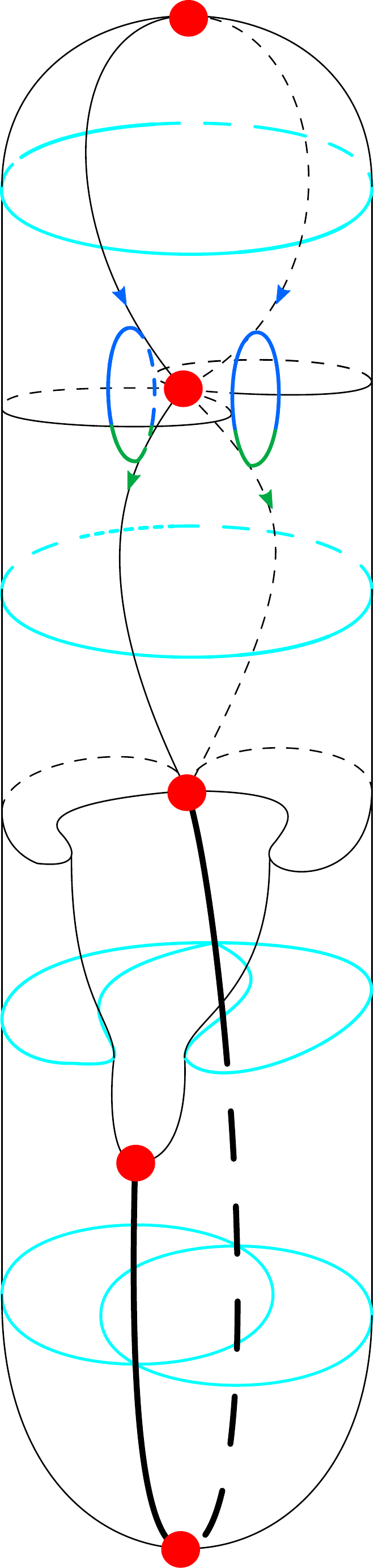}
  \end{tikzcd}
\end{equation*}
\caption{Realization of a GS graph as a GS flow}
\label{fig:realizationRCWDlinear}
\end{figure}
\end{Ex}

It is interesting to note that bifurcation vertices with incident edges that are not labelled with minimal weights, introduce complex combinatorial questions on the choice of the distinguished branched $1$-manifolds that are boundary components of isolating blocks. The following example illustrates this fact.
%As can be seen in Example \ref{ex:NR}.

\begin{Ex}[Non-realizable]\label{ex:NR}

In Figure \ref{fig:GSnot},  the GS graph $L$ contains a bifurcation vertex labelled with a singularity of type $\mathcal{W}$ whose weights are not minimal. Note that 
all semi-graphs containing a unique vertex of $L$ admit only one realization as a GS isolating block. However, the distinguished branched $1$-manifolds of weight $3$ that make up the exiting and entering boundaries of the first two blocks can not be glued to each other. Hence, the graph is non realizable.

\begin{figure}[htb]
\centering
\includegraphics[scale=0.6,align=c]{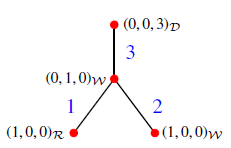} \hspace{2cm} \includegraphics[scale=0.3,align=c]{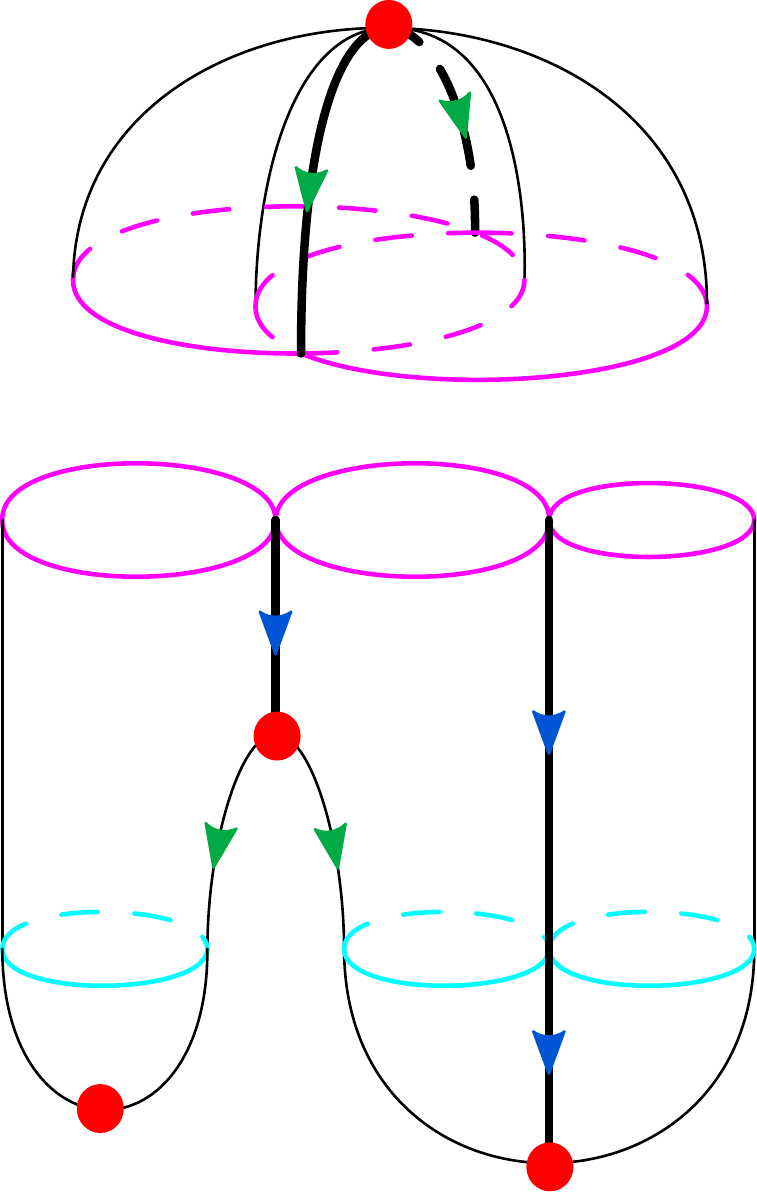}
\caption{A GS graph that is non realizable as a GS flow} 
\label{fig:GSnot}
\end{figure}

\end{Ex}

The next theorem is an attempt to deal with the problem presented in Example \ref{ex:NR}. It will blend in the results in Theorems \ref{teo:minimal} and \ref{teo:linear} by imposing  minimal weights only on incident edges of bifurcation vertices.

\begin{Teo}\label{teo:blend}
Let  $L$ be a GS graph containing  singularities of types  $\mathcal{R}, \mathcal{C}, \mathcal{W}$ and $\mathcal{D}$. Suppose that all incident edges of vertices $v$ of $L$ with degree greater than or equal to $3$ has minimal weights.
% then a GS semi-graph given by a vertex $v$ and its incident edges is minimal. 
Then $L$ is realizable.
\end{Teo}

\begin{proof}

The proof is a direct consequence of Theorems \ref{teo:minimal} and \ref{teo:linear}.

Let $L_{v}$ be a GS semi-graph consisting of a single vertex $v \in L$ and its incident edges. Consider the decomposition of $L$ given as follows:
$$L = \left( \bigcup_{deg(v) < 3} L_{v} \right) \cup \left( \bigcup_{deg(v) \geq 3} L_{v} \right) $$

We have that each connected component of $\bigcup_{deg(v) < 3} L_{v}$ is realizable by Theorem \ref{teo:linear} with the choice of distinguished branched $1$-manifolds presented in Table \ref{tab:FirstFamily}.  And each connected component of $\bigcup_{deg(v) \geq 3} L_{v}$ is realizable by a minimal GS isolating block.

Notice that the intersection $$\left( \bigcup_{deg(v) < 3} L_{v} \right) \cap \left( \bigcup_{deg(v) \geq 3} L_{v} \right)$$ is made up of edges with weights $1$, $2$ and $3$. Since the choice of distinguished branched $1$-manifolds match for all these weights, the realizations of the connected components of these two unions of $L_{v}'s$ can all be glued together. Hence, $L$ is realizable.

\end{proof}

\subsection{Realization of GS graphs in the presence of bifurcation vertices  }

As we have seen in the last section, bifurcation vertices can constitute obstructions in the realizability of a GS graph. The simultaneous presence of bifurcation vertices in GS graphs  labelled with Whitney and  double crossing singularities may in some cases not be realizable. See Example  \ref{ex:NR}.

For this reason our next theorem will restrict the labelings of the vertices on a GS graph to singularities of type $\mathcal{R}$, $\mathcal{C}$ and $\mathcal{W}$, excluding singularities of type $\mathcal{D}$.

\begin{Teo}[{$\mathcal{R} \mathcal{C}\mathcal{W}$-Case}]\label{Teo:W} All GS graphs labelled only with singularities of type  $\mathcal{R}, \mathcal{C}$ and $\mathcal{W}$ are  realizable.
\end{Teo}

\begin{proof}

Consider the choice of distinguished branched $1$-manifolds 
for the  edges of  $L$ labelled with weights $w$ given in Table \ref{tab:FirstFamily}.

\begin{table}[!htb]
			\centering
			\resizebox{0.85\linewidth}{!}{
			\begin{tabular}{c|cccccccccccccccccc}
			 \cellcolor{gray!20} $w=$weight & \cellcolor{gray!20} & \cellcolor{gray!20} $w = 1$ & \cellcolor{gray!20} & \cellcolor{gray!20} & \cellcolor{gray!20} $w = 2$ & \cellcolor{gray!20} & \cellcolor{gray!20} & \cellcolor{gray!20} $w = 3$ & \cellcolor{gray!20} & \cellcolor{gray!20} & \cellcolor{gray!20} $w = 4$ & \cellcolor{gray!20} & \cellcolor{gray!20} &
			 \cellcolor{gray!20} $\ldots$ & \cellcolor{gray!20} & \cellcolor{gray!20} &
			 \cellcolor{gray!20} $w = n$ & \cellcolor{gray!20} \\
			  \hline
			   \cellcolor{gray!20} & &  &  &  &  &  &  &  &  &  &  & & & & &  &  &  \\
			   \cellcolor{gray!20} $N^{-}$ & & \includegraphics[align=c,scale=0.3]{Branched1.pdf} &  &  & \includegraphics[align=c,scale=0.3]{Branched2.pdf} &  &  & \includegraphics[align=c,scale=0.3]{Branched3b.pdf} &  &  & \includegraphics[align=c,scale=0.3]{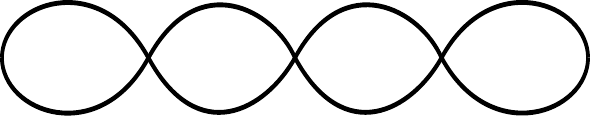} &  &  &  & & & $\overbrace{\includegraphics[align=c,scale=0.3]{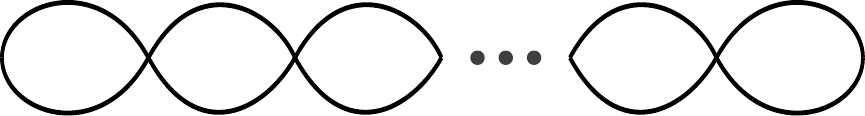}}$ &  \\
			   \multicolumn{19}{c}{}
			\end{tabular}}
			\caption{Distinguished branched $1$-manifolds with weight $w$}
			\label{tab:FirstFamily}
		\end{table}

Each GS semi-graph $L_{v}$ formed by a single vertex $v \in L$ and its incident edges can be realized from the realization of the corresponding minimal GS semi-graph $\ell_{v}$. Indeed, starting from a minimal GS isolating block $N_{0}$ with boundary components in Table \ref{tab:FirstFamily} which realizes $\ell_{v}$, one may perform a sequence of identifications of pairs of points in $N^{+}$, as shown in Figure \ref{fig:sequence2}, where these pairs of points are orbits in $N$ whose $\omega$-limit sets do not belong to $N$.

\begin{figure}[h] 
\begin{equation*}
  \begin{tikzcd}
   \includegraphics[align=c,scale=0.4]{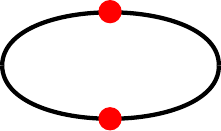} \hspace{0.4cm} \rightarrow \hspace{0.4cm} \includegraphics[align=c,scale=0.4]{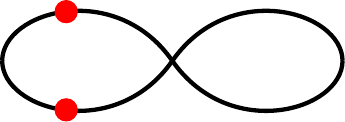} \hspace{0.4cm} \rightarrow \hspace{0.4cm} \includegraphics[align=c,scale=0.4]{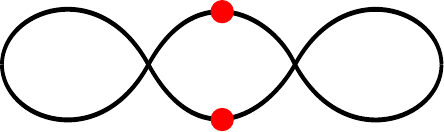} \hspace{0.4cm} \rightarrow \hspace{0.4cm} \includegraphics[align=c,scale=0.4]{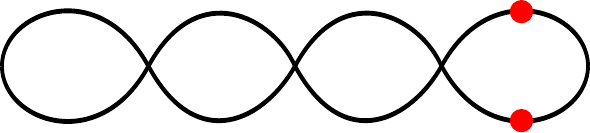} \hspace{0.4cm} \rightarrow  \hspace{0.4cm} \ldots
  \end{tikzcd}
\end{equation*}
\caption{Identification sequence}
\label{fig:sequence2}
\end{figure}

These identifications must be performed until the weights on the boundaries of the block match the weights on the edges of $L_{v}$. 

Notice that the topological effect of each identification on the boundary $N^{+}$ is to produce a branch point that determines a fold along which a cylinder is glued to the block. We remark that, for a vertex of degree $2$, the relative position of these cylinders with respect to the  minimal GS block may produce topologically non-equivalent  realizations of the same semi-graph.

Lastly, the realizations of all GS semi-graphs $L_{v}$ can be glued together due to the fact that all boundaries of weight $w$ have the same distinguished branched $1$-manifold.  
 
\end{proof}

\begin{Ex}\label{ex:RCW}
In  Figure \ref{fig:realization2}, we present  two  topologically non equivalent  realizations of a GS graph  satisfying Theorem  \ref{Teo:W}.

\begin{figure}[h] 
\centering
\begin{equation*}
  \begin{tikzcd}
   \includegraphics[align=c,scale=0.6]{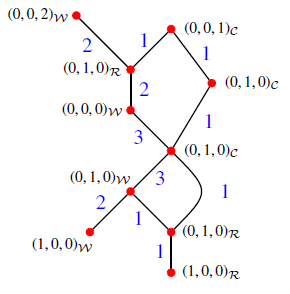} \xrightarrow{Realizations} \hspace{0.1cm} \includegraphics[align=c,scale=0.3]{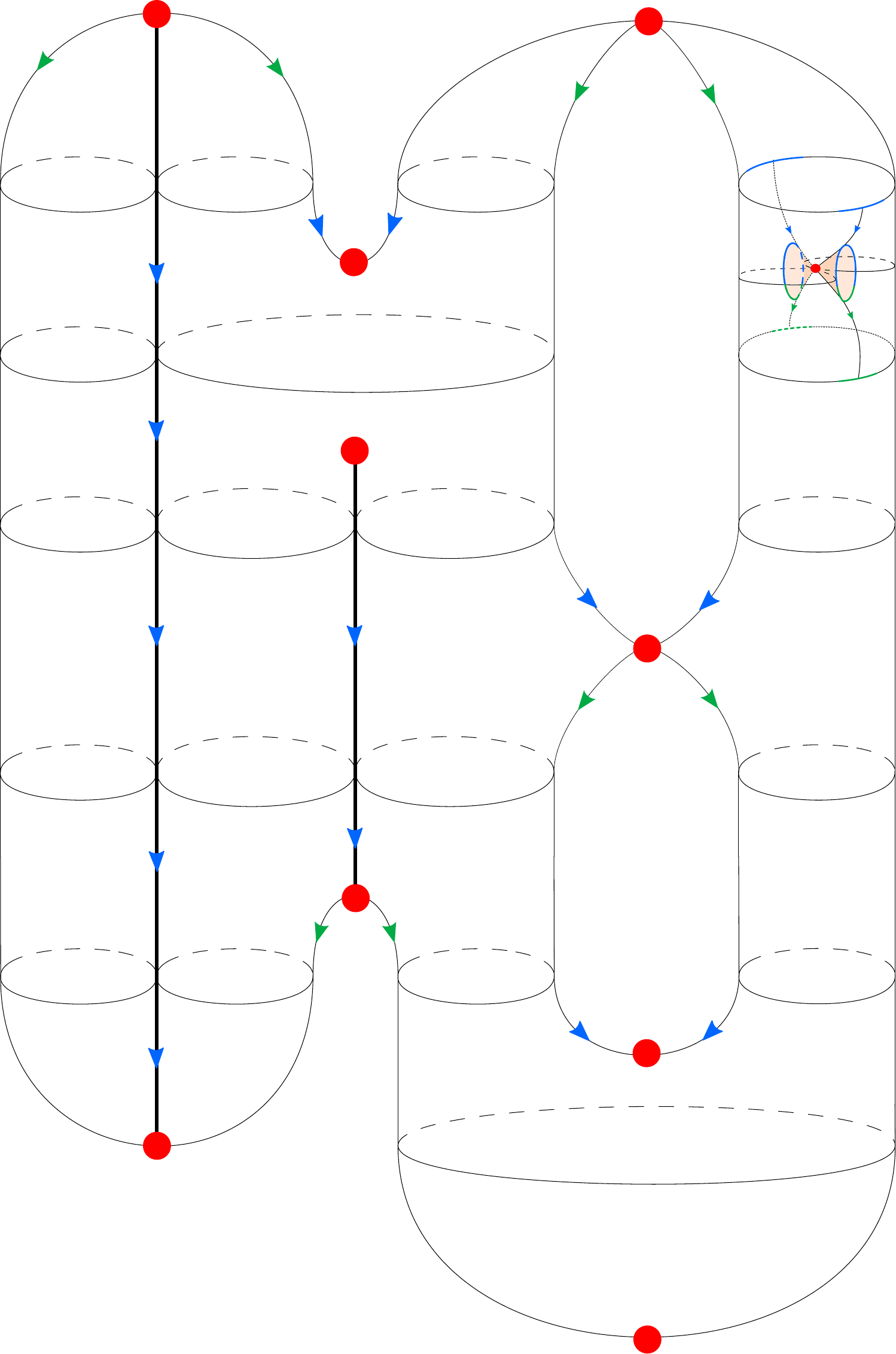} \hspace{0.5cm} \includegraphics[align=c,scale=0.3]{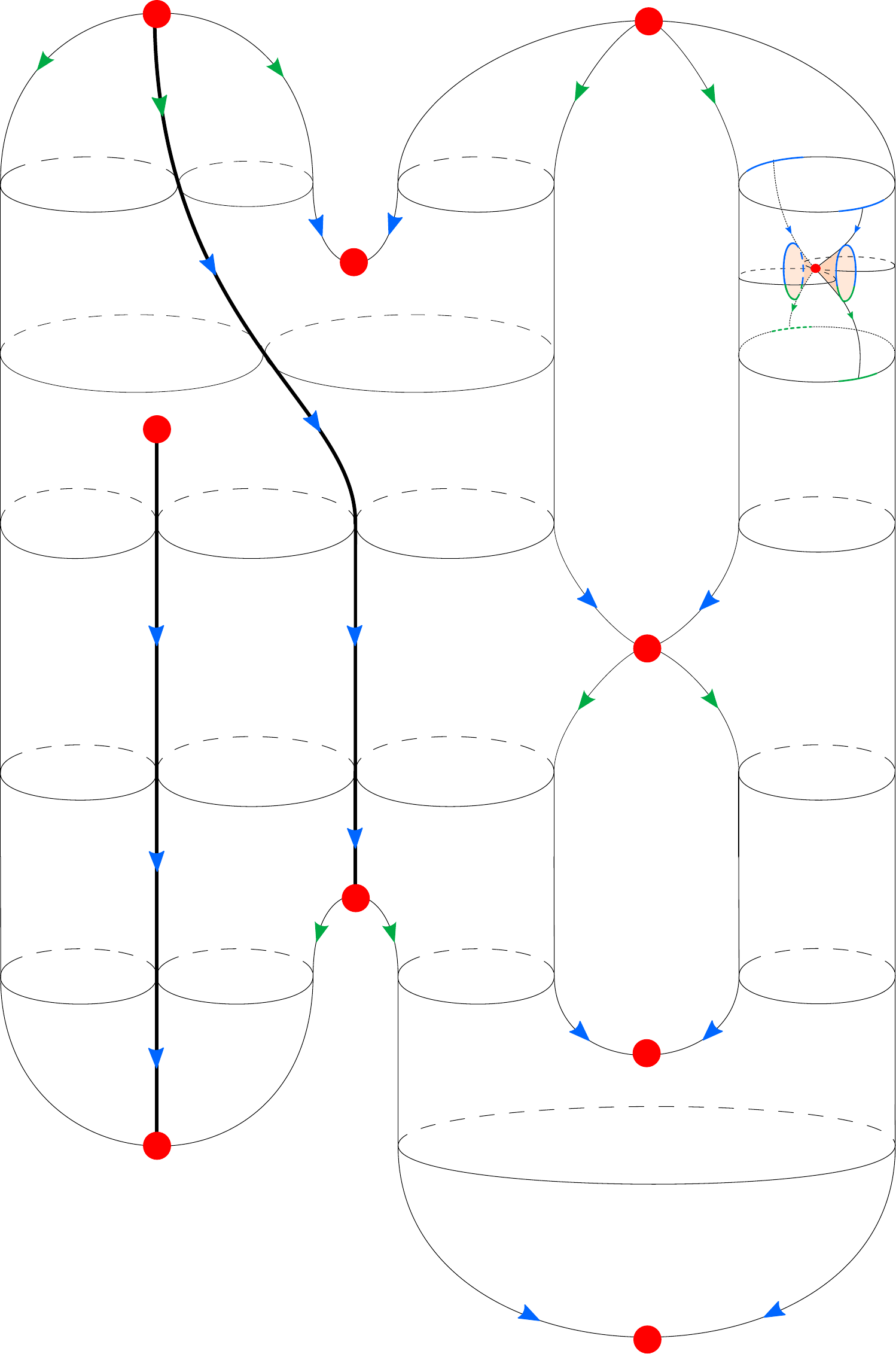}
  \end{tikzcd}
\end{equation*}
\caption{Realizations of a  GS graph as a GS flow}
\label{fig:realization2}
\end{figure}
\end{Ex}

Although Theorem \ref{Teo:W} was obtained by excluding the presence of vertices labelled with singularities of type $\mathcal{D}$, if we pay closer attention to the family of distinguished branched $1$-manifolds in Table \ref{tab:FirstFamily}, we notice that it is actually possible to realize almost every GS semi-graph with a single vertex labelled with a singularity of type $\mathcal{D}$ having this family of distinguished branched $1$-manifolds as boundary components, the exceptions being the GS semi-graph of degree $1$ labelled with natures $a$ or $r$, the GS semi-graph of degree $3$ labelled with natures $sa$ or $sr$, and the GS semi-graph of degree $5$ labelled with natures $ss_{s}$ or $ss_{u}$, as states the next Lemma.

\begin{Lem}\label{lem:FirstFamily}
Let $L_{v}$ be a GS semi-graph with a single vertex $v$ and its incident edges. Then, $L_{v}$ can be realized by the family of distinguished branched $1$-manifolds in Table \ref{tab:FirstFamily} if and only if $L_{v}$ satisfies:

\begin{itemize}

\item[i) ] $v$ is labelled with a singularity of type $\mathcal{R}, \mathcal{C}, \mathcal{W}$ or $\mathcal{D}$;

\item[ii) ] if $v$ has degree $1$, then its incident edge has weight equal to $1$ or $2$; 

\item[iii) ] if $v$ is labelled with a singularity of type $\mathcal{D}$ and nature $sa$ or $sr$, then $v$ has degree $2$;

\item[iv) ] $v$ has degree less or equal to $4$;

\item[v) ] if $v$ is labelled with a singularity of type $\mathcal{D}$ and nature $ss_{s}$ (resp. $ss_{u}$), with $e_{v}^{+} = 1, e_{v}^{-} = 2$ (resp. $e_{v}^{+} = 2, e_{v}^{-} = 1$), then $\{b_{1}^{-}, b_{2}^{-}\}$ (resp. $\{b_{1}^{+}, b_{2}^{+}\}$) is equal to $\{1, \mathcal{B}^{+} - 2\}$ (resp. $\{1, \mathcal{B}^{-} - 2\}$);

\item[vi) ] if $e_{v}^{-} = 3$ (resp. $e_{v}^{+} = 3$), then $\{b_{1}^{-}, b_{2}^{-}, b_{3}^{-}\}$ (resp. $\{b_{1}^{+}, b_{2}^{+}, b_{3}^{+}\}$) is equal to $\{1, 1, \mathcal{B}^{-} - 2\}$;

\end{itemize}

\end{Lem}

\begin{proof}

The proof relies heavily on the characterization of possible distinguished branched $1$-manifolds that make up the boundary components of minimal GS isolating blocks given in Theorem \ref{teo:colecao}.

$(\Rightarrow)$ First, suppose $L_{v}$ can be realized as a GS flow on a GS isolating block, which has as boundary components a distinguished branched $1$-manifolds of Table \ref{tab:FirstFamily}. Now we verify that $L_{v}$ satisfies conditions $i)$ through $vi)$:

\begin{itemize}

\item[i) ] Theorem \ref{teo:colecao} states that there is no minimal GS isolating block for a singularity of type $\mathcal{T}$ with the choice of distinguished branched $1$-manifolds in Table \ref{tab:FirstFamily} as boundary components. Therefore, there cannot be GS isolating blocks with passageways for singulatiries of type $\mathcal{T}$ with this choice either. Hence, $v$ must be labelled with a singularity of type $\mathcal{R}, \mathcal{C}, \mathcal{W}$ or $\mathcal{D}$;

\item[ii) ] If $v$ has degree $1$, then the possible labels on $v$ are either attracting or repelling singularities of type $\mathcal{R}, \mathcal{W}$ or $\mathcal{D}$. However, it follows from Theorem \ref{teo:colecao} that a realization of a GS isolating block for an attracting or repelling singularity of type $\mathcal{D}$ cannot be achieved with the distinguished branched $1$-manifold of weight $3$ pictured in Table \ref{tab:FirstFamily}. Thus, $v$ is either labelled with an attracting or repelling singularity of type $\mathcal{R}$ or $\mathcal{W}$, which implies weight $1$ or $2$ on its incident edge, respectively;

\item[iii) ] The possible degrees of a vertex $v$ labelled with a singularity of type $\mathcal{D}$ and nature $sa$ or $sr$ are $2$ or $3$. However, according to Theorem \ref{teo:colecao}, only the cases where $v$ has degree $2$ admit realizations matching the choice of distinguished branched $1$-manifolds in Table \ref{tab:FirstFamily};

\item[iv) ] If we suppose by contradiction that the degree of $v$ is greater than $4$, than $v$ must have degree $5$, which implies that $v$ is labelled with a singularity of type $\mathcal{D}$ and nature $ss_{s}$ or $ss_{u}$. Either way, Theorem \ref{teo:colecao} implies that such $L_{v}$ do not admit a realization matching the distinguished branched $1$-manifolds of Table \ref{tab:FirstFamily}. This follows since the distinguished branched $1$-manifold of weight $3$ on Table \ref{tab:FirstFamily} cannot be chosen as the boundary of a minimal GS isolating block for such singularities. Furthermore, adding passageways to their minimal isolating blocks produce boundaries that also do not belong to Table \ref{tab:FirstFamily}. This contradicts our initial hypothesis that $L_{v}$ admits a realization with boundary components within Table \ref{tab:FirstFamily}. Hence, $v$ must have degree less or equal to $4$;

\item[v) ] The realization of $L_{v}$, in the case that $v$ is labelled with a singularity of type $\mathcal
D$ and nature $ss_{s}$, with minimal weights on its incident edges, is given in Table \ref{doubleSSsblocks}. We recall that turning the block upside down gives a minimal GS isolting block for nature $ss_{u}$. The minimal GS isolating block shows that, even though $e_{v}^{-} = 2$ (resp. $e_{v}^{+} = 2$), in order to create a GS isolating block with passageways and boundaries matching Table \ref{tab:FirstFamily}, the folds must exit the block through only one of the exiting (resp. entering) boundaries. Thus the weights $\{b_{1}^{-}, b_{2}^{-}\}$ (resp. $\{b_{1}^{+}, b_{2}^{+}\}$) must be equal to $\{1, \mathcal{B}^{+} - 2\}$ (resp. $\{1, \mathcal{B}^{-} - 2\}$);

\item[vi) ] Analogous to item $v)$;

\end{itemize}

$(\Leftarrow)$ Suppose $L_{v}$ satisfies conditions $i)$ through $vi)$ and denote by $\ell_{v}$ the minimal GS semi-graph associated to $L_{v}$, that is, it has the same labelling, indegree and outdegree, but with minimal weights on its incident edges satisfying the Poincaré-Hopf condition. Then, by Theorem \ref{teo:colecao}, $\ell_{v}$ admits a realization as a GS flow on a minimal GS isolating block $N$ with boundary components given by the distinguished branched $1$-manifolds in Table \ref{tab:FirstFamily}. One can identify orbits of $N$ as in Figure \ref{fig:sequence2}, such that the topological effect of each identification produces a branch point that determines a fold along which a cylinder is glued to the block.  One can perform a finite number of such identifications until the weights on the boundaries of $N$ become equal to the respective weights on the edges of $L_{v}$. Moreover, the distinguished branched $1$-manifold produced by each of these identifications belong to Table \ref{tab:FirstFamily}, thus completing the proof.

\end{proof}

It turns out that attractors, repellers and other singularities of type $\mathcal{D}$ that do not satisfy Lemma \ref{lem:FirstFamily} may obstruct the realization of a GS graph, even in the case where all folds have singularities of type $\mathcal{D}$ as their $\alpha, \omega$ limit set.

Recall that an isolating block with passageways arises from the identification of specific orbits on a minimal isolating block. Since the identification of orbits may lower the number of connected components on the boundaries of the block, in some cases, we must choose orbits that enter and exit the block through the same connected components. 

We remark that this fact is one of the causes that may prevent the gluing of two blocks. Indeed, if two blocks have a different number of connected components, it may happen that the orbits that must be identified, in order to produce the same distinguished branched $1$-manifold, may {enter or exit the block through different connected components}, preventing the realization. See Example \ref{ex:Nosplit} below. It is worth mentioning that this scenario never occurs in a GS graph labelled with singularities of type $\mathcal{R}, \mathcal{C}$ and $\mathcal{W}$ as proved in Theorem \ref{Teo:W}. And also does not occur whenever bifurcation vertices are not present, as proved in Theorem \ref{teo:linear}.

\begin{Ex}[Non-realizable]\label{ex:Nosplit} In the example presented in  Figure  \ref{fig:Nrealization} (a), the GS graph which is labelled with singularities of type   $\mathcal{R}$ and $\mathcal{D}$, is non-realizable although  satisfying the Poincaré-Hopf condition on all edges. The obstruction arises since there is no common choice of distinguished branched $1$-manifolds of weight $3$ that permits its realization. 
More specifically, all vertices of $L$ admit a unique realization as a GS flow on a GS isolating block, but the distinguished branched $1$-manifolds of weight $3$ of these realizations are not homeomorphic, thus preventing the gluing of all blocks. On the other hand, if we consider homeomorphic distinguished branched $1$-manifolds of weight $3$, note that the weights on the exiting boundaries of the GS isolating block for a regular saddle do not match the weights on $L$, see  Figure \ref{fig:Nrealization} (b).

\begin{figure}[htb] 
\centering
\begin{minipage}{0.45\textwidth}
$(a)$ \\
%\begin{equation*}
%%  \begin{tikzcd}
\includegraphics[align=c,scale=0.6]{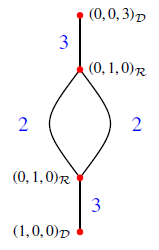}  \hspace{1cm} \includegraphics[align=c,scale=0.3]{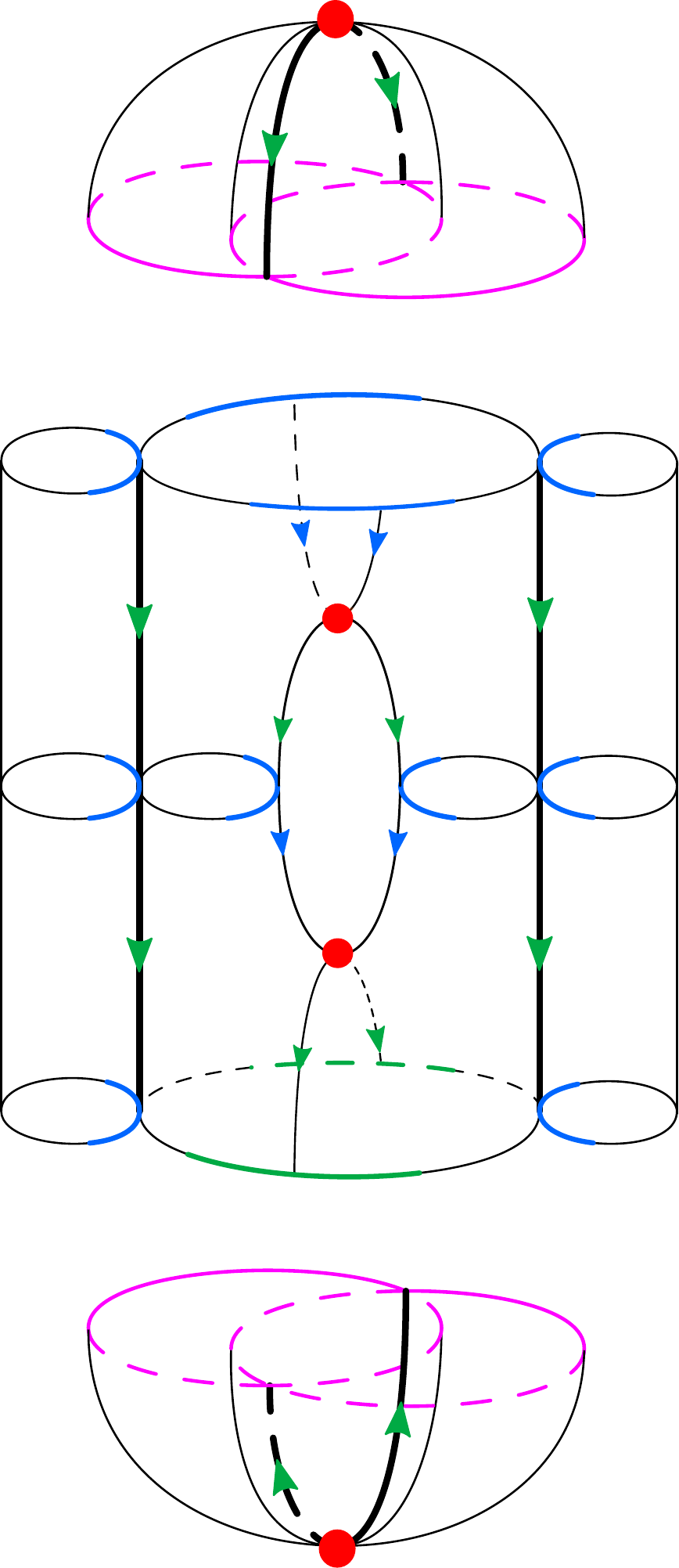} \\
%\caption{Non-realizable GS graph}
%\label{fig:Nrealization}%\hspace{2cm}  \includegraphics[align=c,scale=0.3]{Fblocks.pdf}
\end{minipage}
\hfill\vline\hfill
%%%%%%%%%%%%%%%%%%%%%%%%%%%%%%%%%%%%%%%%%%%%%%%%%
\begin{minipage}{0.45\textwidth}
$(b)$ \\
%  \end{tikzcd}
%\end{equation*}
%\caption{(a) Non-realizable GS graph;  (b) Regular saddle isolating  block with $N^{+}$ of weight $3$}
\begin{tabular}{cc}
& \includegraphics[align=c,scale=0.3]{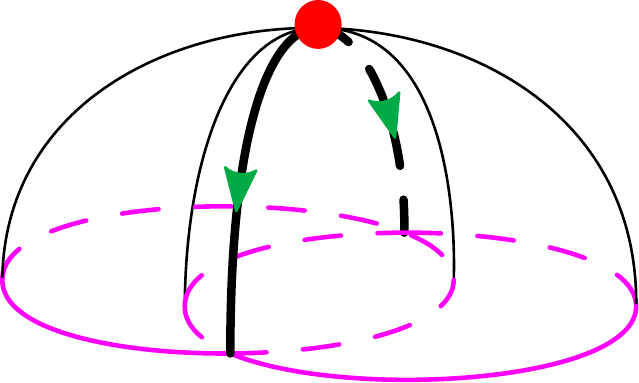} \\
& \\
\includegraphics[align=c,scale=0.5]{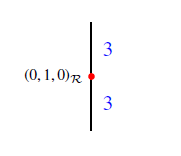} & \includegraphics[align=c,scale=0.3]{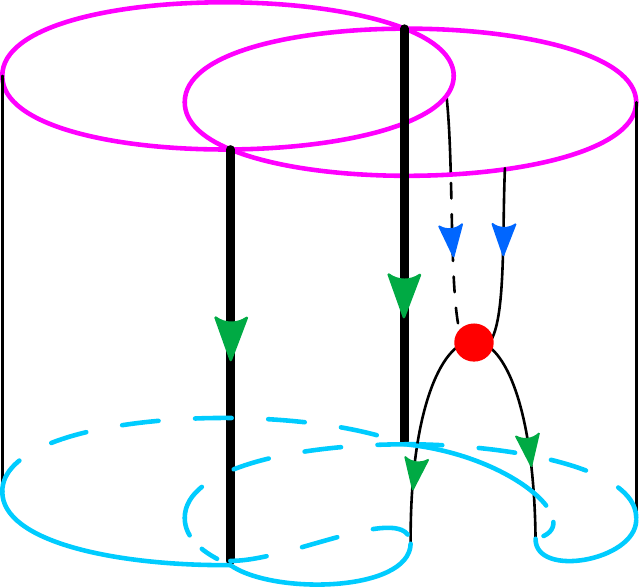} \\
& \\
\includegraphics[align=c,scale=0.5]{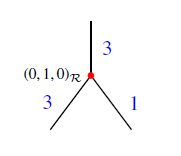} & \includegraphics[align=c,scale=0.3]{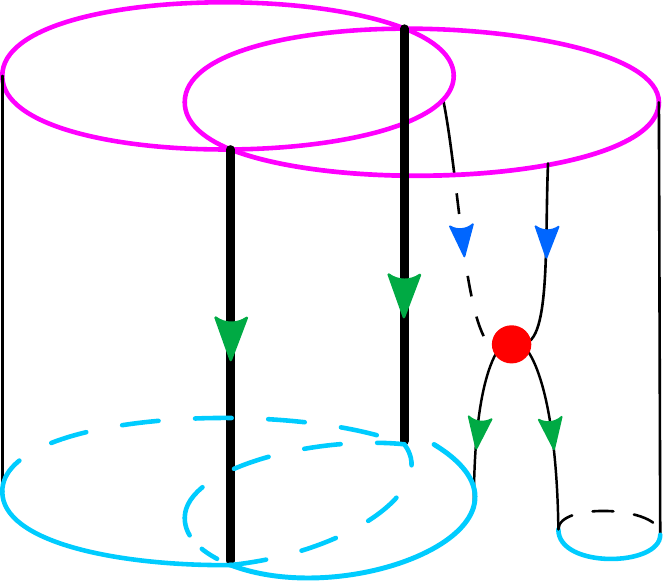} \\
& \\
\end{tabular}
%\caption{Regular saddle isolating  block with $N^{+}$ of weight $3$}
\end{minipage}
\caption{(a) Non-realizable GS graph;  (b) Regular saddle isolating  block with $N^{+}$ of weight $3$}
\label{fig:Nrealization}
\end{figure}

\end{Ex}

As seen in Example \ref{ex:Nosplit}, bifurcation vertices without the minimal weight condition may not be realizable.
The problem in this example resides in the fact that there are two non homeomorphic distinguished branched $1$-manifolds of weight $3$ that form the boundary of two isolating blocks that must be glued to each other. In order to obtain a result for bifurcation vertices in general, we must restrict the allowable boundary components that make up the isolating blocks.

In what follows, we will consider the family of distinguished branched $1$-manifolds in Table \ref{tab:bordosAAA} and analyze under what conditions bifurcation vertices and its incident edges can be realized with this family of distinguished branched $1$-manifolds. This is the content of the next Lemma.

\begin{Lem}\label{lem:FamilyB}
Let $L_{v}$ be a GS semi-graph with a single vertex $v$ and its incident edges, belonging to Table \ref{tab:bordos-colecao}, such that $v$ is labelled with a singularity of type $\mathcal{R}, \mathcal{C}, \mathcal{W}$ or $\mathcal{
D}$. Then, $L_{v}$ can be realized by the family of distinguished branched $1$-manifolds in Table \ref{tab:bordosAAA} if and only if $L_{v}$ satisfies the following conditions, according to the singularity with which $v$ is labelled:

\begin{itemize}

\item[i) ] (regular or double crossing) if $e_{v}^{\pm} = 1, \ e_{v}^{\mp} = 2, \ |\mathcal{B^{+}} - \mathcal{B}^{-}| = 1$ and $\mathcal{B}^{\pm}$ is odd, then $b_{1}^{\mp}, b_{2}^{\mp}$ are odd;

\item[ii) ] (Whitney) if $e_{v}^{\pm} = 1$ and $e_{v}^{\mp} = 2$, then $\mathcal{B}^{\pm}$ is even and $\{b_{1}^{\mp}, b_{2}^{\mp}\}$ is equal to $\{1, \mathcal{B}^{\pm} - 1\}$;

\item[iii) ] (double crossing) if $e_{v}^{+} = e_{v}^{-} = 2$ and $\mathcal{B}^{\pm} > \mathcal{B}^{\mp}$, then the weights $b_{1}^{\pm}, b_{2}^{\pm}$ are even, and $\{b_{1}^{\mp}, b_{2}^{\mp}\}$ is either equal to $\{1, \mathcal{B}^{\pm} - 3 \}$ or $\{b_{1}^{\pm} - 1, b_{2}^{\pm} - 1\}$;

\item[iv) ] (double crossing) if $e_{v}^{\mp} = 3$, then $b_{i}^{\mp} = 1$ for at least one index $i \in \{1,2,3\}$. Moreover, if $\mathcal{B}^{\pm}$ is odd, then the weights $b_{1}^{\mp}, b_{2}^{\mp}, b_{3}^{\mp}$ are all odd;

\item[v) ] (double crossing) If $e_{v}^{\mp} = 4$, then $b_{i}^{\mp} = 1$ for at least two edges. Moreover, if $\mathcal{B}^{\pm}$ is odd, then the weights $b_{1}^{\mp}, b_{2}^{\mp}, b_{3}^{\mp}, b_{4}^{\mp}$ are all odd;

\end{itemize}

\end{Lem}

\begin{proof}

The proof relies heavily on Lemma \ref{lem:passageway_minimal}, which states that all GS isolating blocks with passageways for GS singularities can be constructed from minimal GS isolating blocks via identifications of pair of orbits.

$(\Rightarrow)$ First, suppose $L_{v}$ can be realized {as a GS flow on} a GS isolating block $N$ with the distinguished branched $1$-manifolds of Table \ref{tab:bordosAAA} as boundary components. Now we verify that $L_{v}$ satisfies conditions $i)$ through $v)$:

\begin{itemize}

\item[i) ] 
{Let $W^{s}(p)$ (resp. $W^{u}(p)$) be the stable (resp. unstable) manifold of the singularity $p \in N$. Then, whether $p$ is a singularity of type $\mathcal{R}$ or $\mathcal{D}$, it follows that $W^{s}(p) \cap N^{+}$ (resp. $W^{u}(p) \cap N^{-}$) is formed by two points $q_{1}$ and $q_{2}$. In the case $\mathcal{B}^{\pm}$ is odd, we have that each connected component of $N^{\pm} \setminus \{q_{1}, q_{2}\}$ has an even number of branch points. Since all folds associated to branch points  on the same connected component of $N^{+} \setminus \{q_{1}, q_{2}\}$ (resp. $N^{-} \setminus \{q_{1}, q_{2}\}$) must exit (resp. enter) $N$ through the same connected component of $N^{-}$ (resp. $N^{+}$), it follows that each connected component of $N^{-}$ (resp. $N^{+}$) has an even number of branch points. Hence, we must have $b_{1}^{\mp}$ and $b_{2}^{\mp}$ both odd.}

\item[ii) ]
{If $N$ is a minimal GS isolating block, then $\mathcal{B}^{\pm} = 2$ and $b_{1}^{\mp} = b_{2}^{\mp} = 1$. Hence, $\{b_{1}^{\mp}, b_{2}^{\mp}\}$ is equal to $\{1, \mathcal{B}^{\pm}-1\}$. Moreover, $N^{\pm}$ is a figure eight. Denote by $q$ the branch point of $N^{\pm}$, we have that $N^{\pm} \setminus \{q\}$ has two connected components such that orbits that enter (resp. exit) $N$ through each connected component of $N^{\pm} \setminus \{q\}$ exit (resp. enter) $N$ through  different connected components of $N^{\mp}$.}

{If $N$ is not minimal, Lemma \ref{lem:passageway_minimal} asserts that $N$ can be constructed from a minimal isolating block $N_{0}$. Denote by $q$ the branch point of $N_{0}^{\pm}$. Since we cannot identify orbits that enter (resp. exit) $N_{0}^{\pm}$ through different connected components of $N_{0}^{\pm} \setminus \{q\}$, it follows that $N$ also has a branch point $\tilde{q}$ such that $N^{\pm} \setminus \{\tilde{q}\}$ has two connected components.} 

{Note that every distinguished branched $1$-manifold of Table \ref{tab:bordosAAA} with odd weight remains connected after the removal of one branch point. Thus, $\mathcal{B}^{\pm}$ cannot be odd. Hence, $\mathcal{B}^{\pm}$ is even. Furthermore, after removing the branch point that disconnects a distinguished branched $1$-manifold of Table \ref{tab:bordosAAA} with even weight, the remaining $\mathcal{B}^{\pm} -2$ branch points all lie in the same connected component. Thus, all folds associated to these branch points exit $N^{\mp}$ through the same connected component. Therefore, we have that $\{b_{1}^{\mp}, b_{2}^{\mp}\}$ is equal to $\{1, \mathcal{B}^{\pm} - 1\}$. }

\item[iii) ] If $L_{v}$ is minimal, then {$\mathcal{B}^{\pm} = 4$, where $b_{1}^{\pm} = b_{2}^{\pm} = 2$ and $\mathcal{B}^{\mp} = 2$, with $b_{1}^{\mp} = b_{2}^{\mp} = 1$. Hence, $b_{1}^{\pm}$ and $b_{2}^{\pm}$ are even and $\{b_{1}^{\mp}, b_{2}^{\mp}\}$ is equal to  $\{1, \mathcal{B}^{\pm} - 3\}$  which coincides with $\{b_{1}^{\pm} - 1, b_{2}^{\pm} - 1\}$. Moreover,} $N^{\pm} = N_1^{\pm}\cup N_2^{\pm}$ is the disjoint union of two figures eight. Denote by $q_j$  the branch point of $N^{\pm}_j$  for $j\in\{1,2\}$.    As in item $ii)$, we have that $N_j^{\pm} \setminus \{q\}$ has two connected components such that orbits that enter (resp. exit) $N$ through each connected component of $N_j^{\pm} \setminus \{q\}$ exit (resp. enter) $N$ through  different connected components of $N^{\mp}$.
Also, by the same argument used in item $ii)$, one shows that if $L_{v}$ is not minimal, then $b_{1}^{\pm}$ and $  b_{2}^{\pm}$ are even.

{Now consider the case where}  $L_{v}$ has weights $b_{1}^{\pm} >2$ and $ b_{2}^{\pm}>2$, i.e., $N^{\pm} = N^{\pm}_1 \cup N^{\pm}_2$  is the disjoint union  of two  distinguished branched   $1$-manifolds  of weights   $b_{1}^{\pm}$ and $ b_{2}^{\pm}$, respectively. {Then it follows}  that
 $N$ has $(b_{1}^{\pm}-2) +(b_{2}^{\pm}-2)$ passageways. By using a two step process, consider the $b_{1}^{\pm}-2$ passageways determined by $N_{1}^{\pm}$.   As in item $ii)$ the folds that  {intersect} $N_{1}^{\pm}\setminus\{p\}$ must exit (resp. enter) either $N_{1}^{\mp}$ or $N_{2}^{\mp}$. Without loss of generality, we can assume that these orbits exit (resp., enter) $N_{1}^{\mp}$.  At this point, this implies that the weight of  $N_{1}^{\mp}$  is greater or equal to $b_{1}^{\pm} -1$. Now consider  the  folds  that  {intersect} $N_{2}^{\pm}\setminus\{q\}$ which must exit (resp., enter) either through $N_{1}^{\mp}$ or $N_{2}^{\mp}$. If they exit (resp., enter) through $N_{1}^{\mp}$  we are in the case  that $\{b_{1}^{\mp}, b_{2}^{\mp}\}$ is equal to $\{ \mathcal{B}^{\pm} - 3 , 1\}$; 
 if they exit (resp., enter) through $N_{2}^{\mp}$  we are in the case  that $\{b_{1}^{\mp}, b_{2}^{\mp}\}$ is equal to  $\{b_{1}^{\pm} - 1, b_{2}^{\pm} - 1\}$.

\item[iv) ] 
{If $N$ is a minimal GS isolating block, then $\mathcal{B}^{\pm} = 3$ and $b_{1}^{\mp} = b_{2}^{\mp} = b_{3}^{\mp} = 1$. Moreover, $N^{\pm}$ has two branch points $q_{1}$ and $q_{2}$ such that $N^{\pm} \setminus \{q_{1}, q_{2}\}$ has four connected components. In addition, orbits that enter (resp. exit) $N$ through two of these connected components exit (resp. enter) $N$ through the same connected component of $N^{-}$ (resp. $N^{+}$), while orbits entering (resp. exiting) $N$ through the other two connected components of $N^{\pm} \setminus \{q_{1}, q_{2}\}$ exit (resp. enter) each through a distinct connected component of the remaining ones in $N^{\mp}$.}

{In the case $N$ is not minimal, let $W^{s}(p)$ (resp. $W^{u}(p)$ be the stable (resp. unstable) manifold of the singularity $p \in N$. Then, $W^{s}(p) \cap N^{+}$ (resp. $W^{u}(p) \cap N^{-}$) is also formed by two branch points $\{\tilde{q_{1}}, \tilde{q_{2}}\}$ such that $N^{\pm} \setminus \{\tilde{q_{1}}, \tilde{q_{2}}\}$ has the same properties of $N_{0}^{\pm} \setminus {q_{1}, q_{2}}$ described in the previous paragraph for a minimal isolating block $N_{0}$. Since all branch points other than $\tilde{q_{1}}$ and $\tilde{q_{2}}$ belong to at most two connected components of $N^{\pm} \setminus \{\tilde{q_{1}}, \tilde{q_{2}}\}$ and $N^{\mp}$ has three connected components, it follows that at least one connected component of $N^{\mp}$ has weight $1$, that is, $b_{i}^{\mp} = 1$ for at least one $i \in \{1, 2, 3\}.$ Furthermore, if $\mathcal{B}^{\pm}$ is odd, each connected component of $N^{\pm} \setminus \{\tilde{q_{1}}, \tilde{q_{2}}\}$ has an even number of branch points. Hence, $b_{1}^{\mp}, b_{2}^{\mp}$ and $b_{3}^{\mp}$ are all odd. }

\item[v) ] In this case, the orbits through {each one of} the four arcs formed by $N_0^{\pm} \setminus \{p, \ q\}$  exit through a distinct connected component of $N^{\mp}$. The remaining argument is analogous to item $iv)$.

\end{itemize}

$(\Leftarrow)$ Now, we show that all semi-graphs $L_{v}$ of Table \ref{tab:bordos-colecao} subject to conditions $i), ii), iii), iv), v)$ of Lemma \ref{lem:FamilyB} are realizable as a GS flow on a GS isolating block with boundary components in Table \ref{tab:bordosAAA}.

Note that, if $L_{v}$ is minimal, then Theorem \ref{teo:colecao} asserts that there is a minimal GS isolating block with boundary components in Table \ref{tab:bordosAAA} that realizes $L_{v}$. Some of these blocks can be seen in Tables \ref{NEWconeblocks}, \ref{NEWwhitneyblocks}, \ref{doubleSAblocks} and \ref{doubleSSsblocks}.

If $L_{v}$ is not minimal, we proceed as follows:

First, we consider a minimal GS isolating block $N_{0}$ for a singularity $p$, such that the Conley index of $p$ is the same that which $v$ is labelled with, $N_{0}^{\pm}$ has $e_{v}^{\pm}$ connected components, and $N_{0}^{\pm}$ belongs to Table \ref{tab:bordosAAA}.

Then in order to create passageways in $N_{0}$, we analyze the possible ways of embedding $N_{0}^{+}$ in a distinguished branched $1$-manifold $N^{+}$ of Table \ref{tab:bordosAAA} with weight $\mathcal{B}^{+}$, such that the folds  {associated to} branch points of $N^{+} \setminus N_{0}$ exit the connected components of $N_{0}^{-}$ according to the weights $b_{i}^{-}$ of $L_{v}$, for $i \in \{1, \ldots, e_{v}^{-}\}$.

For this purpose, we divide the semi-graphs $L_{v}$ of Table \ref{tab:bordos-colecao} in the following cases: 

\begin{itemize}

\item[a) ] if $e_{v}^{+} = e_{v}^{-} = 1$, Lemma \ref{lem:FamilyB} imposes no restrictions on $L_{v}$. %So, let $N^{+}$ be the distinguished branched $1$-manifold of Table \ref{tab:bordosAAA} with weight $\mathcal{B}^{+}$. Next,
So, suppose $v$ is labelled with a singularity of type $\mathcal{R}, \mathcal{C}$ or $\mathcal{D}$ and consider {any} embedding $f: N_{0}^{+} \hookrightarrow N^{+}$. Then consider a tubular flow on $\overline{N^{+} \setminus f(N_{0}^{+})} \times [0,1]$ and for each point $q_{i} \in \overline{N^{+} \setminus f(N_{0}^{+})} \cap f(N_{0}^{+})$, identify the orbit of $\overline{N^{+} \setminus f(N_{0}^{+})} \times [0,1]$ that passes through in $q_{i}$ with the orbit of $N_{0}$ that passes through in $f^{-1}(q_{i})$. The resulting block $N$ is a realization of $L_{v}$ such that the boundary components of $N$ belong to Table \ref{tab:bordosAAA}.

The case where $v$ is labelled with a singularity of type $\mathcal{W}$ requires an intermediate step before realizing the embedding $f: N_{0}^{+} \hookrightarrow N^{+}$ only if $N^{+}$ has odd weight. This is due to the fact that $N_{0}^{+}$ is a figure eight and it cannot be embedded in a distinguished branched $1$-manifold of Table \ref{tab:bordosAAA} of odd weight. Hence, the  identification {shown in Figure \ref{fig:priorW}} must be performed prior to the embedding.

\begin{figure}[h] 
\begin{equation*}
  \begin{tikzcd}
   \includegraphics[align=c,scale=0.6]{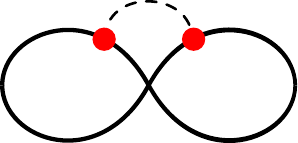} \hspace{0.4cm} \rightarrow \hspace{0.4cm} \includegraphics[align=c,scale=0.6]{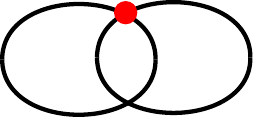} \hspace{0.4cm} 
  \end{tikzcd}
\end{equation*}
\caption{Creating a branch point from a figure eight.}
\label{fig:priorW}
\end{figure}

\item[b) ] if $1 = e_{v}^{+} < e_{v}^{-}$, then $L_{v}$ is subject to one of the conditions $i), ii), iv)$ or $v)$ of Lemma \ref{lem:FamilyB}. 

Let $W^{s}(p)$ be the stable manifold of the singularity $p \in N_{0}$. Note that, if $p$ is a singularity of type $\mathcal{R}$, $W^{s}(p) \cap N_{0}^{+}$ is the disjoint union of two points. In the case $p$ is a Whitney or double crossing singularity, $W^{s}(p) \cap N_{0}^{+}$ is exactly the branch points in $N_{0}^{+}$, which is equal to $1$ or $2$ branch points, respectively. By removing these points from $N_{0}^{+}$ we are left with two connected components in the case $p$ is of type $\mathcal{R}$ or $\mathcal{W}$, and four connected components in the case $p$ is of type $\mathcal{D}$. The orbits that pass through each of these connected components exit $N_{0}$ through a connnected component of $N_{0}^{-}$.

With this at hand, consider an embedding $f: N_{0}^{+} \hookrightarrow N^{+}$ such that, for each $j \in \{1, \ldots, e_{v}^{-}\}$ there is a connected component $K_{j}^{+}$ of $N^{+} \setminus (f(W^{s}(p) \cap N_{0}^{+})$ with $b_{j}^{-} - 1$ branch points. Now, consider a tubular flow on each connected component of $\overline{N^{+} \setminus f(N_{0}^{+})} \times [0,1]$. For each point $q_{i} \in \overline{N^{+} \setminus f(N_{0}^{+})} \cap f(N_{0}^{+})$, identify the orbit of $\overline{N^{+} \setminus f(N_{0}^{+})} \times [0,1]$ that passes through in $q_{i}$ with the orbit of $N_{0}$ that passes through in $f^{-1}(q_{i})$. The resulting block $N$ is a realization of $L_{v}$ such that the boundary components of $N$ belong to Table \ref{tab:bordosAAA}.

\item[c) ] {if $e_{v}^{+} = 2, e_{v}^{-} = 1$, then Lemma \ref{lem:FamilyB} imposes no restrictions on $L_{v}$. In this case, $N_{0}^{+}$ is the disjoint union of two figures eight, i.e., $N_{0}^{+} = F^{8}_{1} \cup F^{8}_{2}$. }

{Let $N^{+} = N_{1}^{+} \cup N_{2}^{+}$ be a distinguished branched $1$-manifold such that $N_{i}^{+}$ belongs to Table \ref{tab:bordosAAA} and has weight $b_{i}^{+}$ for $i \in \{1, 2\}$. In the case that the weights $b_{1}^{+}$ and $b_{2}^{+}$ are both odd, or in the case one of them is odd and the other is an even number greater than $2$, we must perform the identification shown in Figure \ref{fig:priorW} to $F_{1}^{8}$ and $F_{2}^{8}$ prior to the embeddings $f_{1}: F_{1}^{8} \hookrightarrow N_{1}^{+}$ and $f_{2}: F_{2}^{8} \hookrightarrow N_{2}^{+}$. Then one considers the same tubular flow and identifications of orbits of items $a)$ and $b)$. The resulting block $N$ is a realization of $L_{v}$ such that the boundary components of $N$ belong to Table \ref{tab:bordosAAA}. }

\item[d) ] if $e_{v}^{+} = e_{v}^{-} = 2$, then consider $N^{+} = N_{1}^{+} \cup N_{2}^{+}$ such that $N_{i}^{+}$ belongs to Table \ref{tab:bordosAAA} and has weight $b_{i}^{+}$, for $i \in \{1, 2\}$. According to Table \ref{tab:bordos-colecao}, $v$ can be labelled with a singularity of type $\mathcal{C}$ or $\mathcal{D}$. 

If $v$ is labelled with a cone singularity, $N_{0}^{+}$ is the disjoint union of two circles, i.e. $N_{0}^{+} = S^{1}_{1} \cup S^{1}_{2}$ and Lemma \ref{lem:FamilyB} imposes no restrictions on $L_{v}$. Thus, we proceed as in case $a)$ for $f_{1}: S^{1}_{1} \hookrightarrow N_{1}^{+}$ and for $f_{2}: S^{1}_{2} \hookrightarrow N_{2}^{+}$.

Otherwise, if $v$ is labelled with a singularity of type $\mathcal{D}$,  $N_{0}^{+}$ is the disjoint union of two figures eight, i.e. $N_{0}^{+} = F^{8}_{1} \cup F^{8}_{2}$ and $L_{v}$ is subject to condition $iii)$ of Lemma \ref{lem:FamilyB}. Thus, we proceed as in case $b)$ for $f_{1}: F^{8}_{1} \hookrightarrow N_{1}^{+}$ and for $f_{2}: F^{8}_{2} \hookrightarrow N_{2}^{+}$.

Hence, the resulting block $N$ is a realization of $L_{v}$ such that the boundary components of $N$ belong to Table \ref{tab:bordosAAA}.

\end{itemize}

Cases $a), \ b)$ , $c)$ and $d)$ prove that $L_{v}$ is realizable for all semi-graphs in Table \ref{tab:bordos-colecao} subject to $i), \ ii), \ iii), \ iv)$ and $v)$.

\end{proof}

The next theorem is a realization theorem which includes bifurcation vertices and labellings with both singularities of type $\mathcal{W}$ and $\mathcal{D}$, and as such will be a generalization of Theorems \ref{teo:linear} and \ref{Teo:W}.

\begin{Teo}\label{teo:families}
Let $L$ be a GS graph labelled with singularities of type $\mathcal{R}, \mathcal{C}, \mathcal{W}$ and $\mathcal{D}$ such that each GS semi-graph $L_{v}$ belongs to Table \ref{tab:bordos-colecao}. Furthermore, suppose that:

\begin{itemize}

\item[i) ] $L_{v}$ satisfies conditions $i), ii), iii), iv), v), vi)$ of Lemma \ref{lem:FirstFamily}, for all vertices $v \in L$; or 
\item[ii) ] $L_{v}$ satisfies  conditions 
$i), ii), iii), iv), v)$ of 
Lemma \ref{lem:FamilyB}, for all vertices $v \in L$, %with degree greater or equal to $3$. 
\end{itemize}
then $L$ is realizable.

\end{Teo}

\begin{proof}

In the case condition $i)$ is satisfied, it follows from Lemma \ref{lem:FirstFamily} that all semi-graphs $L_{v}$ can be realized as a GS flow on a GS isolating block with the distinguished branched $1$-manifolds of Table \ref{tab:FirstFamily} as boundary components. Hence, one glues these isolating blocks according to $L$.

Similarly, if condition $ii)$ is satisfied, then all semi-graphs $L_{v}$ can be realized as GS flows on GS isolating blocks with boundary components on Table \ref{tab:bordosAAA}, as stated in Lemma \ref{lem:FamilyB}. Hence, all isolating blocks can be glued together according to $L$. Thus, $L$ is realizable.

\end{proof}

We remark that  the conditions of Theorem \ref{teo:families} are sufficient to ensure the realization of a GS graph under those hypothesis although they are far from  necessary, as shown in the next example.

\begin{Ex}

In figure \ref{fig:Complementar}, we present a GS graph $L$ that does not satify Lemma \ref{lem:FirstFamily}, because of the presence of singularities of type $\mathcal{D}$ and attracting or repelling natures. Also, $L$ does not satisfy Lemma \ref{lem:FamilyB} because the bifurcation vertex, which is labelled with a singularity of type $\mathcal{W}$, has an odd weight in its positively incident edge. However, $L$ is realizable as a GS flow on a GS $2$-manifold. As we can see in Figure \ref{fig:Complementar}, the distinguished branched $1$-manifold on the boundaries with weight $5$ does not belong to Table \ref{tab:bordosAAA} nor Table \ref{tab:FirstFamily}.    

\begin{figure}[h] 
\centering
\begin{equation*}
  \begin{tikzcd}
   \includegraphics[align=c,scale=0.6]{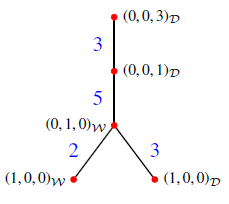} \xrightarrow{Realization} \hspace{0.1cm} \includegraphics[align=c,scale=0.3]{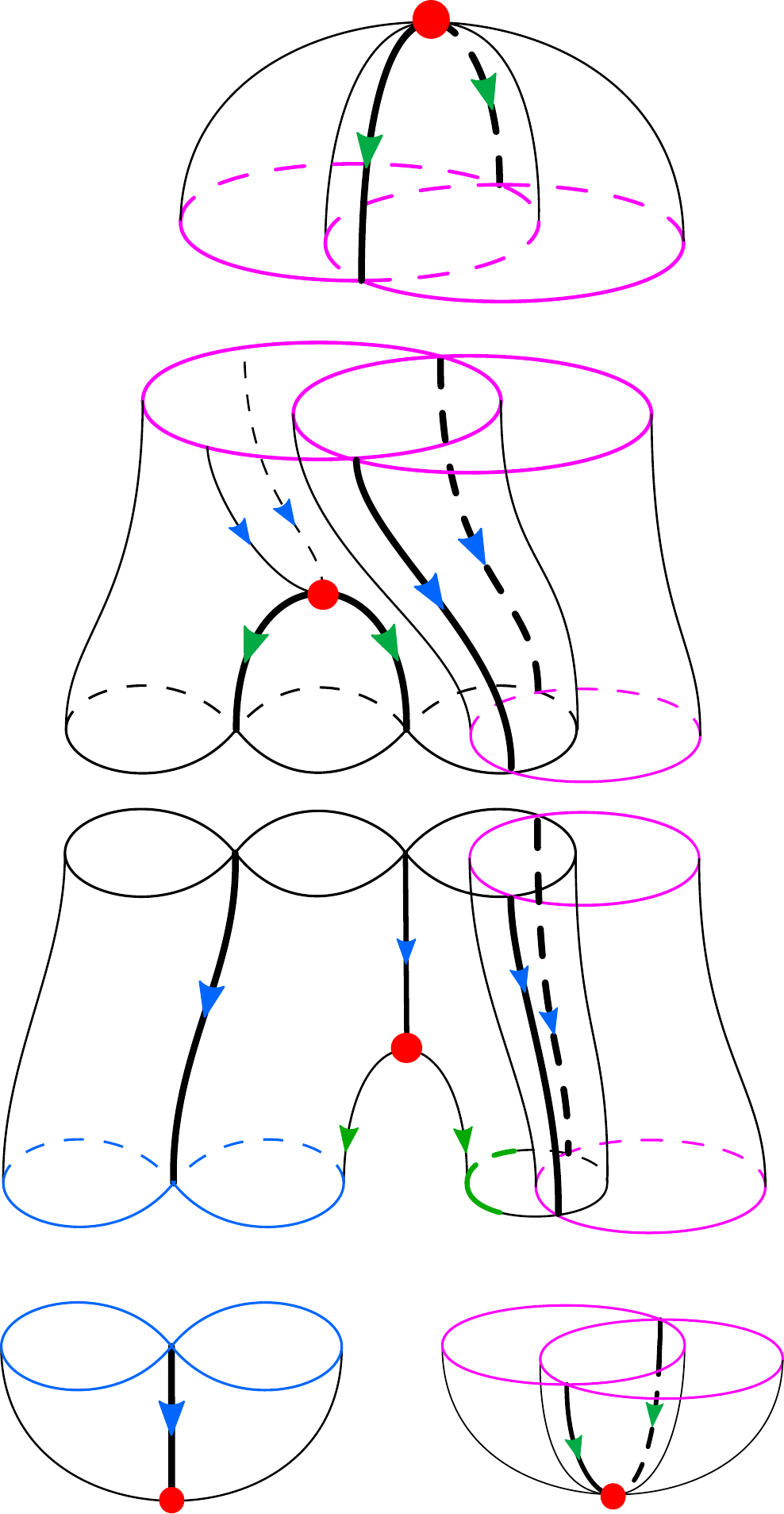} 
  \end{tikzcd}
\end{equation*}
\caption{Realization of a GS graph as a GS flow}
\label{fig:Complementar}
\end{figure}

\end{Ex}

Also, we remark that bifurcation vertices labelled with singularities of type $\mathcal{T}$ cannot be realized with the distinguished branched $1$-manifolds of  Table \ref{tab:bordosAAA} nor of Table \ref{tab:FirstFamily}.

The realization theorems presented in this section were obtained subject to  some type of simplification. For instance, restrictions on the type of  singularities with which the graph is labelled, or restrictions on the weights of the edges  or on the degree of the vertices. 
 Had this not been done, the increase of weights on the edges  would imply in a greater number of choices of distinguished branched $1$-manifolds and the local information on the graph is not sufficient to guarantee the gluing of the isolating blocks as indicated by the graph. However, even if we can not guarantee the existence of a triple $(\mathbf{M},X_{t},f)$, we still have control over the Euler characteristic of a possible realization on $M$. The reason for this is that the Euler characteristic can be computed in terms of the  natures and types of GS singularities, as will be shown in the next section.

\section{Euler Characteristic of  GS 2-manifolds}

The Euler characteristic can be defined for any  topological space  $\mathbf{T}$.  and it is  denoted by  $\mathcal{X}(\mathbf{T})$. It  is given by:
\begin{equation}\label{euler}
    \mathcal{X}(\mathbf{T}) = \sum_{i\in \mathbb{N}}{(-1)^{i}\beta_{i}},
\end{equation} 
where $\beta_{i}$  is the rank of  the $i$-th homology group $H_{i}(\mathbf{T})$, i.e., the  $i$-th Betti number of $\mathbf{T}$.

In \cite{hernan2019},  a filtration  of a closed GS $2$-manifold  $\mathbf{M}$  was constructed: $$G_{0} \subset G_{1} \subset \ldots \subset G_{m} = \mathbf{M},$$ such that $G_{k}$ contains exactly $k$ singularities and $(G_{i}, G_{i-1})$  is an index pair for the  $i$-th singularity of $\mathbf{M}$. Then, from the long exact sequence associated to the pair $(G_{i}, G_{i-1})$: $$\ldots \overset{p_{j}}{\longrightarrow} H_{j}(G_{i}, G_{i-1}) \overset{\partial_{j}}{\longrightarrow} H_{j-1}(G_{i-1}) \overset{i_{\ast}}{\longrightarrow} H_{j-1}(G_{i}) \overset{p_{j-1}}{\longrightarrow} H_{j-1}(G_{i},G_{i-1}) \overset{\partial_{j-1}}{\longrightarrow} \ldots$$ it was proven  that the Euler characteristic of $\mathbf{M}$ is equal to  the alternating sum  of the numerical Conley indices of the singularities, i.e.:
\begin{equation}\label{Conley}
    \mathcal{X}(\mathbf{M}) = \sum_{p_{i} \in Sing(\mathbf{M})}^{}{(h_{0}^{i} - h_{1}^{i} + h_{2}^{i})}.
\end{equation}

In what follows, we give an alternative  formula  for (\ref{Conley}) for closed GS  $2$-manifolds, by expressing  the Euler characteristic  in terms of  the total number $a$, $s$  and $r$ of  attracting, saddle and repelling natures, respectively,  of the GS singularities. In this sense, we remark that the total number of natures $a$, $s$ and $r$ of singularities of type $\mathcal{D}$ is equal to two. For instance, a GS singularity of type $\mathcal{D}$ and nature $ss_{s}$ or $ss_{u}$ has two saddle natures, and zero attracting and repelling natures. On the other hand, the total number of natures $a$, $s$ and $r$ of singularities of type $\mathcal{T}$ is equal to three. A GS singularity of type $\mathcal{T}$ and nature $ssa$  has two saddle natures, one attracting nature and zero repelling nature. Meanwhile, for an attractor singularity of type $\mathcal{T}$, we have that $a = 3$ and $s = r = 0$.

\begin{Prop}\label{prop:X}
Let $L$ be a GS graph. If  $\mathbf{M}$ is a realization of  $L$, then: $$\mathcal{X}(\mathbf{M}) = a - s + r + \frac{W}{2} + T,$$ where $W$ and $T$ are the number of vertices in $L$ labelled with singularities of type $\mathcal{W}$ and $\mathcal{T}$, respectively, while  $a$, $s$ and $r$ are, respectively, the total number of attracting, saddle and repelling natures of the singularities with which the vertices are labeled.
\end{Prop}

\begin{proof}
One has that 
\begin{equation}\label{eq:ConleyEuler} \mathcal{X}(\mathbf{M}) = \sum_{p_{i} \in Sing(\mathbf{M})}^{}{(h_{0}^{i} - h_{1}^{i} + h_{2}^{i})}.
\end{equation}

Denote by  $\mathcal{P}_{\eta}$ the total number of vertices in $L$ labeled with a GS singularities of type $\mathcal{P}$, where $\mathcal{P}\in \{\mathcal{R}, \mathcal{C}, \mathcal{W}, \mathcal{D}, \mathcal{T}\}$,  with nature $\eta$, that is, a singularity $p \in \mathcal{H}^{\mathcal{P}}_{\eta}$.

In what follows, we summarize the numerical Conley index $(h_{0}, h_{1}, h_{2})$ of $p \in \mathcal{H}^{\mathcal{P}}_{\eta}$, which was computed in \cite{hernan2019}:

\begin{table}[htb]
\centering
\begin{tabular}{c|cccccccccc}
\cellcolor{gray!20} & \cellcolor{gray!20} & \multicolumn{3}{c}{\cellcolor{gray!20} $\mathcal{R}$} & \cellcolor{gray!20} & \cellcolor{gray!20} & \multicolumn{3}{c}{\cellcolor{gray!20} $\mathcal{C}$} & \cellcolor{gray!20} \\
Nature & & $a$ & $s$ & $r$ & & & $a$ & $s$ & $r$ & \\
\hline
& & & & & & & & & & \\
$(h_{0}, h_{1}, h_{2})$ & & $(1,0,0)$ & $(0,1,0)$ & $(0,0,1)$ & & & $(1,0,0)$ & $(0,1,0)$ & $(0,1,2)$ & \\ 
& & & & & & & & & & \\
\end{tabular}
\end{table}

\begin{table}[htb]
\centering
\begin{tabular}{c|ccccccccccc}
\cellcolor{gray!20} & \cellcolor{gray!20} & \multicolumn{4}{c}{\cellcolor{gray!20} $\mathcal{W}$} & \cellcolor{gray!20} & \cellcolor{gray!20} & \multicolumn{4}{c}{\cellcolor{gray!20} $\mathcal{T}$} \\
Nature & & $a$ & $s_{s}$ & $s_{u}$ & $r$ & & & $a$ & $ssa$ & $ssr$ & $r$ \\
\hline
& & & & & & & & & & & \\
$(h_{0}, h_{1}, h_{2})$ & & $(1,0,0)$ & $(0,1,0)$ & $(0,0,0)$ & $(0,0,1)$ & & & $(1,0,0)$ & $(0,1,0)$ & $(0,1,2)$ & $(0,0,7)$ \\
& & & & & & & & & & & \\ 
\end{tabular}
\end{table}

\begin{table}[htb]
\centering
\begin{tabular}{c|ccccccc}
\cellcolor{gray!20} & \cellcolor{gray!20} & \multicolumn{6}{c}{\cellcolor{gray!20} $\mathcal{D}$} \\
Nature & & $a$ & $sa$ & $ss_{s}$ & $ss_{u}$ & $sr$ & $r$ \\
\hline
& & & & & & \\
$(h_{0}, h_{1}, h_{2})$ & & $(1,0,0)$ & $(0,1,0)$ & $(0,3,0)$ & $(0,1,0)$ & $(0,0,1)$ & $(0,0,3)$ \\ 
& & & & & & \\
\end{tabular}
\end{table}

Using  the numerical Conley index of a GS singularity of type $\mathcal{P}$  with nature $\eta$  and substituting this in (\ref{eq:ConleyEuler}), we have that:

\begin{eqnarray}\label{eq:Conley}
\mathcal{X}(\mathbf{M}) & = & (\mathcal{R}_{a} - \mathcal{R}_{s} + \mathcal{R}_{r})\ \ + \ \ (\mathcal{C}_{a} - \mathcal{C}_{s} + \mathcal{C}_{r}) \ \ + \ \ (\mathcal{W}_{a} - \mathcal{W}_{s_{s}} + 0\mathcal{W}_{s_{u}} + 2\mathcal{W}_{r}) \nonumber \\
 & & \\
& & \ \ + \ \ (\mathcal{D}_{a} - \mathcal{D}_{sa} - 3\mathcal{D}_{ss_{s}} - \mathcal{D}_{ss_{u}} + \mathcal{D}_{sr} + 3\mathcal{D}_{r})\ \ + \ \ (\mathcal{T}_{a} - \mathcal{T}_{ssa} + \mathcal{T}_{ssr} + 7\mathcal{T}_{r}) \nonumber 
\end{eqnarray}

Now, adding and subtracting terms in (\ref{eq:Conley}) corresponding to the number of attracting, saddle  and repelling natures within each GS singularity, one has:

\begin{eqnarray}
\mathcal{X}(\mathbf{M}) & = & (\mathcal{R}_{a} - \mathcal{R}_{s} + \mathcal{R}_{r})\ \ + \ \ (\mathcal{C}_{a} - \mathcal{C}_{s} + \mathcal{C}_{r}) \ \ + \ \ (\mathcal{W}_{a} - \mathcal{W}_{s_{s}} - \mathcal{W}_{s_{u}} +  \mathcal{W}_{r}) \nonumber \\
& & \nonumber \\
& & \ \ + \ \ (2\mathcal{D}_{a} - 0\mathcal{D}_{sa} - 2\mathcal{D}_{ss_{s}} - 2\mathcal{D}_{ss_{u}} + 0\mathcal{D}_{sr} + 2\mathcal{D}_{r})\ \ + \ \ (3\mathcal{T}_{a} - \mathcal{T}_{ssa} - \mathcal{T}_{ssr} + 3\mathcal{T}_{r}) \nonumber \\ 
& & \nonumber \\
& & \ \ + \ \ [\ \ (\mathcal{W}_{s_{u}} + \mathcal{W}_{r}) \ \ + \ \  (\mathcal{D}_{ss_{u}} + \mathcal{D}_{sr} + \mathcal{D}_{r}) \ \ - \ \ (\mathcal{D}_{ss_{s}} + \mathcal{D}_{sa} + \mathcal{D}_{a}) \ \ + \ \ (4\mathcal{T}_{r} - 2\mathcal{T}_{a} + 2\mathcal{T}_{ssr})\ \ ] \nonumber
\end{eqnarray}
 
Hence, one has:
\begin{eqnarray}\label{eq:formula}
\mathcal{X}(\mathbf{M}) \ \ = \ \ a - s + r \ \ + \ \ [\ \ (\mathcal{W}_{s_{u}} + \mathcal{W}_{r}) \ \ + \ \ (\mathcal{D}_{ss_{u}} + \mathcal{D}_{sr} + \mathcal{D}_{r}) \ \ - \ \ (\mathcal{D}_{ss_{s}} + \mathcal{D}_{sa} + \mathcal{D}_{a}) \nonumber \\
\nonumber \\
+ \ \ (4\mathcal{T}_{r} - 2\mathcal{T}_{a} + 2\mathcal{T}_{ssr})\ \ ]
\end{eqnarray} 

Since $\mathbf{M}$ is a closed manifold, any fold in $\mathbf{M}$ admits $\alpha$-limit and $\omega$-limit. Counting the total number of folds in $\mathbf{M}$ in terms of the GS singularities that are $\alpha$-limit of each fold and then $\omega$-limit of each fold, one obtains the following equalities:
\begin{eqnarray}\label{eq:folds}
\mbox{ number of folds in } \mathbf{M} & = & (\mathcal{W}_{s_{u}} + \mathcal{W}_{r}) \ \ + \ \ 2(\mathcal{D}_{ss_{u}} + \mathcal{D}_{sr} + \mathcal{D}_{r}) \ \ + \ \ (6\mathcal{T}_{r} + 4\mathcal{T}_{ssr} + 2\mathcal{T}_{ssa}) \ \ \ \ \ \ \ \ \ \ \nonumber \\
& & \nonumber \\
& = & (\mathcal{W}_{s_{s}} + \mathcal{W}_{a}) \ \ + \ \ 2(\mathcal{D}_{ss_{s}} + \mathcal{D}_{sa} + \mathcal{D}_{a}) \ \ + \ \ (6\mathcal{T}_{a} + 2\mathcal{T}_{ssr} + 4\mathcal{T}_{ssa})
\end{eqnarray}

It follows from the equality (\ref{eq:folds}) that:
\begin{eqnarray}\label{eq:termo}
(\mathcal{W}_{s_{u}} + \mathcal{W}_{r}) \ \ + \ \ (\mathcal{D}_{ss_{u}} + \mathcal{D}_{sr} + \mathcal{D}_{r}) \ \ - \ \ (\mathcal{D}_{ss_{s}} + \mathcal{D}_{sa} + \mathcal{D}_{a}) \ \ + \ \ (4\mathcal{T}_{r} - 2\mathcal{T}_{a} + 2\mathcal{T}_{ssr}) \nonumber \\
\nonumber \\
= \ \ (\mathcal{W}_{s_{s}} + \mathcal{W}_{a}) \ \ + \ \ (\mathcal{D}_{ss_{s}} + \mathcal{D}_{sa} + \mathcal{D}_{a}) \ \ - \ \ (\mathcal{D}_{ss_{u}} + \mathcal{D}_{sr} + \mathcal{D}_{r}) \ \ + \ \ (4\mathcal{T}_{a} + 2\mathcal{T}_{ssa} - 2\mathcal{T}_{r})
\end{eqnarray}

By adding the two sides of equality (\ref{eq:termo}), one has:
\begin{eqnarray}\label{eq:final}
2[(\mathcal{W}_{s_{u}} + \mathcal{W}_{r}) \ \ + \ \ (\mathcal{D}_{ss_{u}} + \mathcal{D}_{sr} + \mathcal{D}_{r}) \ \ - \ \ (\mathcal{D}_{ss_{s}} + \mathcal{D}_{sa} + \mathcal{D}_{a}) \ \ + \ \ (4\mathcal{T}_{r} - 2\mathcal{T}_{a} + 2\mathcal{T}_{ssr})]\nonumber \\
\nonumber \\
= \ \  (\mathcal{W}_{s_{u}} + \mathcal{W}_{r} + \mathcal{W}_{s_{s}} + \mathcal{W}_{a}) \ \ + \ \ (2\mathcal{T}_{r} + 2\mathcal{T}_{a} + 2\mathcal{T}_{ssr} + 2\mathcal{T}_{ssa}) \nonumber \\
\nonumber \\
= \ \ W + 2T
\end{eqnarray}

And finally, by substituting (\ref{eq:final}) in (\ref{eq:formula}), the proof is concluded.

\end{proof}

We remark that a similar formula for the Euler characteristic of singular surfaces with cross-caps and triple crossing singularities was proved in \cite{izumiya1995topologically} in a different setting.

In what follows, we use the formula given in Proposition \ref{prop:X} to compute the Euler characteristic of the GS $2$-manifolds presented in the examples throughout section \ref{sec:realization}.

\begin{Ex}
Let $\mathbf{M}$ be the GS $2$-manifold of Example \ref{ex:minRCWD}. One has that, $\mathcal{X}(\mathbf{M}) = 3 - 5 + 2 + \frac{2}{2} + 0 = 1$.

In the case $\mathbf{M}$ is the GS $2$-manifold of  Example \ref{ex:RCWDlinear}, one has that $\mathcal{X}(\mathbf{M}) = 2 - 3 + 1 + \frac{2}{2} + 0 = 1$.

On the other hand, if  $\mathbf{M}$ is the GS $2$-manifold of Example \ref{ex:RCW}, then $\mathcal{X}(\mathbf{M}) = 2 - 6 + 2 + \frac{4}{2} + 0 = 0$. 

Lastly, if $\mathbf{M}$ is the GS $2$-manifold of Example \ref{fig:Complementar}, it follows that $\mathcal{X}(\mathbf{M}) = 3 - 2 + 3 + \frac{2}{2} + 0 = 5$.

%Seja $\mathbf{M}$ a $2$-variedade GS do Exemplo \ref{ex:minTar}. Então, $\mathcal{X}(\mathbf{M}) = 3 - 0 + 3 + 0 + 2 = 8$.

\end{Ex}

\section{Concluding remarks}   

{In this work we determine  sufficient conditions for the realization of abstract Lyapunov graphs labelled with singularities of type regular ($\mathcal{R}$), cone ($\mathcal{C}$), Whitney ($\mathcal{W}$), double crossing ($\mathcal{D}$) and triple crossing ($\mathcal{T}$) as Gutierrez-Sotomayor flows on closed singular two-manifolds. We show here that for a graph which strictly satisfies the set of necessary conditions of Theorem \ref{teoPH}, its realization is not always possible. At times, not even locally as we show in Theorem \ref{teo:combinatorial}.

{Locally, that is, for a semi-graph $L_{v}$ consisting of a single vertex $v$ and its incident edges, sufficient conditions for the realization of $L_{v}$ as a GS flow on a GS isolating block are presented. Moreover, in the case the weights on the incident edges are the lowest satisfying the Poincaré-Hopf condition, we provide in Theorem  \ref{teo:colecao} a complete characterization of the realizations of $L_{v}$ by describing the distinguished branched $1$-manifolds that make up the boundary components of all minimal GS isolating blocks for each singularity type. Furthermore, for higher weights on the incident edges satisfying the Poincaré-Hopf condition, we show in Theorem \ref{teo:passageway_cone}  that  GS isolating blocks with passageways which realize $L_{v}$ arise from a process of identification of pairs of orbits on minimal GS isolating blocks.}

{Concerning the global realization question, it can be posed as a problem of assigning distinguished branched $1$-manifolds to the edges of a graph, so that the induced assignment on the incident edges of each vertex is obtainable as boundary components of an isolating block for the singularity with which the vertex is labelled. Since the set of isolating blocks for a singularity is given by all the possible ways of adding passageways to a minimal isolating block, this question is equivalent to determining reachability from minimal boundaries to distinguished branched $1$-manifolds with higher weights. }

{The sufficient conditions for global realization presented in Theorem \ref{teo:families}  are obtained from two families of distinguished branched $1$-manifolds for which the increase of branch points is given in a controlled fashion. Moreover,  each of these two families assigns the same distinguished branched $1$-manifold to every edge of the graph with a common weight, thus ensuring that boundary components with common weights of two isolating blocks are always homeomorphic. Thus, all boundaries of the isolating blocks are glued according to the graph guaranteeing the global realization.  
By no means are these families unique since the increase of branch points could be attained by many more identifications than the ones considered in this work.}

{One may hope to obtain stronger global results by considering families of distinguished branched $1$-manifolds with more than one choice of assignment for each weight. Nevertheless, we remark that the number of non-homeomorphic distinguished branched $1$-manifolds increases quickly as one goes subsequently from one weight to the next. For instance, for weights $1$, $2$, $3$, $4$ and $5$ one has $1$, $1$, $2$, $4$ and at least $11$ non-homeomorphic distinguished branched $1$-manifolds, respectively. Verifying whether two distinguished branched $1$-manifolds with arbitrary weight are homeomorphic or not is a difficult question pertaining to the theory of $4$-regular pseudo-graphs. Thus, determining sufficient conditions for the global realization of Lyapunov graphs 
 is more challenging in this setting.}

{Mostly, the global results presented in this work include singularities of types $\mathcal{R}, \mathcal{C}, \mathcal{W}$ and $\mathcal{D}$, leaving out singularities of type $\mathcal{T}$  due to their intrinsic  complexity. With the exception of the realization of graphs with minimal weights on all edges proved in Theorem \ref{teo:minimal}, determining  sufficient conditions for the global realization of graphs including   $\mathcal{R}$, $\mathcal{C}, \mathcal{W}$, $\mathcal{D}$ and  $\mathcal{T}$   singularities remains an open question.} 

%{Also, one would like to consider the inclusion of periodic orbits and singular cycles in a study similar to this one.}

\nocite{aluffi2009chern}
\nocite{brasselet1981classes}
\nocite{brasselet2009vector}
\nocite{brasselet2010hirzebruch}
\nocite{esole2019euler}
\nocite{macpherson1974chern}
\nocite{schwartz1965classes}

\bibliography{references}
\bibliographystyle{ieeetr}

\end{document}